\documentclass[12pt,reqno]{amsart}

\usepackage[colorlinks=true,urlcolor=blue,
bookmarks=true,bookmarksopen=true,citecolor=blue]{hyperref}

\usepackage{mathtools} 
\usepackage{graphicx}
\usepackage[usenames,dvipsnames]{color}
\usepackage[all]{xy}
\usepackage{amsmath}
\usepackage{amsfonts,amssymb}
\usepackage{amsthm}
\usepackage{bbm}
\usepackage{mathrsfs}
\usepackage{marginnote}
\usepackage{pdfsync}
\usepackage{xfrac}
\usepackage{varioref}
\usepackage{a4wide}
\makeatletter
\DeclareFontFamily{OMX}{MnSymbolE}{}
\DeclareSymbolFont{MnLargeSymbols}{OMX}{MnSymbolE}{m}{n}
\SetSymbolFont{MnLargeSymbols}{bold}{OMX}{MnSymbolE}{b}{n}
\DeclareFontShape{OMX}{MnSymbolE}{m}{n}{
    <-6>  MnSymbolE5
   <6-7>  MnSymbolE6
   <7-8>  MnSymbolE7
   <8-9>  MnSymbolE8
   <9-10> MnSymbolE9
  <10-12> MnSymbolE10
  <12->   MnSymbolE12
}{}
\DeclareFontShape{OMX}{MnSymbolE}{b}{n}{
    <-6>  MnSymbolE-Bold5
   <6-7>  MnSymbolE-Bold6
   <7-8>  MnSymbolE-Bold7
   <8-9>  MnSymbolE-Bold8
   <9-10> MnSymbolE-Bold9
  <10-12> MnSymbolE-Bold10
  <12->   MnSymbolE-Bold12
}{}

\let\llangle\@undefined
\let\rrangle\@undefined
\DeclareMathDelimiter{\llangle}{\mathopen}%
                     {MnLargeSymbols}{'164}{MnLargeSymbols}{'164}
\DeclareMathDelimiter{\rrangle}{\mathclose}%
                     {MnLargeSymbols}{'171}{MnLargeSymbols}{'171}
\makeatother
\renewenvironment{proof}[1][\proofname]{\noindent{\bfseries #1.  }}{\qed}

\newcommand{\N}{\mathbbm{N}}                     % the natural numbers
\newcommand{\Z}{\mathbbm{Z}}                     % the integer numbers
\newcommand{\R}{\mathbbm{R}}                     % the real line
\newcommand{\C}{\mathbbm{C}}                     % the complex plane
\newcommand{\T}{\mathbbm{T}}                     % the torus
\newcommand{\A}{\mathbbm{A}}                    % action
\newcommand{\E}{\mathbbm{E}}                    % energy
\newcommand{\then}{\Rightarrow}                 % implication arrow
\newcommand{\gr}{\mathrm{gr}}            % graph
\newcommand{\spec}{\mathrm{Spec}}               % Spectrum
\newcommand{\crit}{\mathrm{Crit\,}}               % Critical set
\newcommand{\fix}{\mathrm{Fix}}               % Fix point set
\newcommand{\RFH}{\mathrm{RFH}}    %RFH
\newcommand{\HF}{\mathrm{HF}}    %HF
\newcommand{\loc}{\mathrm{loc}} %loc
\newcommand{\HM}{\mathrm{HM}} % HM

\newtheorem{thm}{Theorem}[section]               % numbered absolutely
\newtheorem*{thm*}{Theorem}               % no number
\newtheorem{cor}[thm]{Corollary}        % numbered along with Theorem
\newtheorem*{cor*}{Corollary}        % no number
\newtheorem{lem}[thm]{Lemma}            % numbered along with Theorem
\newtheorem{prop}[thm]{Proposition}     % numbered along with Theorem
\newtheorem{conj}[thm]{Conjecture}      % numbered along with Theorem

\theoremstyle{definition}
\newtheorem{defn}[thm]{Definition}      % numbered along with Theorem
\newtheorem{rem}[thm]{Remark}           % numbered along with Theorem  
\newtheorem{ex}[thm]{Example}           % numbered along with Theorem
 \newtheorem*{acknowledgement*}{\protect\acknowledgementname}
 \providecommand{\acknowledgementname}{Acknowledgement}

\title{On the existence of infinitely many invariant Reeb orbits}

\author{Will J. Merry and Kathrin Naef} 
\date{} %Removed date so as not to contradict arxiv date

\address{Will J.~Merry\\
Department of Mathematics\\
ETH Z\"urich}
\email{\texttt{merry@math.ethz.ch}}

\address{Kathrin Naef\\
Department of Mathematics\\
ETH Z\"urich}
\email{\texttt{kathrin.naef@math.ethz.ch}}

\begin{document}

\begin{abstract}
In this article we extend results of Grove and Tanaka \cite{GroveTanaka1976, GroveTanaka1978,Tanaka1982} on the existence of isometry-invariant geodesics to the setting of Reeb flows and strict contactomorphisms. Specifically, we prove that if $M$ is a closed connected manifold with the property that the Betti numbers of the free loop space $ \Lambda (M)$ are asymptotically  unbounded then for every fibrewise star-shaped hypersurface $\Sigma \subset T^*M$ and every strict contactomorphism $ \varphi \colon \Sigma \to \Sigma$ which is contact-isotopic to the identity, there are infinitely many invariant Reeb orbits. 
\tableofcontents
\end{abstract}
\maketitle

\section{Introduction}
\label{sec:introduction}

The problem of the existence of closed geodesics is one of the oldest and richest fields of study in Riemannian geometry. In 1951 Lyusternik and Fet \cite{LyusternikFet1951}  proved that every closed Riemannian manifold $(Q,g)$ has at least one closed geodesic. In 1969 Gromoll and Meyer \cite{GromollMeyer1969} proved the following remarkable extension: if $Q$ is a closed simply connected manifold with the property that the Betti numbers of the free loop space $ \Lambda (Q)$ are asymptotically unbounded, then every Riemannian metric $g$ on $Q$ has \textbf{infinitely} many embedded closed geodesics.  \\

Suppose now that one is given an isometry $ f $ of a Riemannian manifold $(Q,g)$. A related problem is the existence of \textbf{$f$-invariant geodesics}, that is, geodesics $\gamma \colon \R \to Q$ such that $f( \gamma(s)) = \gamma(s+ \tau)$ for some non-zero $\tau \in \R $ and all $s \in \R$. With this terminology, a closed geodesic is precisely an $ \mathrm{Id}$-invariant geodesic. The problem of the existence of invariant geodesics was first studied by Grove \cite{Grove1973,Grove1973a}. The analogue of the Gromoll-Meyer theorem was proved by Grove and Tanaka \cite{GroveTanaka1976, GroveTanaka1978,Tanaka1982}: if $Q$ is a closed simply connected manifold with the property that the Betti numbers of the free loop space $ \Lambda (Q)$ are asymptotically  unbounded, then for every Riemannian metric $g$ on $Q$ and every isometry $f$ of $(Q,g)$ which is homotopic to the identity, there are infinitely many invariant geodesics.  \\

The Gromoll-Meyer theorem alluded to above can be seen as a special case of a more general result on the existence of closed Reeb orbits on fibrewise star-shaped hypersurfaces. Denote by $ \lambda \in \Omega^1(T^*Q)$ the canonical Liouville form. If $ \Sigma \subset T^*Q$ is a fibrewise star-shaped hypersurface, then the restriction of $\lambda$ to $ \Sigma $ is a contact form. In this setting the corresponding problem concerns the existence of closed Reeb orbits. Using Floer-theoretical methods, McLean \cite{McLean2012}, and independently Hryniewicz and Macarini \cite{HryniewiczMacarini2012}, proved: if $Q$ is a closed  manifold with the property that the Betti numbers of the free loop space $ \Lambda (Q)$ are asymptotically unbounded, then every fibrewise star-shaped hypersurface has infinitely many embedded closed Reeb orbits. \\

The natural generalisation to the contact setting for invariant geodesics was first  proposed by Mazzucchelli \cite{Mazzucchelli2014a,Mazzucchelli2014}. Suppose $ \Sigma \subset T^*Q$ is a fibrewise star-shaped hypersurface. Denote by $ \alpha := \lambda|_{\Sigma}$ the induced contact form. A contactomorphism $ \varphi \colon \Sigma \to \Sigma $ is called \textbf{strict} if $ \varphi^* \alpha = \alpha$. A Reeb orbit $ x \colon \R \to \Sigma$ is $\varphi$-invariant if $ \varphi ( x(s)) = x(s +\tau)$ for some non-zero $\tau \in \R$ and all $ s \in \R$ (this notion only makes sense for strict contactomorphisms). As with the geodesic case, with this terminology a closed Reeb orbit is precisely an $\mathrm{Id}$-invariant Reeb orbit. One can then ask whether there are infinitely many invariant Reeb orbits.  In this paper we will prove the following generalisation of the Grove-Tanaka theorem.

\begin{thm}
\label{thm:grove_tanaka_reeb}
Suppose $Q$ is a closed connected  manifold with the property that the Betti numbers of the free loop space $ \Lambda (Q)$ are asymptotically  unbounded. Then for every fibrewise star-shaped hypersurface $\Sigma \subset T^*Q$ and every strict contactomorphism $ \varphi \colon \Sigma \to \Sigma$ which is contact-isotopic to the identity, there are infinitely many invariant Reeb orbits. 
\end{thm}

\begin{rem}
If $f : Q \to Q$ is an isometry with respect to $g$, then $f$ lifts to define a strict contactomorphism $\varphi_f$ of the unit cotangent bundle $S_g^*Q$ via the formula
\[
  \varphi_f(q,p) = (f(q), p \circ Df(q)^{-1}).
\]
Thus Theorem \ref{thm:grove_tanaka_reeb} can be seen as a generalisation of the original Grove-Tanaka result. Unfortunately it is not strictly speaking a true generalisation, because the Grove-Tanaka theorem requires only that the  isometry $f$ is homotopic to the identity, whereas in contrast our result requires the lifted contactomorphism $\varphi_f$ to be contact isotopic to the identity (which is the case if $f$ is isotopic to the identity). Aside from this point however, note that Theorem \ref{thm:grove_tanaka_reeb} includes the case of (possibly asymmetric) Finsler metrics:
\end{rem}
\begin{cor}
\label{cor:finsler}
If $Q$ is a closed manifold with the property that the Betti numbers of the free loop space $ \Lambda (Q)$ are asymptotically  unbounded, then for every (possibly asymmetric) Finsler metric $F$ on $Q$ and every isometry $f$ of $(Q,F)$ which is isotopic to the identity, there are infinitely many $f$-invariant Finsler geodesics. 
\end{cor}
\begin{rem}
Corollary \ref{cor:finsler} has been proved independently by Lu in \cite{Lu2014} using different methods.
\end{rem}
A further application of Theorem \ref{thm:actual_theorem} is given in Theorem \ref{thm:magnetic} below.\\

In fact, similarly to how McLean \cite{McLean2012} and Hryniewicz and Macarini \cite{HryniewiczMacarini2012} proved their extension of the Gromoll-Meyer theorem, we will deduce Theorem \ref{thm:grove_tanaka_reeb} from the following more general result. A contact manifold $ (\Sigma, \alpha)$ is \textbf{Liouville-fillable} if $\Sigma $ is the boundary of a Liouville domain $(M_1,  \lambda_1)$, and $\alpha =\lambda_1|_{\Sigma}$.  Given a Liouville domain $(M_1,\lambda_1)$, Cieliebak and Frauenfelder \cite{CieliebakFrauenfelder2009} have associated an invariant $\RFH_*(M_1,\lambda_1)$ called the \textbf{Rabinowitz Floer homology}. In \cite{Wiegel2013}, Weigel introduced the notion of the \textbf{positive growth rate} $\Gamma_+( M_1, \lambda_1) \in \{ -\infty \} \cup [0, + \infty]$ of a Liouville domain $(M_1, \lambda_1)$, which roughly speaking measures the growth of the filtered positive Rabinowitz Floer homology. A finite growth rate indicates polynomial growth, while an infinite growth rate implies super-polynomial (for instance, exponential) growth.

\begin{thm}
\label{thm:actual_theorem}
Suppose that $(\Sigma ,\alpha)$ is a Liouville-fillable contact manifold which admits a filling $(M_1,  \lambda_1)$ with $ \Gamma_+( M_1, \lambda_1)  > 1$. Then every strict contactomorphism $ \varphi \colon \Sigma \to \Sigma$ which is contact-isotopic to the identity has infinitely many invariant Reeb orbits. 
\end{thm} 

\begin{rem}
We emphasise that in Theorem \ref{thm:actual_theorem}, we do \textbf{not} require $\varphi$ to be isotopic to the identity through strict contactomorphisms.
\end{rem}

Here is another setting where our results are applicable. Suppose $Q$ is a closed manifold and $\Omega$ is an \textbf{closed} 2-form on $Q$. One should think of $\Omega$ as representing a \textbf{magnetic field}. We use $\Omega$ to build a  \textbf{twisted symplectic form} $\omega = d \lambda + \pi^* \Omega$, on $T^*Q$, where as before $\lambda$ is the canonical Liouville 1-form. Suppose  $H:T^*Q \to \R$ is a \textbf{Tonelli} Hamiltonian: this means that $H$ is smooth function on $T^*Q$ which is $C^2$-strictly convex and superlinear on the fibres of $T^*Q$. We are interested in studying the flow of $\phi_H^t : T^*Q \to T^*Q$ of the symplectic gradient $X_H$ of $H$, taken with respect to the twisted symplectic form $\omega$. For instance, if $H(q,p) = \frac{1}{2} | p |^2 +U(q)$ is a mechanical Hamiltonian of the form kinetic plus potential energy, then $\phi_H^t$ can be thought of as modelling the motion of a charged particle in a magnetic field. We refer the reader to \cite{Ginzburg1996} for an in-depth treatment of magnetic flows in symplectic geometry. 

Given $e > 0$, let $\Sigma_e : = H^{-1}(e) \subset T^*Q$. Since $H$ is autonomous, the flow $\phi_H^t : T^*Q \to T^*Q$ of the symplectic gradient $X_H$ preserves the energy level $\Sigma_e$. A \textbf{magnetic geodesic} $\gamma : \R \to Q$ of energy $e$ is the projection to $Q$ of an orbit of $\phi_H^t|_{\Sigma_e}$.

Let us denote by $\mathcal{G}(H,\Omega)$ the group of \textbf{symmetries} of the system:
\[
  \mathcal{G}(H,\Omega) := \left\{  f \in \mathrm{Diff}(Q) \mid f^*\Omega = \Omega, \text{ and } H(f(q),p) = H(q, p \circ Df(q)), \ \forall \,(q,p) \in T^*Q \right\}.
  \]
Let $\mathcal{G}_0(H,\Omega)$ denote the connected component of $\mathcal{G}(H,\Omega)$ containing $\mathrm{Id}$. For instance, if $H(q,p) = \frac{1}{2}| p |^2 + U(q)$ is a mechanical Hamiltonian, then elements of $\mathcal{G}(H,\Omega)$ are simply the isometries of $(Q,g)$ that preserve the 2-form $\Omega$ and the potential $U$.

Assume now that $\Omega$ is \textbf{exact}. We define the \textbf{strict Ma\~n\'e critical value} $c_0= c_0(H, \Omega)$ by
\begin{equation}
\label{eq:mcv}
  c_0 := \inf_{ \theta} \sup_{q \in Q} H(q, -\theta_q),
\end{equation}
where the infimum\footnote{The fact that one takes $-\theta$ in the definition of $c_0$ is due to our sign conventions.} is over the set of all primitives $\theta$ of $\Omega$.  If $e > c_0$ then $\Sigma_e \subset T^*Q$ is a hypersurface of restricted contact type in the symplectic manifold $(T^*Q, \omega)$ (see Lemma \ref{lem:rct} below). As with the case of isometries earlier, a diffeomorphism $f \in \mathcal{G}(H,\Omega)$ lifts to define a symplectomorphism
 \[
    \phi_f :T^*Q \to T^*Q , \qquad  \phi_f(q, p) = (f(q), p \circ Df(q)^{-1}),
  \] 
which preserves the hypersurface $\Sigma_e$ and whose restriction $\phi_f|_{\Sigma_e}$ lies in $\mathrm{Cont}(\Sigma_e, \ker \,\omega|_{\Sigma_e})$. 

\begin{thm}
\label{thm:magnetic}
 Suppose $Q$ is a closed connected  manifold with the property that the Betti numbers of the free loop space $ \Lambda (Q)$ are asymptotically  unbounded. Suppose $e > c_0(H, \Omega)$. Then given any  symmetry $f \in \mathcal{G}_0(H, \Omega)$, there exist infinitely many invariant magnetic geodesics with energy $e$.
 \end{thm} 

The proof is given in Section \ref{sec:exact_magnetic_flows} below.

\begin{rem}
Instead of assuming that $\Omega$ is exact, one can instead make the weaker assumption that $\Omega$ is \textbf{weakly exact}. This means that the lift $\widetilde{\Omega}$ of $\Omega$ to the universal cover $\widetilde{Q}$ of $Q$ is exact. In this case one can define the \textbf{universal Ma\~n\'e critical value} $c_u = c_u(H, \Omega)$ by first lifting $H$ to a Hamiltonian $\widetilde{H} : T^* \widetilde{Q} \to \R $ and then defining $c_u$ in exactly the same way as in \eqref{eq:mcv}, but for $\widetilde{H}$ and primitives of $\widetilde{\Omega}$ instead. If $\Omega$ is exact then one has $c_u \le c_0$, and in general the inequality can be strict \cite{PaternainPaternain1997a,CieliebakFrauenfelderPaternain2010}. The main result of \cite{Merry2011a, BaeFrauenfelder2010} asserts that for $ e > c_u$, one can still define the Rabinowitz Floer homology $\RFH_*(\Sigma_e, T^*Q)$, and that in fact it holds that  $\RFH_*(\Sigma_e, T^*Q) \cong \RFH_*(S_g^*Q ,T^*Q)$. These results imply that it is possible to extend Theorem \ref{thm:magnetic} to cover this case. The precise statement and full proof can be found in \cite{Naef_thesis}.
\end{rem}

The existence of invariant Reeb orbits can be seen as a special case of the leaf-wise intersection problem. Suppose as above that $( \Sigma, \alpha)$ is a Liouville-fillable contact manifold with filling $(M_1,  \lambda_1)$. Denote by $M$ the non-compact symplectic manifold obtained by gluing $ \Sigma \times [1, +\infty)$ onto $M_1$, and let $ \phi \colon M \to M$ denote a compactly supported Hamiltonian diffeomorphism. A point $x \in \Sigma$ is a \textbf{leaf-wise intersection point} for $\phi$ if $\phi(x)$ belongs to the same Reeb orbit as $x$ does. The leaf-wise intersection problem was introduced by Moser \cite{Moser1978}, and in a series of papers Albers and Frauenfelder \cite{AlbersFrauenfelder2010c,AlbersFrauenfelder2012b,AlbersFrauenfelder2010} showed how Rabinowitz Floer homology can detect leaf-wise intersection points. If $\varphi$ is a contactomorphism of $ \Sigma $ then one can lift $\varphi$ to a  compactly supported Hamiltonian diffeomorphism $\phi $ of $M$. In this setting leaf-wise intersection points of $\phi$ are also called \textbf{translated points} of $\varphi$ by Sandon \cite{Sandon2012}. When $\varphi$ is a strict contactomorphism, a Reeb orbit is invariant if and only if some (and therefore all) of the points on the Reeb orbit are translated points of $\varphi$.  \\

In \cite{AlbersFrauenfelder2012a}, Albers and Frauenfelder asked whether the analogue of the Gromoll-Meyer theorem holds for leaf-wise intersections. The natural conjecture is:

\begin{conj}
\label{conj:gromoll_meyer_leafwise}
Suppose that $(\Sigma ,\alpha)$ is a Liouville-fillable contact manifold which admits a Liouville filling $(M_1,  \lambda_1)$ with $ \Gamma_+( M_1, \lambda_1)  > 1$. Then every compactly supported Hamiltonian diffeomorphism of $(M, d \lambda)$ has leaf-wise intersection points on infinitely many different Reeb orbits.  
\end{conj}

Theorem \ref{thm:actual_theorem} is thus the special case of Conjecture \ref{conj:gromoll_meyer_leafwise} when the Hamiltonian diffeomorphism is the lift of a strict contactomorphism. Unfortunately we were unable to prove Conjecture \ref{conj:gromoll_meyer_leafwise}; see Remark \ref{rem:removing_commuting_would_be_real_nice} below for an explanation of where our proof breaks down in the general case. 

\begin{acknowledgement*}
We are very grateful to Marco Mazzucchelli for explaining to us why studying strict contactomorphisms is interesting in this setting, and for numerous helpful discussions on generating functions, and to Viktor Ginzburg for his many detailed and useful comments, and in particular for suggesting Theorem \ref{thm:magnetic} to us. The first author thanks Alberto Abbondandolo for many discussions about the $L^{\infty}$-estimates in Section \ref{sec:proof_of_thm_estimate}.  
\end{acknowledgement*} 

\section{Preliminaries}
\label{sec:preliminaries}

Let $( \Sigma, \alpha)$ denote a Liouville-fillable contact manifold. By definition this means there exists a \textbf{Liouville domain} $(M_1, \lambda_1)$ such that  $(M_1,d \lambda_1)$ is a compact symplectic manifold, $\Sigma = \partial M_1$, and $\alpha = \lambda_1 |_{\Sigma}$. The vector field $Y_1$ on $M_1$ defined by $ \imath_{Y_1} (d \lambda_1) = \lambda_1$ is transverse to $\Sigma$ and points outwards. Since $M_1$ is compact, the flow $ \phi_{Y_1}^s  \colon M_1 \to M_1$ is defined for all $ s \le 0$, and thus induces an embedding 
\begin{equation}
\label{eq:the_embedding_I}
I : \Sigma \times (0,1] \hookrightarrow M_1, \qquad I( x, r):= \phi_{Y_1}^{ \log r}(x).
\end{equation}
Note that $I^* \lambda_1 = r \alpha$, and $I_*( r \partial_r) = Y_1$. We denote by $M$ the completion of $M_1$, defined by 
\[
 M := M_1 \cup_{ \Sigma} ( \Sigma \times [1, +\infty)).
 \] 
 We extend $\lambda_1$ and $Y_1$ to a one-form $\lambda$ and a vector field $Y$ respectively on all of $M$ by setting $\lambda|_{ M_1} : = \lambda_1$, $Y|_{M_1} := Y_1$ and
  \[
 \lambda|_{ \Sigma \times [1, + \infty)} := r \alpha, \qquad Y|_{ \Sigma \times [1, + \infty)} : = r \partial_r.
 \]
Then $(M, d \lambda)$ is an exact symplectic manifold containing $ \Sigma $ as separating hypersurface.  Moreover the embedding $I$ from \eqref{eq:the_embedding_I} extends to define an embedding 
\begin{equation}
\label{eq:the_extended_embedding_I}
I \colon(S \Sigma , d(r \alpha)) \hookrightarrow ( M, d \lambda),
\end{equation}
where $S \Sigma := \Sigma \times (0 , + \infty)$ is the \textbf{symplectisation} of $\Sigma$.  We will always identify $\Sigma $ with $\Sigma \times \{ 1 \} \subset S \Sigma \subset M$.

The \textbf{extended phase space} is the symplectic manifold $(\widetilde{M}, \omega)$ where 
\[
\widetilde{M} : = M \times T^* \R, 
\]
and $\omega $ is the symplectic form
\begin{equation}
\label{eq:sym_form_on_W}
\omega : = d \lambda - d\sigma \wedge d \tau,
\end{equation}
where $(\tau ,\sigma) \in \R \times \R^* \cong T^*\R$.\\

We denote by $R $ the Reeb vector field of $ \alpha$, and $ \varphi_R^s : \Sigma \to \Sigma $ the Reeb flow.  The following definition was introduced by Sandon \cite{Sandon2012}. 
\begin{defn}
Suppose $ \varphi \colon \Sigma \to \Sigma$ is a contactomorphism. Thus there exists a smooth positive function $\rho \colon \Sigma \to (0 ,+ \infty) $ such that $ \varphi^*\alpha = \rho \alpha$. A point $x \in \Sigma$ is a \textbf{translated point} of $\varphi$ if there exists $ \tau \in \R $ such that
\begin{equation}
\label{eq:translated_point}
 \varphi( \varphi_R^{\tau}(x)) = x, \qquad \rho(x) = 1.
\end{equation}
We denote by 
\begin{equation}
\label{eq:spec_varphi}
\spec(\varphi) := \left\{ \tau \in \R \mid \text{ there exists } x \in \Sigma \text{ such that } \eqref{eq:translated_point} \text{ holds,} \right\}
\end{equation}
\end{defn}

\begin{ex}
\label{ex:strict}
Suppose $\varphi$ is a \textbf{strict} contactomorphism, i.e. $\varphi^* \alpha = \alpha$. In this case a translated point is simply a point $x  \in \Sigma$ such that $\varphi( \varphi_R^ \tau(x)) =x $ for some $\tau \in \R$. But since strict contactomorphisms commute with the Reeb flow (as $\varphi_*(R) = R$), we see that if $x$ is a translated point of $\varphi$ then every point on the Reeb orbit $ \{ \varphi_R^s(x) \mid s \in \R\}$ is also a translated point: 
\[
\varphi( \varphi_R^{\tau+s}(x)) = \varphi_R^s(x).
\]
Thus the Reeb orbit $\{ \varphi_R^s(x) \mid s \in \R\}$ is \textbf{$\varphi$-invariant}. Theorem \ref{thm:actual_theorem}, the main result of this paper, gives conditions under which every strict contactomorphism admits infinitely many distinct invariant Reeb orbits. 
\end{ex}

We will now show how to associate to each contactomorphism $\varphi \colon \Sigma \to \Sigma$ which is contact-isotopic to the identity a Hamiltonian diffeomorphism $\Phi $ of $\widetilde{M}$ with the property that its fixed points can be identified with the translated points of $ \varphi$. We first will need to introduce a number of auxiliary functions. \\

In general given a contactomorphism $\varphi$ which is contact-isotopic to the identity, we use the notation $\widehat{\varphi} $ to indicate a smootly parametrized path $\{ \varphi_t \}_{t \in [0,1]}$ such that $\varphi_0 = \mathrm{Id}$ and $\varphi_1 = \varphi$.
\begin{defn}
A path $\widehat{\varphi} = \{ \varphi_t \}_{0 \le t \le 1}$ is called \textbf{admissible} if it is stationary on time $[0,1/2]$:
\begin{equation}
\label{eq:admissible_path}
\varphi_t = \mathrm{Id}, \qquad \text{for all } t \in [0,1/2].
\end{equation}
\end{defn}
\begin{rem}
This requirement \eqref{eq:admissible_path} may seem somewhat artifical; its motivation will become clear in the proof of Lemma \ref{lem:tp_are_fp} below. Note that for every contactomorphism $\varphi : \Sigma \to \Sigma$ which is contact-isotopic to the identity there exists an admissible path $ \widehat{\varphi}$ terminating at $\varphi$: if $ \{ \varphi_t \}_{0 \le t \le 1}$ is any path connecting $ \varphi = \varphi_1$ to $\mathrm{Id} = \varphi_0$, then if $\chi \colon [0,1] \to [0,1]$  is a smooth monotone increasing map with $\chi(\tfrac{1}{2})=0$, the path $ \widehat{\varphi}:= \{ \varphi_{\chi(t)} \}_{0 \le t \le 1}$ is an admissible path.
\end{rem}

Now let $\widehat{\varphi} =  \{ \varphi_t \}_{ 0 \le t \le 1 }$ denote any (not necessarily admissible) smooth path of contactomorphisms from $ \mathrm{Id} = \varphi_0$ to $ \varphi: = \varphi_1$.  Thus by definition there exists a smooth family of positive functions $ \rho_t : \Sigma \to (0 ,+ \infty)$ such that 
\[
\varphi_t^* \alpha = \rho_t \alpha.
\]
The \textbf{contact Hamiltonian} of $\widehat{\varphi}$ is the function $l: \Sigma  \times [0,1] \to \R$ defined by 
\[
l_t  \circ \varphi_t =  \alpha \left( \frac{d}{dt} \varphi_t \right) .
\]
Here, as in many other places in this article, we write $l_t( \cdot) $ for the function $l( \cdot,t)$. Now consider the smooth function 
\begin{equation}
\label{eq:L_is_lifted_l}
L :  S \Sigma \times [0,1 ] \to \R, \qquad L_t(x,r) : = rl_t(x) .
\end{equation}
The Hamiltonian diffeomorphism $\phi_L^s \colon S \Sigma \to S \Sigma$ associated to $L$ is given by 
\begin{equation}
\label{eq:flow_of_L}
\phi_L^s ( x,r ):= \left(\varphi_s(x), \frac{r}{\rho_s(x)} \right).
\end{equation}
\begin{rem}
A point $x \in \Sigma$ is a translated point of $ \varphi_1$ if and only if $  (x,1) \in M$ is a \textbf{leaf-wise intersection point} for $ \phi_L^1$.
\end{rem}

Let us now take this one step further. We will extend $L :S \Sigma \times [0,1] \to \R$ to a function 
\[
	\widetilde{L} :  S \Sigma \times T^* \R  \times [0,1] \to \R.
\]
This requires several preliminary definitions. Define a smooth monotone increasing function $H \colon (0, + \infty) \to [-1,1]$ such that
\begin{equation}
\label{eq:hamiltonian_H}
H(r)  =  \begin{cases}
r-1 , & \text{for all } r \in (\tfrac{1}{2}, \tfrac{3}{2}), \\
  \tfrac{9}{16}, & \text{for all } r \in (\tfrac{7}{4}, + \infty), \\
 - \tfrac{9}{16}, & \text{for all } r \in  (0, \tfrac{1}{4}),
 \end{cases} 
 \qquad \left| \frac{ \partial H}{ \partial r} \right| \le 1,
 \end{equation}
By a slight abuse of notation we denote also by $H$ the function on $S \Sigma$ defined by $H(x,r) = H(r)$. Note that 
\begin{equation}
\label{eq:X_H}
X_H(x,r) = \frac{\partial H }{\partial r}(x,r) R(x).
\end{equation}
Define  $\widetilde{H} \colon S \Sigma \times T^*\R \to \R $ by 
\begin{equation}
\label{eq:widetilde_H}
\widetilde{H} \colon  S \Sigma \times T^* \R  \to \R , \qquad \widetilde{H}(x, r, \tau ,\sigma) := \tau H(x,r) + \tfrac{1}{2} \sigma^2.
\end{equation}
Now let $\kappa:  S^1 \to \R$ denote a smooth
function with
\begin{equation}
\label{eq:cutoff_function_kappa}
\kappa(t)=0\text{ for all } t\in[\tfrac{1}{2},1], \qquad  \text{and} \qquad \int_{0}^{1}\kappa(t)dt=1.
\end{equation}
We use $\kappa$ to modify the function \eqref{eq:widetilde_H}:
\begin{equation}
\label{eq:widetilde_H_kappa}
\widetilde{H}^{\kappa} :  S \Sigma \times T^* \R \times S^1  \to \R , \qquad  H_t^\kappa(x,r,\tau, \sigma) := \tau \kappa(t) H(x,r) + \tfrac{1}{2} \sigma^2.
\end{equation}
and then finally define 
\begin{equation}
\label{eq:widetilde_L}
\widetilde{L} :  S \Sigma \times T^* \R  \times [0,1] \to \R , \qquad \widetilde{L}_t(x,r, \tau, \sigma  ) := \widetilde{H}_t^{\kappa}(x,r, \tau ,\sigma) + L_t(x,r).
\end{equation}

The following lemma is straightforward (compare \cite[Lemma 2.2]{AlbersMerry2013a} and \cite[Appendix A]{AbbondandoloMerry2014a}).

\begin{lem}
\label{lem:tp_are_fp}
Let $\widehat{\varphi} = \{ \varphi_t \}_{0 \le t \le 1}$ denote an admissible path, and let $l_t $ denote its contact Hamiltonian, and define $\widetilde{L}$ as in \eqref{eq:widetilde_L}. Then there is a bijection between the translated points of $\varphi_1$ and the fixed points of $\Phi_{\widetilde{L}}
^1$.
\end{lem}
\begin{proof}
We first need to compute the Hamiltonian flow of the function $\widetilde{H}$ defined in \eqref{eq:widetilde_H}. The Hamiltonian vector field $X_{\widetilde{H}}$ is given by
\[
X_{\widetilde{H}}(x,r,\tau,\sigma) = \tau  X_H(x,r) -  \sigma \partial_{ \tau}+  H(x,r) \partial_{\sigma}.
\]
Using \eqref{eq:X_H}, we see that a path $ s \mapsto  (x(s), r(s), \tau(s) , \sigma(s)) $ is an orbit of $X_{\widetilde{H}}$ if and only if  
\[
\begin{split}
x'(s) &= \tau(s) H'(r(s)) R(x(s)) \\
r'(s) & = 0, \\
\tau'(s) & = - \sigma(s)\\
\sigma'(s) & = H(r(s)).
\end{split}
\]
Thus the flow $\Phi_{\widetilde{H}}^s $ of $X_{\widetilde{H}}$ is given by
\begin{equation}
\label{flow_of_Phi_H}
\Phi_{\widetilde{H}}^s (  x, r, \tau ,\sigma) = \left( \varphi_R^{ \tau H'(r) s}(x), r,  \tau - s\sigma - \frac{1}{2} H(r) s^2, \sigma + sH(r) \right).
\end{equation}
Since the cutoff functions $\kappa $ and $\chi$ have disjoint time support, up to reparametrisation the flow $\Phi_{ \widetilde{L}}^s$ of $\widetilde{L}$ first follows the flow of $X_{\widetilde{H}}$ and then follows the flow of the function $L$ (thought of as a function on $ S \Sigma \times T^* \R \times [0,1]$). 

However, when we regard $L$ as a function on $ S \Sigma \times T^*\R$, since $L$ does not depend on the $\tau$ and $\sigma$ variables, the Hamiltonian flow of $L$ on $S \Sigma \times T^*\R $ will preserve those coordinates.  Thus we see that the Hamiltonian flow of $\widetilde{L}$ is given by:
\begin{equation}
\label{eq:flow_of_Phi_widetilde_L}
\Phi_{\widetilde{L}}^s  \begin{pmatrix} 
x \\ r \\ \tau \\ \sigma \end{pmatrix} 
= \begin{pmatrix}
 \varphi_s \left( \varphi_R^{ \tau H'(r) \int_0^s \kappa(a) \,da }(x) \right) \\
 r \left( \rho_s ( \varphi_R^{ \tau }(x) )\right)^{-1} \\
 \tau - \sigma \int_0^s \kappa(a) \, da - \tfrac{1}{2} H(r) \left( \int_0^s \kappa(a) \,da  \right)^2\\ 
 \sigma + H(r) \int_0^s \kappa(a) \,da  
 \end{pmatrix}
\end{equation}
The $\sigma$-component of \eqref{eq:flow_of_Phi_widetilde_L} tells us that if $(x,r,\tau,\sigma)$ is a fixed point then $H(x,r) = 0$, and so we must have $r=1$. The $\tau$-component tells us that $\sigma = 0$. Then comparing the $x$-component and the $r$-component of \eqref{eq:flow_of_Phi_widetilde_L} with \eqref{eq:translated_point}, we see that $(x, 1, \tau ,0)$ is a fixed point of $\Phi_{\widetilde{L}}^1$ if and only if $ x$ is a translated point of $\varphi_1$. This completes the proof.
\end{proof} 

We would like to extend the function $L$ from \eqref{eq:L_is_lifted_l} to a Hamiltonian defined on all of $M$, and similarly the function $\widetilde{L}$ from \eqref{eq:widetilde_L} to a Hamiltonian defined on  all of $\widetilde{M}$.  This is  easy to accomplish, but we wish to do so in such a way that all $1$-periodic orbits of $X_{\widetilde{L}}$ are left completely  unchanged. This will require a little bit of care; the treatment here follows that of \cite{AlbersMerry2013a}. Given a constant $c >0$, let $\beta_c \in C^{\infty}([0,\infty),[0,1])$
denote a smooth function such that 
\begin{equation}
\label{eq:cutoff_function_beta}
\beta_c(r)=\begin{cases}
1, & r\in[e^{-c},e^c],\\
0, & r\in[0,e^{-2c}]\cup[e^c+1,+\infty),
\end{cases}
\end{equation}
and such that 
\[
0\leq\beta_c'(r)\leq2e^{2c}, \qquad \text{for } r \in [e^{-2c},e^{-c}].
\]
We now consider the function $L^c :  M \times [0,1] \to \R$ defined by
\[
L^c_t(z)= \begin{cases}
\beta_c (r)L_t(x,r) , & z = (x,r) \in  S \Sigma \subset M, \\
0, & z \in  M \setminus S \Sigma.
\end{cases}
\]
The Hamiltonian flow $ \phi_{L^c}^s : M \to M$ of $ L^c$ agrees with that of $\phi_L^s \colon S \Sigma \to S \Sigma $ on the neighbourhood $ \Sigma \times ( e^{-c},e^c) $ of $\Sigma \subset M$.  Next, since the Hamiltonian $H$ defined in \eqref{eq:hamiltonian_H} is constant on $(0, 1/4)$, we can extend $\widetilde{H}$ to all of $ M$ by defining 
\[
\widetilde{H}(z) = - \frac{9}{16} , \qquad \text{for all } z \in M \setminus S \Sigma.
\]
By a slight abuse of notation we continue to denote this extended function also by $\widetilde{H}$. Having done this we extend the modified function $\widetilde{H}^{\kappa} : \widetilde{M} \times S^1 \to \R$ from \eqref{eq:widetilde_H_kappa} similarly and then define as before
\begin{equation}
\label{eq:def_of_Lc}
\widetilde{L}^c  :\widetilde{M} \times [0,1] \to \R , \qquad \widetilde{L}^c_t(z, \tau, \sigma) := \widetilde{H}_t^{\kappa}(x,r, \tau ,\sigma) + L^c_t(z) .
\end{equation}
We shall show that provided the constant $c >0 $ is sufficiently large, the $1$-periodic orbits of $X_{\widetilde{L}^c}$ are unchanged. The following argument is taken from \cite[Proposition 2.5]{AlbersMerry2013a}. Suppose $\widetilde{z}(t) = (x(t),r(t),\tau(t), \sigma(t))$ is a 1-periodic orbit of $X_{\widetilde{L}^c}$. As before we see that $\tau(t)\equiv \tau$ is constant and $\sigma(t) \equiv 0$. Moreover we know that $r(t) = 1$  for all $t \in[0, 1/2]$. Thus if we set
\[
S:=\left\{ t\in S^1 \mid r(t)\in(e^{-c},e^c)\right\} 
\]
 then $S$ is a non-empty open interval containing the interval $[0,\tfrac12]$. Let $S_0 \subseteq S$ denote the connected component containing
0. We show that $S_0$ is closed, whence $S_0=S=[0,1]$. 
If $x(t)\in\Sigma\times(e^{-c},e^c)$ and $t\in[\tfrac12,1]$ then $r(t)$
satisfies the equation
\[
\dot{r}(t)=-\frac{\dot{\rho}_t(x(t))}{\rho_t^2(x(t))}\cdot r(t).
\]
Define a constant $C(\widehat{\varphi}) \ge 0$ by 
\begin{equation}
\label{eq:C_phi}
C(\widehat{\varphi}):=\max_{t\in[0,1]}\int_0^t\max_{x\in\Sigma}\left| \frac{\dot{\rho}_{\tau}(x)}{\rho_{\tau}(x)^{2}} \right| d\tau.
\end{equation}
We see that for $t\in S_0 \cap[\tfrac12,1]$
it holds that 
\[
e^{-C}\le r(t)\leq e^C.
\]
In particular, provided we choose the constant $c$ to satisfy $c>C(\widehat{\varphi})$ then we see that $S_{0}$ is closed. We have proved:
\begin{lem}
\label{lem:where_the_orbits_live}
If $c > C(\widehat{\varphi})$,  then every 1-periodic orbit of $X_{\widetilde{L}^c}$ has image contained in $\Sigma \times (e^{-c},e^c) \times O_{\R}$, where $O_{\R} \subset T^*\R$ denotes the zero section. In particular, every 1-periodic orbit is contained in the subset  $ \{ \widetilde{L}^c \equiv \widetilde{L} \} \subset \widetilde{M}$.
\end{lem}

\section{Floer homology on the extended phase space}
\label{sec:Floer_homology}

We denote by $ \Lambda (\widetilde{M}) $ the space of all smooth loops $ \widetilde{z} : S^1 \to \widetilde{M}$. We typically write $ \widetilde{z}(t)=(z(t), \tau(t), \sigma(t))$, so that $z  \colon S^1 \to M $ is a loop in $M $ and $(\tau(t), \sigma(t))$ is a loop in $T^*\R$. 

\begin{defn}
We denote by $\A_{\widetilde{L}^c} : \Lambda(\widetilde{M}) \to \R$ the classical Hamiltonian action functional associated to the Hamiltonian $\widetilde{L}^c$ from \eqref{eq:def_of_Lc}, defined by
\[
\A_{\widetilde{L}^c}(\widetilde{z}) := \int_{S^1}z^* \lambda - \int_{S^1}\langle \dot{\tau}, \sigma \rangle  \, dt -\int_{S^1}\widetilde{L}^c_t( \widetilde{z}) \, dt,
\]
where we wrote $\widetilde{z}=(z, \tau ,\sigma) $ as above. 
\end{defn} 

The critical points of $ \A_{ \widetilde{L}^c}$ are precisely the contractible 1-periodic orbits of $X_{\widetilde{L}^c}$. The aim of this section is to explain how to construct the \textbf{Floer homology groups $\HF_*(\A_{\widetilde{L}^c})$} associated to $\A_{\widetilde{L}^c}$. The construction is very standard, apart from in two respects. Namely, the Hamiltonian $\widetilde{L}^c $ is \textbf{not} coercive. As a result obtaining the $L^{\infty}$-bounds required to define the boundary operator  is rather involved. This difficulty was solved by  Abbondandolo and the first author in \cite{AbbondandoloMerry2014a}. The setting in this paper is slightly different though, and thus we will go through the compactness statements in detail below, see Section \ref{sec:proof_of_thm_estimate}. Secondly, there is the question of transversality. The compactness statements proved in Section \ref{sec:proof_of_thm_estimate} require us to work with almost complex structures of a specific form, introduced in Definition \ref{defn:tau_dep_acs} below. Thus one needs to know that transversality can be achieved within this class of almost complex structures. This is by no means obvious, but the proof in \cite[Section 6]{AbbondandoloMerry2014a} carries through verbatim here, and hence we will not dwell on this issue.

Here is the aformentioned class of almost complex structures that we will work with.
\begin{defn}
\label{defn:tau_dep_acs}
We denote by $\mathcal{J}$ the set of smooth families 
\[
\mathrm{J} =  \{ J_t(\cdot, \tau) \}_{ (t, \tau) \in S^1 \times \R}
\]
  of almost complex structures on $M$, which are compatible with $d \lambda$, meaning that for each $(t,z,\tau) \in S^1 \times M \times \R$,
\[
\langle \cdot,\cdot \rangle_{J_t(z,\tau)} := d \lambda_z(J_t(z, \tau)\cdot,\cdot), 
\]
defines a Riemannian metric on $T_z M$, whose associated norm is denoted by $|\cdot|_{J_t(z,\tau)}$ (\textbf{warning:} this sign convention is slightly unusual). In addition we require that
\begin{equation}
\label{eq:acs_bounded}
  \sup_{ (t,\tau) \in S^1 \times \R } \| J_t(\cdot , \tau) \|_{C^k} < + \infty, \qquad \forall \, k \in \N,
\end{equation}  
where $\| \cdot \|_{C^k} $  is the norm taken with respect to some background metric on $M$.  Finally we require that $\mathrm{J}$ is of \textbf{contact type at infinity}, which means that there exists $r_0> 2$ such that the pulled back almost complex structure $I^*(J_t(\cdot ,\tau))$ of $J_t(\cdot, \tau)$ on $\Sigma \times (r_0,+\infty)$ is independent of both $t\in S^1$ and $\tau \in \R$, and satisfies
\begin{equation}
\label{ct}
dr \circ I^*(J_t(\cdot ,\tau)) = r \alpha \qquad \text{on} \quad \Sigma \times (r_0,+\infty).
\end{equation}
Given $ \mathrm{J }\in \mathcal{J}$ we then consider the loop $\widetilde{J}_t$ of almost complex structures on $\widetilde{M}$ which is defined for $\widetilde{z} = (z, \tau, \sigma) \in \tilde{M}$ by
\begin{equation}
\label{special form of acs}
\widetilde{J}_t(\widetilde{z}) = J_t (z,\tau) \oplus \begin{pmatrix} 0 & 1 \\ -1 & 0 \end{pmatrix}  : T_zM \oplus T_{(\tau,\sigma)}T^* \R \to T_zM \oplus T_{( \tau, \sigma)}T^*\R.
\end{equation}
Thus $\widetilde{J}_t$, $t \in S^1$, is a loop of almost complex structures compatible with $\omega$. The corresponding metric 
\[
\langle \cdot ,\cdot \rangle_{\widetilde{J}_t(\widetilde{z})} := \omega_{\tilde{z}} ( \widetilde{J}_t (\widetilde{z})\cdot, \cdot)
\]
is the product metric of $\langle \cdot,\cdot\rangle_{J_t(z, \tau)}$ with the Euclidean metric of $T^* \R \cong \R^2$. 
\end{defn}

Fix $\mathrm{J} \in \mathcal{J}$. We denote by $\llangle \cdot, \cdot \rrangle_{\mathrm{J}}$ the induced $L^2$-inner product on $ \Lambda (\widetilde{M})$ arising from $\left\langle \cdot,\cdot \right\rangle_{\widetilde{J}_t(\widetilde{z})} $. 
The $L^2$-gradient of $\A_{\widetilde{L}^c}$ has the form
\begin{equation}
\label{eq:the_gradient}
\nabla \A_{\widetilde{L}^c} (\widetilde{z}) =  \begin{pmatrix}
J_t(z,\tau) \left(z' - \tau \kappa(t)  X_ H(z) - X_{L^c_t}(z) \right) \\ \sigma' - \kappa(t) H(z) \\ -\tau' + \sigma 
\end{pmatrix},
\end{equation}
for $\widetilde{z}=(z,\tau, \sigma) \in \Lambda (\widetilde{M})$. Thus the Floer negative gradient equation for $\A_{\widetilde{L}^c}$, that is,
\[
\frac{d\widetilde{u}}{ds}  + \nabla \A_{ \widetilde{L}^c} (\widetilde{u})=0, \qquad \text{for } \widetilde{u} : \R \rightarrow \Lambda(\widetilde{M}),
\]
is the following system of PDEs 
\begin{equation}
\label{eq:floer}
\begin{split}
\partial_s u + J_{t}(u,\eta) \left( \partial_t u - \eta \kappa(t) X_H(u)- X_{L^c_t}(u) \right) &= 0, \\
\partial_s \eta + \partial_t \zeta -  \kappa(t) H(u) &= 0, \\
\partial_s \zeta - \partial_t \eta + \zeta &= 0.
\end{split}
\end{equation}
for
\[
\widetilde{u} = (u,\eta,\zeta) : \R \times S^1 \rightarrow M \times T^* \R = \widetilde{M}.
\]
We are interested in finite-energy solutions of the above system, that is in solutions $\widetilde{u} = (u,\eta,\zeta)$ for which the quantity
\begin{equation}
\label{eq:energy}
\E(\widetilde{u}) := \int_{- \infty}^{ + \infty} \int_{S^1}  \llangle \partial_s \widetilde{u} , \partial_s \widetilde{u}  \rrangle_{ \mathrm{J}} \, ds \, dt 
\end{equation}
is finite. Note that as $\widetilde{u}$ is a \textbf{negative} gradient flow line, one has
\[
\begin{split}
\E(\widetilde{u}) = - \int_{- \infty}^{ +\infty} \frac{d}{ds} \A_{\widetilde{L}^c}(\widetilde{u}(s,\cdot))\, ds &= 
\lim_{s\rightarrow -\infty}  \A_{\widetilde{L}^c}(\widetilde{u}(s,\cdot)) - \lim_{s\rightarrow +\infty} \A_{\widetilde{L}^c}(\widetilde{u}(s,\cdot)) \\ &= \sup_{s\in \R} \A_{\widetilde{L}^c}(\widetilde{u}(s,\cdot)) - \inf_{s\in \R}\A_{\widetilde{L}^c}(\widetilde{u}(s,\cdot)).
\end{split}
\]
\begin{defn}
We define the \textbf{action spectrum} of $\A_{\widetilde{L}^c}$ to be its set of critical values:  
\[
\spec( \A_{\widetilde{L}^c}) := \A_{\widetilde{L}^c}(\crit \A_{\widetilde{L}^c}).
\]
\end{defn}

\begin{lem}
\label{lem:relating_spectrums}
Let $\widehat{\varphi}$ denote an admissible path terminating at $\varphi$. Then if $c > C( \widehat{\varphi})$ one has 
\[
\spec( \A_{\widetilde{L}^c}) = \spec ( \varphi_1).
\]
\end{lem}
\begin{proof}
Suppose $\widetilde{z}$ is a critical point of $ \A_{\widetilde{L}^c}$ for some $c > C( \widehat{\varphi})$. Then by Lemma \ref{lem:tp_are_fp} and Lemma \ref{lem:where_the_orbits_live}, we can write  $\widetilde{z}(t) = (x(t), r(t), \tau, 0)$, such that if $p := x(0)$ then $\varphi_1(  \varphi_R^{\tau}(p))= p $. It thus suffices to show that
\begin{equation}
\label{eq:action_value}
\A_{\widetilde{L}^c}(\widetilde{z}) = \tau.
\end{equation}
For this we compute:
\begin{align*}
\A_{\widetilde{L}^c}(\widetilde{z}) & = \int_{S^1}z^* \lambda - \int_{S^1}\underbrace{\langle \dot{\tau}, \sigma \rangle }_{ = 0} \, dt -\int_{S^1}\widetilde{L}^c_t( \widetilde{z}) \, dt \\
& = \int_{S^1}z^* \lambda  - \tau \int_0^{1/2} \kappa(t)\underbrace{H(r(t))}_{=0} \,dt - \int_{1/2}^1 L^c_t( z) \, dt \\
& = \tau \underbrace{\int_0^{1/2} \kappa(t) \alpha(R(x(t))) \,dt}_{= 1}+ \int_{1/2}^1 \underbrace{\lambda(X_{L}(z)) - L_t(z)}_{=0}dt \\
& = \tau.
\end{align*}
\end{proof}

The key compactness statement, which is very similar to \cite[Proposition 1.1]{AbbondandoloMerry2014a}, is the following result.

\begin{thm}
\label{thm:estimate}
Fix $\mathrm{J} \in \mathcal{J}$. Then for any $A\in \R$ there is a number $C=C(A)$, such that for every solution $\widetilde{u} = (u,\eta,\zeta)$ of the Floer equation \eqref{eq:floer} with
\[
| \A_{\widetilde{L}^c}(\widetilde{u}(s))| \leq A \qquad \text{for all }s\in \R,
\]
one has
\[
\|\eta\|_{L^{\infty}(\R\times S^1)} \leq C \qquad \|\zeta\|_{L^{\infty}(\R\times S^1)} \leq C, \qquad u(\R \times S^1) \subset ( M_1 \cup_{ \Sigma } ( \Sigma \times (1,r_0]).
\]

\end{thm}
The proof is deferred to Section \ref{sec:proof_of_thm_estimate} below. 
\begin{defn}
We say that an admissible path $\widehat{\varphi} = \{ \varphi_t \}_{0 \le t \le 1}$ is \textbf{non-degenerate} if the action functional $\A_{\widetilde{L}^c}$ is a Morse function for some (and hence any) $c > C(\widehat{\varphi})$.  By combining \cite[Theorem 1.4]{AlbersMerry2013a} and \cite[Section 6]{AbbondandoloMerry2014a} we see that a generic admissible path is non-degenerate.
\end{defn}

As mentioned earlier, the following theorem can be proved in exactly the same way as \cite{AbbondandoloMerry2014a}. 
\begin{thm}[\cite{AbbondandoloMerry2014a}]
\label{thm:transversality}
The set $\mathcal{J}$ introduced in Definition \ref{defn:tau_dep_acs} is rich enough for transversality to hold: there exists a comeagre subset $\mathcal{J}_{\mathrm{reg}}(\widetilde{L}^c) \subset \mathcal{J}$ such that for $ \mathrm{J } \in \mathcal{J}_{\mathrm{reg}}(\widetilde{L}^c)$ the linearisation of the problem \eqref{eq:floer} is onto.
\end{thm}

Theorems \ref{thm:estimate} and \ref{thm:transversality} imply that for a non-degenerate admissible path we can speak of the filtered  \textbf{Floer homology} $\HF^{(a,b)}_*(\A_{ \widetilde{L}^c})$ for $a ,b \in (- \infty, + \infty] \setminus \spec(\A_{ \widetilde{L}^c})$, $a < b$. This is the homology of a chain complex whose generators are the 1-periodic orbits $\widetilde{z}$ of $X_{\widetilde{L}^c}$ with action $ \A_{ \widetilde{L}^c}(\widetilde{z}) \in (a,b)$. The boundary operator is defined by counting ``rigid" negative gradient flow lines $ \widetilde{u}$ of $\A_{ \widetilde{L}^c}$ (i.e. Fredholm index one) connecting different 1-periodic orbits of $X_{\widetilde{L}^c}$. This  Floer homology depends only on the  admissible path $\widehat{\varphi}$ and the filling $(M_1, d \lambda_1)$ of our contact manifold $(\Sigma, \alpha)$. We will use the shorthand  notation 
\[
\HF^{(a,b)}_*( \widehat{\varphi}) := \HF^{(a,b)}_*( \A_{ \widetilde{L}^c})
\]
to denote this Floer theory.  We abbreviate $\HF^a_*(\widehat{\varphi}) : = \HF^{(- \infty, a)}_*(\widehat{\varphi})$ and $\HF_*( \widehat{\varphi}) : = \HF^{ + \infty}_*(\widehat{\varphi})$, and we always tacitly assume that the endpoints of the action windows do not belong to $\spec( \A_{\widetilde{L}^c})$, even if this is not explicitly said. \\

The filtered Floer homology is stable under sufficiently small perturbations. This allows us to extend the definition of $\HF_*^{(a,b)}( \widehat{\varphi})$ to the case where the admissible $ \widehat{\varphi}$ is not necessarily non-degenerate. Namely, after making a $C^{\infty}$-small perturbation, one obtains a new admissible path $\widehat{\varphi}'$ that is non-degenerate.  The aforementioned stability property implies that one can unambiguously define
 \begin{equation}
 \label{eq:degnerate_case}
 \HF_*^{(a,b)}( \widehat{\varphi}) := \HF_*^{(a,b)}( \widehat{\varphi}').
 \end{equation}

Given $ a < b$ and $a' < b'$ such that $a < a'$ and $b < b'$, there is a well defined map $\HF_*^{(a,b)}(\widehat{\varphi}) \to \HF_*^{(a',b')}( \widehat{\varphi}) $. We now use these maps to define the \textbf{positive growth rate}. 

\begin{defn}
\label{defn:growth_rate}
Let $a_0$ denote any finite real number not belonging to $\spec( \A_{ \widetilde{L}^c})$, and let $\{a_k \}_{ k = 0 ,1 ,2 ,\dots }$ be any sequence of real numbers $a_0 < a_1 <a_2 < \dots$ such that $a_k \to \infty$ and such that $a_k \notin \mathrm{Spec}(\A_{\widetilde{L}^c})$. We define the \textbf{positive growth rate} of $\widehat{\varphi}$ to be 
\[
\Gamma_+(\widehat{\varphi}) := \limsup_{k \to +\infty} \frac{\log \left( \dim \left( \mathrm{im} \left[ \HF^{(a_0, a_k)}_*( \widehat{\varphi}) \to \HF_*^{(a_0, + \infty)}(\widehat{\varphi}) \right] \right) \right) }{\log k}.
\]
The number $\Gamma_+( \widehat{\varphi})$ takes values in $\{ - \infty\} \cup [0, + \infty]$  and does not depend on the choice of $a_0$ and the $a_k$. The word ``positive'' is a slight misnomer (as $a_0$ does not need to be positive); nevertheless the motivation for the choice of name will shortly become clear. 
\end{defn}

Theorem \ref{thm:actual_theorem} is stated in terms of the \textbf{Rabinowitz Floer homology} of the Liouville domain $(M_1,\lambda_1)$. Rabinowitz Floer homology was discovered by Cieliebak and Frauenfelder \cite{CieliebakFrauenfelder2009}, and has since generated many applications. We refer the reader to the survey paper \cite{AlbersFrauenfelder2012a} and the references therein for more information. We will not define Rabinowitz Floer homology here, but instead list the properties that we need:

\begin{enumerate}
  \item The Rabinowitz Floer homology is an invariant of a Liouville domain $(M_1, \lambda_1)$. The underlying chain complex is  a free $\Z_2$-module generated by closed orbits in $\Sigma := \partial M_1$ of the Reeb vector field $R$ arising from the contact form $ \alpha := \lambda_1|_{\Sigma} $, together with their inverse parametrisations, and the points of $\Sigma$, interpreted as constant loops. 
  \item The Rabinowitz Floer homology $\RFH_*(M_1,\lambda_1)$ is equipped with an $\R$-filtration, where the subcomplex $\RFH_*^{(a,b)}(M_1, \lambda_1)$  is generated by those orbits with period in $(a,b)$. 
  \item The \textbf{positive} Rabinowitz Floer homology is defined as
  \[
  \RFH_*^+(M_1,\lambda_1) := \RFH^{(-\varepsilon , + \infty)}_*(M_1, \lambda_1),
  \]
  where $ \varepsilon $ is any sufficiently small positive number. 
  \item \cite[Proposition 1.4]{CieliebakFrauenfelderOancea2010} There is a long exact sequence relating $\RFH_*(M_1, \lambda_1)$ with the \textbf{symplectic homology} of $(M_1, \lambda_1)$:
  \[
  \dots \to \mathrm{H}^{-*+n}(M_1, \Sigma ;\Z_2) \to \mathrm{SH}_*(M_1, \lambda_1) \to \RFH_*(M_1, \lambda_1) \to \mathrm{H}^{-*+ 1 +n}(M_1, \Sigma) \to \dots  
  \] 
\item The \textbf{positive growth rate} $\Gamma_+(M_1, \lambda_1)$ is defined in a similar way to Definition \ref{defn:growth_rate}, and was first introduced by Weigel \cite{Wiegel2013}. Namely, one chooses an increasing  sequence $\{ a_k \}_{k \in \N} $ of positive real numbers such that $a_k \to \infty$ and defines 
\[
\Gamma_+(M_1, \lambda_1) := \limsup_{k \to +\infty} \frac{\log \left( \dim \left( \mathrm{im} \left[ \RFH^{(-\varepsilon, a_k)}_*( M_1, \lambda_1) \to \RFH_*^+(M_1, \lambda_1) \right] \right) \right) }{\log k}.
\]
This number takes values in $\{ - \infty\} \cup [0, +\infty]$. It follows from a result of McLean \cite{McLean2011} and the long exact sequence above that the positive growth rate is invariant under Liouville isomorphism.
\item  Let $(Q,g)$ denote a closed Riemannian manifold, and let $D_g^*Q$ denote the unit disk bundle. Let $\lambda$ denote the canonical Liouville 1-form and $\lambda_1 := \lambda|_{D^*_gQ} $. Then $(D^*_g Q, \lambda_1)$ is a Liouville domain. It follows from \cite{Viterbo1996,SalamonWeber2006,AbbondandoloSchwarz2006} and \cite{Gromov1978} that if the function  $ k \mapsto \mathrm{rank\,H}_k(\Lambda(Q); \Z_2)$ is asymptotically unbounded then one has  $\Gamma_+(D^*_gQ, \lambda_1) > 1$. 
\item Combining the last two points, we see that if $Q$ is a closed manifold such that  he function  $ k \mapsto \mathrm{rank\,H}_k(\Lambda(Q); \Z_2)$ is asymptotically unbounded then for any fibrewise star-shaped hypersurface $\Sigma \subset T^*Q$, if $D(\Sigma)$ denotes the compact region bounded by $\Sigma$ then $\Gamma_+(D(\Sigma), \lambda|_{D(\Sigma)}) > 1$. 
\end{enumerate}
The reason we are interested in Rabinowitz Floer homology is the following result, which is the main theorem in \cite{AbbondandoloMerry2014a}.
\begin{thm}
\label{thm:HF=RFH}
Given any non-degenerate admissible path $\widehat{\varphi}$, there is a canonical isomorphism between $\HF_*( \widehat{\varphi})$ and the Rabinowitz Floer homology of the pair $ (M_1, \lambda_1)$:
\[
\HF_*( \widehat{\varphi})  \cong \RFH_*(M_1, \lambda_1).
\]
Moreover one has
\begin{equation}
\label{eq:growth_rates_equal}
\Gamma_+( \widehat{\varphi}) = \Gamma_+(M_1, \lambda_1).
\end{equation}
\end{thm}
The following result is the main one of this paper. Theorem \ref{thm:actual_theorem} is an immediate consequence of it, Theorem \ref{thm:HF=RFH} and point (7) above. 

\begin{thm}
\label{thm:infinitely_many}
Let $(\Sigma, \alpha)$ denote a Liouville fillable contact manifold. Suppose $\varphi$ is a strict contactomorphism with the property that there are only finitely many invariant Reeb orbits. Then if $(M_1,\lambda_1)$ is any Liouville filling of $(\Sigma, \alpha)$ one has
\[
\Gamma_+(M_1, \lambda_1) \le 1.
\]
\end{thm}
 We will prove Theorem \ref{thm:infinitely_many} in Section \ref{sec:local_floer_homology} below.

\section{Local Floer homology}
\label{sec:local_floer_homology}

Our main tool for proving Theorem \ref{thm:infinitely_many} uses the idea of \textbf{local Floer homology}. The idea behind local Floer homology dates back to Floer, and was first systematically exploited in \cite{CieliebakFloerHoferWysocki1996}. Much later Ginzburg used local Floer homology with spectacular success to prove the Conley Conjecture for symplectically aspherical manifolds \cite{Ginzburg2010} (it has since been proved in numerous other situations by Ginzburg, G\"urel and Hein). In this paper we use a minor extension of a result of Ginzburg and G\"urel \cite{GinzburgGurel2010} on the so-called ``persistence of local Floer homology''.  

\begin{rem}
The idea of using Ginzburg and G\"urel's result to prove Gromoll-Meyer type results is not new; Ginzburg and G\"urel themselves indicate such results should be possible \cite[p326]{GinzburgGurel2010}. Moreover as mentioned in the Introduction, both McLean \cite{McLean2012} and Hryniewicz-Macarini \cite{HryniewiczMacarini2012} use this same persistence property to prove related results.
\end{rem}

\subsection{The definition of local Floer homology}
\label{subsec:def_of_local_FH}
The local Floer homology groups are valid in far more general situations than the restricted setting outlined in the previous setting. In fact, the local Floer homology groups can essentially \textbf{always} be defined, whereas in general to speak of the standard (Hamiltonian) Floer homology one needs to make additional assumptions on either the symplectic manifold or the Hamiltonian. For instance, we are always concerned with non-compact symplectic manifolds, and in this case one needs to impose conditions on the behaviour of the Hamiltonians at infinity. \\

Nevertheless, for the sake of a uniform presentation thoughout this section we assume that $(W^{2n},\omega)$ is a \textbf{symplectically atoroidal} manifold. This means that for any smooth map $u : \T^2 \to W$, one has $\int_{\T^2}u^*\omega = 0$. In addition for simplicity we will assume that $c_1(TW)$ is torsion. Suppose $L \in C^\infty(W \times S^1, \R)$. We denote by 
\[
\A_L : \Lambda (W) \to \R, \qquad \A_L(z) := \int_{[0,1] \times S^1} \bar{z}^* \omega - \int_{S^1} L_t(z)\,dt,
\]
where $\bar{z} : [0,1] \times S^1 \to W$ is a family of loops such that $z(0,t) = z(t)$ and $z(1,t) = z_{\mathrm{ref}}(t)$ is some fixed reference loop belonging to the same free homotopy class as $z$. We denote by $\mathcal{P}_1(L) = \crit \A_{L}$ the set of 1-periodic orbits of $X_L$.
\begin{defn}
A subset $\Gamma \subset \mathcal{P}_1(L)$ is said to be \textbf{action-constant} if 
\[
w,z \in \Gamma \qquad \then \qquad \A_{L}(w)=\A_L(z). 
\]
Note that if the subset $\Gamma$ is connected (as a subset of $\Lambda( W)$) then it is automatically action-constant.
\end{defn}
We denote by 
\[
\mathrm{gr}(\Gamma) := \left\{ (z(t), t) \mid  \, z \in \Gamma , t \in S^1\right\} \subset W \times S^1
\]
the graph of the elements of $\Gamma$. Similarly we denote by 
\begin{equation}
\label{eq:ev_Gamma}
P(\Gamma) := \left\{ z(0)  \mid z \in \Gamma \right\} .
\end{equation}
Thus $P(\Gamma) \subset \mathrm{Fix}( \phi_L^1)$. Going the other way, given a Hamiltonian diffeomorphism $\phi$ and a subset $P \subset \mathrm{Fix}(\phi)$, a choice of Hamiltonian $L$ generating $\phi$ gives rise to a subset $\Gamma_L(P) \subset \mathcal{P}_1(L)$ given by 
\[
\Gamma_L(P) = \left\{ t \mapsto \phi_L^t(x) \mid x \in P \right\}.
\]
\begin{defn}
We say that $\Gamma$ is an \textbf{isolated} subset of $\mathcal{P}_1(L)$ if there exists an open precompact subset $N \subset  W \times S^1$ containing $\mathrm{gr}(\Gamma)$ such that if $w \in \mathcal{P}_1(L)$ is any contractible 1-periodic orbit of $X_L$ then
\[
\mathrm{gr}(w) \cap N \ne \emptyset \qquad \then \qquad w  \in \Gamma.
\]
In the case where $L$ is an autonomous Hamiltonian one can equivalently take $N \subset W$ and replace the condition above with the assertion that if $w(S^1) \subset N \ne \emptyset$ then $w \in \Gamma$. 
\end{defn}
\begin{rem}
Suppose $\phi$ is a Hamiltonian diffeomorphism of $(W, \omega)$ and $P \subset \mathrm{Fix}(\phi)$. Then if $L_1$ and $L_2$ are two different Hamiltonians that generate $\phi$ and $\Gamma_1$ and $\Gamma_2$ the corresponding subsets of $\mathcal{P}_1(L_1)$ and $\mathcal{P}_1(L_2)$ such that
\[
P(\Gamma_1 ) = P = P(\Gamma_2),
\]
then $\Gamma_1$ is an isolated subset of $\mathcal{P}_1(L_1)$ if and only if $\Gamma_2$ is an isolated subset of $\mathcal{P}_1(L_2)$. Thus it makes sense to say that a subset $P \subset \fix(\phi)$ is isolated if the corresponding subset $\Gamma_L(P)$ is isolated in $\mathcal{P}_1(L)$ for any Hamiltonian $L$ generating $\phi$.
\end{rem} 

Following McLean \cite{McLean2012}, we  define the local Floer homology $\HF^{ \loc}_*(L ,\Gamma)$ associated to an action-constant isolated subset of $\mathcal{P}_1(L)$. We need the following three facts about such a subset $\Gamma$:\\

\begin{enumerate}
  \item Let $\mathrm{pr} :W \times S^1 \to W $ denote the projection onto the first factor, and let $U := \mathrm{pr}(N)$. Given $\delta > 0$, let $\mathcal{F}(U,\delta)$ denote the set of all smooth functions 
  \[
  \mathcal{F}(U, \delta) := \left\{  F \in C^\infty(W \times S^1 ) \mid \mathrm{supp}(F) \subset U \times S^1 \subset N , \ \| F \|_{C^1( W \times S^1)} < \delta \right\}.
  \]
  Then there exists $\delta_0 >0 $ with the property that if $0 < \delta < \delta_0$ and $F \in \mathcal{F}(U,\delta)$, then if $w$ is any 1-periodic orbit of $X_{L+F}$, one has:
  \[
\gr( w) \cap N \ne \emptyset \qquad \then \qquad \mathrm{gr}(w) \subset N.
\] 
\item For all $\delta>0 $ there exists $F \in \mathcal{F}(U,\delta)$ such that if $w$ is any 1-periodic orbit $w$ of $X_{L+F}$, one has
\[
\gr( w) \cap N \ne \emptyset \qquad \then \qquad w \text{ is non-degenerate},
\] 
that is, $1$ is not an eigenvalue of the linear map $D \phi_{L+F}^1(w(0)) \colon T_{w(0)} W \to T_{w(0)} W$.
\item Suppose $J = \{ J_t \}_{t \in S^1}$ is a family of almost complex structures on $W$ that are $\omega$-compatible. Suppose $N$ is an isolating neighbourhood, and $F \in \mathcal{F}(U,\delta)$, where $U = \mathrm{pr}(N)$. Let $\mathcal{M}(L,F,J,N)$ denote the set of all finite energy maps $u : \R \times S^1 \to U$ which satisfy the Floer equation $ \partial_s u + J_t(u)( \partial_t u - X_{L+F}(t,u)) = 0$. Then the following holds: suppose $N_1 \subset N_2$ are two isolating neighbourhoods of $\Gamma$, with corresponding sets $U_j := \mathrm{pr}(N_j)$. Then there exists $\delta_1 >0 $ such that if $0 < \delta < \delta_1$ and $F \in \mathcal{F}(U_1, \delta)$ then 

\[
\mathcal{M}(L,F,J,N_1)= \mathcal{M}(L,F,J,N_2).
\]
\end{enumerate}

\label{page:explanation}For instance, to prove (1), we argue by contradiction:  If the conclusion is false then we can find sequences $\delta_k \to 0$, elements $F_k \in \mathcal{F}(U,\delta_k)$, and 1-periodic orbits $z_k$ of $X_{L+F_k}$ whose graphs intersect $\partial N$. Since $\sup_{ k \in \Z } \| z_k '\|_{L^2(S^1)} < + \infty$, by combining the Sobolev embedding $W^{1,2}(S^1,W) \hookrightarrow C^0(S^1,W)$ and applying the Arzela-Ascoli Theorem, we deduce that (after possibly passing to a subsequence) there exists $w \in C^0(S^1 ,W)$ such that $z_k \stackrel{C^0}{\to} w$.  Then $w$ is necessarily a 1-periodic orbit of $X_L$, and since $N$ was an isolating neighbourhood for $\Gamma$, in fact $w  \in \Gamma$. But then as $z =w $ is the limit of the $z_k$'s, we also see that $\mathrm{gr}(\Gamma)$ intersects the boundary of $N$. This is a contradiction. The proof of (3) is similar: if as before one finds a sequence $\delta_k \to 0$, a sequence $F_k \in \mathcal{F}(U_1, \delta_k)$, and a sequence $u_k \in \mathcal{M}(L,F_k, J_k, N_2) \setminus \mathcal{M}(L,F_k, J_k, N_1)$ then in the limit Gromov compactness tells us we find an element $u \in \mathcal{M}(L , 0 ,J, N)$. Such a flow line is necessarily constant, and this contradicts the assumption that $N_1$ is an isolating neighbourhood of $\Gamma$.  Actually strictly speaking this argument is not entirely rigorous; a more sophisticated compactness result than the standard Gromov compactness is required in order to deal with the case where the Hamiltonian $L$ is degenerate. See \cite[p1909]{McLean2012} for more details. Finally, (2) can either be proved via a standard Sard-Smale transversality argument, or by a local construction as in  \cite[Theorem 9.1]{SalamonZehnder1992}. \\

The upshot of points (1), (2) and (3) is the following. Fix $\delta>0$ sufficiently small that (1) and (3) hold, and choose $F \in \mathcal{F}(U,\delta)$ such that (2) holds. Define 
\[
\mathrm{CF}^{\loc}_*( L,F,N) := \bigoplus_{w} \Z_2 \left\langle w \right\rangle,
\]
where the sum is over all 1-periodic orbits $w$ of $X_{L+F}$ whose graph intersect $N$. Fix a generic loop $J = \{ J_t \}_{t \in S^1}$ of $\omega$-compatible almost complex structures, and define a boundary operator $\partial $ on $\mathrm{CF}^{\loc}_*( L,F,N)$ as the linear operator
\[
\left\langle w \right\rangle  \mapsto \sum_{w'} n(w, w') \left\langle w' \right\rangle,
\] 
where the matrix coefficient $n(w,w')$ is the number of ``rigid'' (i.e. Fredholm index 1) elements of $\mathcal{M}(L,F,J,N)$ connecting $w$ to $w'$. The resulting homology is denoted by 
\[
\HF^{\loc}_*(L, \Gamma)
\]
and called the \textbf{local Floer homology} of $L$ at $\Gamma$. As the notation suggests, these groups are independent of the various auxilliary choices made; this is proved using a suitable $s$-dependent version of statement (3) above.  

\begin{rem}
Transversality in the local setting can be attained within a certain class of almost complex structures if that class is also rich enough for transversality to hold in the construction of the full Floer homology groups. This is important, since it shows that in our setting when constructing local Floer homology groups we are free to use almost complex structures $\mathrm{J} \in \mathcal{J}$ (cf. Definition \ref{defn:tau_dep_acs} and Theorem \ref{thm:transversality}).
\end{rem}

The following result follows essentially from the definition.

\begin{lem}
\label{lem:building_blocks}
Let $L$ be a Hamiltonian with the property that the full Floer homology groups $\HF(L)$ are well defined. Suppose also that $\mathcal{P}_1(L)$ can be written as a disjoint union of isolated action-constant sets $\{ \Gamma_k \}_{k \in \N}$. Set $c_k := \A_L (\Gamma_k)$. Then for any interval $(a,b) \subset \R$, one has
\[
  \mathrm{rank\,}\HF^{(a,b)}(L) \le \sum_{ k \colon c_k \in (a,b)} \mathrm{rank\,}\HF^{\loc}(L , \Gamma_k).
\]
\end{lem}

In fact, up to a grading shift, the groups $\HF^{\loc}_*(L,\Gamma)$ depend only on $\phi =  \phi_L^1$ and the set $P=  P(\Gamma) \subset \mathrm{Fix}(\phi_L^1)$ from \eqref{eq:ev_Gamma}. Thus we will often use the notation $\HF^{\loc}(\phi, P)$ instead. Here the lack of ``$*$'' is meant to serve as a reminder that the grading is now only defined up to a shift. \\

\begin{ex}
\label{ex:morse_bott_component}
Suppose that $\Gamma \subset \mathcal{P}_1(L)$ is a \textbf{Morse-Bott} component. This means that $P = P(\Gamma)$ is a compact submanifold of $W$ with the property that
\begin{equation}
\label{eq:morse_bott}
T_x P = \ker \left( D \phi_L^1 (x) -I \right), \qquad \text{for all } x \in P. 
\end{equation}
Such a component is necessarily isolated, and each connected component is action-constant. Then a result of Biran-Polterovich-Salamon \cite[Theorem 5.2.2]{BiranPolterovichSalamon2003} tells us that 
\begin{equation}
\label{eq:morse_singular}
\HF^{\loc}(L ,\Gamma) \cong \mathrm{H}^{\mathrm{sing}}(P ; \Z_2).
\end{equation}
\end{ex}

\begin{rem}
\label{rem:no_star_means_no_grading}
In equation \eqref{eq:morse_singular}, we have adopted the following convenient convention: if an equality between two different homology groups is written without the $*$'s, this should be understood to mean that the equality is true up to a grading shift.
\end{rem}

In the next section we will use the following additional results about local Floer homology. Both of them are very standard, although for the convenience of the reader we provide sketches of the proofs. 

\begin{lem}
\label{lem:master_invariance}
Suppose $\{ \omega_s = \omega + d \zeta_s \}_{s \in [0,1]}$ is an exact deformation of symplectic forms. Supppose $ \{ L_s \}_{s \in [0,1]}$ is a family of Hamiltonians, and denote by $\phi_{L_s ;\omega_s}^t : W \to W$ the flow of the symplectic gradient $X_{L_s ;\omega_s}$ with respect to the symplectic form $\omega_s$. Suppose $P \subset \bigcap_{ s \in [0,1]} \mathrm{Fix}( \phi_{L_s ; \omega_s}^1)$ is a common set of fixed points, which is uniformly isolated in the sense that there exists a subset $N \subset W \times S^1$ such that for each $s \in [0,1]$, $N$ is an isolating neighbourhood of the set $\Gamma_{L_s ; \omega_s}(P) \subset \mathcal{P}_1(L; \omega_s)$. Assume in addition that $\Gamma_{L_s; \omega_s}(P)$  is action-constant subset (with the same constant for each $s$).   Then 
\[
\HF^{\loc}_{\omega_0}(\phi_{L_0 ;\omega_0}^1 , P) \cong \HF^{\loc}_{\omega_1}(\phi_{L_1 ; \omega_1}^1 , P),
\]
where $\HF^{\loc}_{\omega_0}$ denotes the local Floer homology defined using the symplectic form $\omega_0$ etc.
\end{lem}

\begin{proof}[Proof. (Sketch)]
By using an adiabatic argument, it suffices to prove the following statement: assume $P$ is an isolated set of fixed points for $ \phi_{L;\omega}^1$, and assume $\Gamma_L(P) $ is an action-constant subset of $\mathcal{P}_1(L ,\omega)$. Fix a family $J = \{ J_t \}_{t \in S^1}$ of $\omega$-compatible almost complex structures, and fix two isolating neighbourhoods $N_1 \subset N_2 \subset W \times S^1$ of $\Gamma_L(P)  $. Then there exists a constant $\delta > 0 $ with the property that if we are given:
\begin{itemize}
\item A family $\{ \zeta_s \}_{s \in \R} \subset \Omega^1 (W)$ of 1-forms, 
\item A family $\{ L_s \}_{s \in \R} \subset C^{ \infty}(W \times S^1)$ of smooth functions,
\item A family  $\{ J_{s,t} \}_{ s \in \R} $ of families of $( \omega + d \zeta_s)$-compatible almost complex structures 
\end{itemize}
such that all families are independent of $s$ for $s \notin [0,1]$ and such that
\begin{equation}
\label{eq:send_to_zero}
\| L_s -L \|_{C^2} + \| J_s - J \|_{C^1} + \| \partial_s L_s \|_{C^2} + \| \partial_s J_s \|_{C^1} + \| \partial_s \zeta_s \|_{L^\infty} < \delta,  
\end{equation}
then any finite energy solution $u : \R \times S^1 \to W$ of the $s$-dependent problem $\partial_s u + J_{s,t}(u)( \partial_t u - X_{L_s  ;\omega_s}(u)) = 0$ with $u( \R \times S^1) \subset N_2$ actually satisfies $u( \R \times S^1) \subset N_1$. The argument is again by contradiction, and the key point is that by making the left-hand side of \eqref{eq:send_to_zero} arbitarily small, one can also make the energy of such a finite energy solution arbitrarily small. The only difference between this statement and the argument sketched \vpageref{page:explanation} is the fact that the symplectic form now additionally depends on $s$, and this gives rise to an extra potentially problematic term in the energy computation. Luckily, this extra term turns out not to be problematic at all, since we have the following estimate, where the constant $C$ changes from line to line:
\[
\begin{split}
\left| \int_{- \infty}^{+ \infty} \frac{\partial}{\partial s} \left( \int_{ [0,1] \times S^1} \bar{u}(s, \cdot)^* ( \omega + d \zeta_s)  \right)  \,ds \right|  & \le C \delta \left( \int_0^1 \left( \int_{S^1} | \partial_t u(s,t) | \,dt  \right)^2 +1 \right) \,ds \\
& \le C \delta \left( \int_{- \infty}^{+ \infty} \| \partial_t u\|_{J_{s,t}}^2 \,ds\,dt  + 1\right) \\
& \le  C \delta \left( \int_{- \infty}^{+ \infty} \| J_{s,t}(u) \partial_s u + X_{L_s ; \omega_s} (u) \|^2_{ J_{s,t}} \,ds\,dt  + 1\right) \\
& \le  C \delta \left( \E( u) + 1 \right).
\end{split}
\]
The upshot is that \eqref{eq:send_to_zero} implies that for any finite energy solution $u : \R \times S^1 \to W$ of the $s$-dependent problem $\partial_s u + J_{s,t}(u)( \partial_t u - X_{L_s ;\omega_s}(u)) = 0$ with $u( \R \times S^1) \subset N_2$ one gets an estimate of the form
\begin{equation}
\label{eq:energy_on_both_sides}
\E (u) \le \A_{L_0 }(u (- \infty)) - \A_{L_1}(u ( + \infty)) + C \delta ( 1  + \E(u)).
\end{equation}
Now the argument proceeds as in the one sketched \vpageref{page:explanation} (since we can carry the energy term on the right-hand side of \eqref{eq:energy_on_both_sides} to the left, provided $C \delta < 1$).
\end{proof} 

\begin{lem}
\label{lem:kunneth}
Suppose $(W, \omega) = (W_1 \times W_2, \omega_1 \oplus \omega_2)$ is a product symplectic manifold and $\phi = ( \phi_1, \phi_2)$ is a product Hamiltonian diffeomorphism. Assume $P_1$ is an isolated subset of fixed points for $\phi_1$ and $P_2$ is an isolated subset of fixed points for $\phi_2$. Then the K\"unneth formula holds:
\[
\HF^{\loc}_ \omega( \phi, P) \cong \HF^{\loc}_{\omega_1}(\phi_1, P_1) \otimes \HF^{ \loc}_{\omega_2}(\phi_2 , P_2).
\]
\end{lem}

This lemma is stated in \cite[Lemma 2.10]{McLean2012} (see also \cite[Property (LF4), Section 3.2]{GinzburgGurel2010}. The point is that we can choose a split perturbation $F= (F_1,F_2)$ to define the local Floer complex, and also work with a split product $J = (J_1, J_2)$ of almost complex structures. Then the statment is essentially obvious.

\subsection{Local Floer homology of invariant Reeb orbits}
\label{subsec:local_FH_of_invariant_Reeb_orbits}

Let us now consider again the situation we are primarily interested in. We adopt the notation introduced in Sections \ref{sec:preliminaries} and \ref{sec:Floer_homology}. Let $\widehat{\varphi} = \{ \varphi_t \}_{0 \le t \le 1}$ denote an admissible path of contactomorphisms, and assume that the time-1 map $\varphi = \varphi_1$ is a \textbf{strict} contactomorphism. Thus if we write $\varphi_t^* \alpha = \rho_t \alpha$ for a positive family $\rho_t : \Sigma \to (0, +\infty)$ of smooth functions, then $\rho_1 \equiv 1$. We emphasise that we do \textbf{not} assume that the entire path $\widehat{\varphi}$ is strict, only that the terminal map is. Thus $\rho_t $ is not necessarily identically equal to $1$ for $0 < t < 1$.  Fix some $c $ strictly greater than the constant $C( \widehat{\varphi})$ defined in \eqref{eq:C_phi}, and denote by $\widetilde{L}^c :W  \times S^1 \to \R$ the function defined in \eqref{eq:def_of_Lc}. \\

As explained in Example \ref{ex:strict}, the assumption that the terminal map $\varphi$ is a strict contactomorphism implies that critical points of the corresponding action function $\A_{\widetilde{L}^c}$ are never isolated. Indeed, suppose $\gamma : \R \to \Sigma$ is a $\varphi$-invariant Reeb orbit, and set $p = \gamma(0)$, so that $\gamma(s) = \varphi_R^s(p)$, and there exists $\eta \in \R$ such that $\varphi(\gamma(s)) = \gamma(s - \eta)$ for all $s \in \R$.

Let us assume that $\gamma$ is isolated in the set of all invariant Reeb orbits for $\varphi$. Then there are two possible pictures, depending as to whether the orbit $\gamma$ is closed or not. Let us first cover the case where $\gamma$ is \textbf{not} a closed orbit.  Then (cf. Lemma \ref{lem:relating_spectrums}) we have a component $\Gamma \cong \R \subset \mathcal{P}_1(\widetilde{L}^c) = \crit \A_{\widetilde{L}^c}$:
\begin{equation}
\label{eq:component_from_invariant_orbit}
\Gamma = \left\{ (z_s ,  \eta , 0) \subset \Lambda (W) \mid  s \in \R \right\},
\end{equation}
where $z_s (t) = (x_s (t), r_s (t)) \in S \Sigma$, and
\begin{equation}
\label{eq:the_loop_x_theta}
x_s (t )= \begin{dcases}
\gamma \left( s  +  \eta  \int_0^t \kappa(a) \,da \right), & 0 \le t \le 1/2 ,\\
\varphi_t(\gamma(s + \eta)), & 1/2 \le t   \le 1.
\end{dcases},
\end{equation}
and 
\begin{equation}
\label{eq:r_theta}
r_s (t) = \begin{cases}
1, & 0 \le t \le 1/2, \\
\rho_t(\gamma(s + \eta)), & 1/2 \le t \le 1.
\end{cases}
\end{equation}
Note that with this convention one has $z_s(0) = (\gamma(s), 1)$. \\

Things get particularly interesting when the invariant Reeb orbit $\gamma : \R \to \Sigma$ is \textbf{closed}, say of minimal period $T > 0$.  In this case the component $\Gamma$ described above is diffeomorphic to the circle rather than the real line. Moreover by iterating the Reeb orbit we get a family $\{ \Gamma_k \}_{k \in \Z}$ of components of $\mathcal{P}_1(\widetilde{L}^c)$. More precisely, for each integer $k \in \Z$ there is a component $\Gamma_k \cong S^1 \subset \mathcal{P}_1(\widetilde{L}^c)$, which corresponds to travelling round the loop $\gamma$ $k$ times (where negative $k$ should be interpreted as going $k$ times backwards along $\gamma$), and then following the loop described above. Explicitly, 
\begin{equation}
\label{eq:Gamma_k}
\Gamma_k = \left\{ (z_{s,k}, \eta + kT  , 0) \subset \Lambda (W) \mid s \in \R / T \Z \right\},
\end{equation}
where as before $z_{s ,k}(t) = (x_{s, k}(t), r_{s ,k}(t)) \in S \Sigma$, and 
\[
x_{s,k}(t) = \begin{dcases}
\gamma \left( s + (\eta +  kT) \int_0^t \kappa(a)\,da \right), & 0 \le t \le 1/2,\\
\varphi_t( \gamma( s + \eta )), & 1/2 \le t \le 1.
\end{dcases}
\]
and $r_{s, k} = r_s$ is the loop 
\[
r_{s,k}(t) = \begin{cases}
1, & 0 \le t \le 1/2, \\
\rho_t(\gamma(s + \eta)), & 1/2 \le t \le 1.
\end{cases}
\]
It is important to understand that the loop $x_{s, k}$ is \textbf{not} simply the $k$th iterate of the loop $x_{s,1}$.  \\

Let $D \subset \Sigma$ denote an embedded hypersurface which is transverse to $\gamma$ at $p := \gamma(0)$. Then the \textbf{Poincar\'e  map} $P \colon U \to D$ is well defined, where $U \subset D$ is a small neighbourhood of $p$. Explicitly, 
\begin{equation}
\label{eq:first_return_P_gamma}
P(x) = \varphi_R^{s(x)}(x), \qquad s(x) := \inf \{s > 0 \mid \varphi_R^s(x) \in D \}.
 \end{equation}
Since $\varphi$ commutes with the Reeb flow, the hypersurface $\varphi^{-1}(D)$ is again transverse to $\gamma$ at the point $\gamma( \eta)$. Thus there is a well defined map $A \colon U \to D$
\begin{equation}
\label{eq:the_varphi_poincare_map}
A(x) = \varphi \left( \varphi_R^{s_{\varphi}(x)}(x) \right),
\end{equation}
where $s_{\varphi} : D \to \R$ is defined by
\[
s_{\varphi}(x) := \inf \left\{ s > 0 \mid  \varphi_R^s(x) \in  \varphi^{-1}(D) \right\}.
\]
Both $P$ and $A$ fix the point $p$. It is well known that $P$ can be seen as a Hamiltonian diffeomorphism of the embedded hypersurface $D$. In fact, the same is true of $A$; we will prove this in Lemma \ref{lem:its_a_ham} below. For now let us just note the following observation.

\begin{lem}
The germs of $P$ and $A$ commute. That is, on a suitably small neighbourhood $U$ of $p$, both $P A$ and $A P$ are defined and equal. 
\end{lem}
\begin{proof}
Since $\varphi$ commutes with the Reeb flow, both $P A$ and $A P $ are maps $U \to D$ of the form $x \mapsto \varphi( \varphi_R^{g_i(x)}(x))$ for appropriate functions $g_1$ and $g_2$. Explicitly, 
\[
g_1(x) = s( A(x)) + s_{ \varphi}(x), \qquad g_2(x) = s_{\varphi}(P(x)) + s(x).
\]
Fix $x \in D$ and consider for $i=1,2$ the path $\delta_i : [0 ,1 ] \to \Sigma$ given by $\delta_i (r) = \varphi_R^{r g_i(x)}(x)$. This path has the property that $\delta_i(0) \in D$, $\varphi(\delta_i(1)) \in D$, and there exists a unique $c_i \in (0,1)$ such that $\delta_i(c_i) \in D$. Up to shrinking $U$, there is at most one such path, and hence in particular $\delta_1(1) = \delta_2(1)$. Thus $g_1 \equiv g_2$ as claimed.
\end{proof}

The assumption that $\gamma$  is isolated in the set of all invariant Reeb orbits for $\varphi$ implies that $p \in D$ is an isolated fixed point of $ A P^k $ for each $k \in \Z$. In particular, the local Floer homology groups $\HF^{\loc}(A P^k,p)$ are well defined (here we should really write $\HF^{\loc}(A P^k,\{ p\})$ to be consistent with our previous notation). The next result should be compared with \cite[Lemma 3.4]{McLean2012} and \cite[Proposition 6.1]{HryniewiczMacarini2012}.
\begin{prop}
\label{prop:reduction_to_poincare_map}
One has
\begin{equation}
\label{eq:the_iso_between_the_two_local_FHs}
\HF^{\loc}( \widetilde{L}^c , \Gamma_k) \cong  \HF^{\loc}( A P^k ,p) \otimes \mathrm{H}(S^1 ; \Z_2).
\end{equation}
\end{prop}
\begin{rem}
Recall from Remark \ref{rem:no_star_means_no_grading} that the lack of $*$'s in \eqref{eq:the_iso_between_the_two_local_FHs} should be understood to mean that the isomorphism is \textbf{not} grading preserving. Indeed, the local Floer homology groups $\HF^{\loc}(A P^k, p)$ are themselves only defined up to a shift in grading.
\end{rem}

In order to prove Proposition \ref{prop:reduction_to_poincare_map}, we will localise the problem inside a tubular neighbourhood of $\gamma$.  This is done by the following lemma, whose proof can be found for instance in \cite[Lemma 5.2]{HryniewiczMacarini2012}.

\begin{lem}
\label{lem:local_model}
Let $\gamma : \R  \to \Sigma$ denote a closed Reeb orbit of minimal period $T>0$. There exists a tubular neighbourhood $N \cong B \times S^1 $ of $\gamma( \R)$, where $B \subset \R^{2n-2}$ is a small ball centred about the origin, such that, if we use coordinates $(q,t) \in B \times S^1$, one has
\begin{enumerate}
  \item $\alpha|_N = \theta - K d t$, where $\theta$ is the standard contact form on $B \subset \R^{2n-2}$ and $K : B \times S^1 \to \R$ is a smooth Hamiltonian such that $K(0,t)  \equiv - T$ and $dK_t(0) \equiv 0$, and finally such that
 \begin{equation}
\label{eq:never_zero} 
 \theta_q (X^B_{K}(q,t))-  K(q,t)     \ne 0, \text{ for all } (q,t ) \in B \times S^1,
  \end{equation}
  where $X^B_K$ denotes the symplectic gradient of $K_t$ with respect to $d \theta$.
\item $\gamma(s) = (0,s/T)$ for all $s \in \R$, where $s/T$ should be read modulo $T$.
  \end{enumerate}
\end{lem}

We will also need the following trivial result.

\begin{lem}
\label{lem:reparametrise}
Let $Y \in \mathrm{Vect}(B)$ denote a smooth vector field on a manifold $B$, and let $b \in C^{\infty}(B)$ denote a smooth function. Let $X : = bY$. Then the flows $\varphi_Y^s$ and $\varphi_X^s$ of $Y$ and $X$ are related by
\[
\varphi_X^s (x) = \varphi_Y^{\beta(s,x)}(x), \qquad \mathrm{where} \quad \beta(s,x) = \int_0^s b( \varphi_X^r (x  ))\,dr.
\]
\end{lem}
One easily checks that the Reeb vector field is given by 
\begin{equation}
\label{eq:reeb_vector_field_in_these_coords}
R(q,t) = \frac{1}{ \theta_q (X^B_{K}(q,t))-  K(q,t)  } \left( \partial_t +  X^B_K(q,t) \right)   
\end{equation}
To avoid confusion we shall denote the flow of $X^B_K$ by $f_K^s$ (instead of say, $\varphi_K^s$). Thus from Lemma \ref{lem:reparametrise}, the Reeb flow $\varphi_R^s : B \times S^1 \to B \times S^1$ is given by
\begin{equation}
\label{eq:reeb_flow}
\varphi_R^s( q, t) = ( f_K^{\beta(s,q,t)}(q),  t+ \beta(s , q,t)),
\end{equation}
where $\beta $ is the function defined in Lemma \ref{lem:reparametrise} associated to 
\begin{equation}
\label{eq:reparam_factor}
b(q,t) : = \frac{1}{ \theta_q (X^B_{K}(q,t))-  K(q,t)  }.
\end{equation}
Using the notation of Lemma \ref{lem:local_model}, let us take 
\[
D :=  B \times \{ 0\}  \cong B
\] 
as an embedded hypersurface transverse to $\gamma$ at $\gamma(0)$. Then we claim that the Poincar\'e map $P: B \to B$ from \eqref{eq:first_return_P_gamma} is given by 
\begin{equation}
\label{eq:poincare_map_in_these_coords}
P(q)=  f_K(q),
\end{equation}
where $f_K = f_K^1$. Indeed, by definition, the map $P$ satisfies
\[
(P(q),0) = \varphi_{R}^{s_q}(q,0),
\]
where $s_q$ is the smallest positive number such that $\beta(s_q , q ,0) =1$, and hence \eqref{eq:poincare_map_in_these_coords} follows directly from \eqref{eq:reeb_flow}.

The map $\varphi$ may not necessarily preserve $B \times S^1$, but we can choose a smaller ball $B' \subset B$ that still contains the origin such that $\varphi(B' \times S^1) \subset  B \times S^1$. Let us write 
\begin{equation}
\label{eq:varphi=xi,a}
\varphi ( q,t) = ( \xi (q, t), a(q,t))
\end{equation}
in these coordinates (since we are only concerned with the germ of $\varphi$ near $\{0 \} \times S^1$, from now on we will abuse notation and think of $B' = B$). We now examine the map $A$ from \eqref{eq:the_varphi_poincare_map} in these coordinates. Recall to define $A$ we start with a point $(q,0) \in B \times \{0\} $. Then we flow along the Reeb flow to the first $s > 0$ such that the point $\varphi_R^s(0,x)$ has the property that $\varphi(\varphi_R^s(0,x)) \in B \times \{0\}$. In other words, we require that $a( \varphi_R^s(q,0)) = 0$. For this choice of $s$ we define $A(q) := \xi( \varphi_R^s(q,0))$. Explicitly, this means
\begin{equation}
\label{eq:varphi_Poincare_map}
A(q) = \xi ( f_K^{w(q)}(q), w(q)),
\end{equation}
where $w(q)$ is defined to be the smallest positive number such that 
\begin{equation}
\label{eq:a_component_of_varphi_Poincare}
a(f_K^{w(q)}(q),w(q)) = 0.
\end{equation}
Note that one can equivalently write $w(q) = \beta(s(q), q,0)$ for some function $s(q)$.  For later use let us note that as $\varphi$ commutes with $\varphi_R^s$, we have
\[
\begin{split}
  (A(q), 0) &  = \varphi( \varphi_R^s(q ,0)) \\
  & = \varphi_R^s(\varphi(q,0)) \\
  & = (f_K^{\beta(s,\varphi(q,0))}(\xi(q,0)), a(q,0)+ \beta(s, \varphi(q,0) ),
  \end{split}
  \]
  and hence we can alternatively define $A$ by
  \begin{equation}
    \label{eq:another_def_of_A}
    A(q) = \big(f_K^{a(q,0)}\big)^{-1}(\xi (q, 0)).
  \end{equation}
Let us now define another function $F : B \times [0,1] \to \R$ via the formula
\begin{equation}
\label{eq:the_function_F}
F(q,t) : = \int_0^t \left(  \theta_q (X^B_{K}(q,r))-  K(q,r)  \right) \circ f_K^r (q) \,dr.
\end{equation}
One easily checks that
\begin{equation}
\label{eq:relation_K_F}
(f_K^t)^* \theta - \theta = dF_t.
\end{equation}
The next result is elementary, and can be proved in a variety of ways (its statement should be intepreted as a sanity check!). The proof we give is a direct computation. 
\begin{lem}
\label{lem:its_a_ham}
The Poincar\'e map $A$  is a Hamiltonian diffeomorphism of the ball $B$. In fact,
\[
A^* \theta - \theta = dG,
\]
where $G$ is the autonomous Hamiltonian $G(q) : = F(q,w(q))$ and $w$ was defined in \eqref{eq:a_component_of_varphi_Poincare}.
\end{lem}

To prove the Lemma we use the fact that $\varphi$ preserves $\alpha$. Since $\varphi^* \alpha = \alpha$ and $\alpha = \theta - K dt$, we obtain the following two formulae relating the maps $\xi: B \times S^1 \to B$ and $a :B \times  S^1 \to B$. Given $(q,t) \in B \times S^1 $ and $(  \hat{q}, \hat{ t}) \in T_{(q,t)}(B \times S^1 )$, write $\hat{t} = c \partial_t$ for some $c \in \R$. Then using
\[
(\varphi^*\alpha)_{(q,t)}( \hat{q}, 0) = \alpha_{(q,t)}( \hat{q},0), 
\]
we obtain
\begin{equation}
\label{eq:xi_relation}
\theta_q (\hat{q}) = \theta_{\xi(q,t)} (D_1 \xi(q,t)[ \hat{q}]) - K (\varphi(q,t)) D_1 a(q,t) [ \hat{q}],
\end{equation}
and similarly from 
\[
(\varphi^*\alpha)_{(q,t)}( 0, \hat{t}) = \alpha_{(q,t)}(0 , \hat{t}), 
\]
we obtain
\begin{equation}
\label{eq:a_relation}
-K(q,t) c = \theta_{\xi(q,x)}(D_2 \xi (q,t)[ \hat{t}] ) - K( \varphi(q,t))D_2 a (q,t) c.
\end{equation}

\begin{proof}[Proof of Lemma \ref{lem:its_a_ham}]
Fix $(\hat{q}, \hat{t} = c \partial_t) \in T_{(q,t)}(B \times S^1)$. Write
\[
\Psi(q) = (f^{w(q)}_K(q), w(q)),
\]
so that $A(q)  = \xi( \Psi(q))$. Note that
\begin{equation}
\label{eq:DPsi}
D \Psi(q)[\hat{q}] = \left( Df_K^{w(q)}(q) [\hat{q}] + dw(q)[ \hat{q}]X_K^B(\Psi(q)), dw(q)[\hat{q}] \right) 
\end{equation}
We compute
\[
\begin{split}
(A^*\theta)_q( \hat{q}) & = \theta_{A(q)}( D A(q) [\hat{q}]) \\
& =  \theta_{A(q)}( D \xi( \Psi(q)) \circ D \Psi (q)[\hat{q}])\\
& =    \underset{ := \mathrm{\,(I)}}{\underbrace{\theta_{A(q)} \left[ D_1 \xi (\Psi(q))\left(Df_K^{w(q)}(q) [\hat{q}] + dw(q)[ \hat{q}]X_K^B(\Psi(q)) \right)  \right]}} \\
&  +  \underset{ := \mathrm{\,(II)}}{\underbrace{\theta_{A(q)} \left( D_2\xi(\Psi(q))dw(q)[\hat{q}] \right)}}
\end{split}
\]
Now using 
 \eqref{eq:xi_relation}, we see that 
\begin{equation}
\label{eq:computing_A}
\begin{split}
\mathrm{(I)} & =  \theta_{f_K^{w(q)}(q)} \left( Df_K^{w(q)}(q) [\hat{q}] + dw(q)[ \hat{q}]X_K^B(\Psi(q)) \right) \\
& +K(\varphi (\Psi(q))) D_1 a(\Psi(q)) \left( Df_K^{w(q)}(q) [\hat{q}] + dw(q)[ \hat{q}]X_K^B(\Psi(q))\right) .
\end{split}
\end{equation}
and similarly using
\eqref{eq:a_relation} we see that
\begin{equation}
\label{eq:computing_B}
 \mathrm{(II)} =K (\varphi(\Psi(q)) D_2 a (\Psi(q)) dw(q)[ \hat{q}]  -  K(\Psi(q)) dw(q)[ \hat{q}].
 \end{equation}
 
To proceed further we differentiate the equation
\[
a(\Psi(q)) = a(f_K^{w(q)}(q),w(q))=0
\]
with respect to $q$, which gives us
\begin{equation}
\label{eq:a=0}
D_1 a (\Psi(q)) \left[ Df_K^{w(q)}(q) [\hat{q}] + dw(q)[ \hat{q}]X_K^B(\Psi(q))  \right] + D_2 a (\Psi(q)) dw(q)[ \hat{q}] =0.
\end{equation}
Thus using \eqref{eq:a=0}, we see we can rewrite \eqref{eq:computing_A} as 
\begin{equation}
\label{eq:new_A}
\begin{split}
\mathrm{(I)} & =   \theta_{f_K^{w(q)}(q)} \left( Df_K^{w(q)}(q) [\hat{q}] + dw(q)[ \hat{q}]X_K^B(\Psi(q)) \right) \\
& - K(\varphi (\Psi(q))) D_2 a (\Psi(q)) dw(q)[ \hat{q}] .
\end{split}
\end{equation}
Now we are happy, since combining \eqref{eq:computing_B} and \eqref{eq:new_A}, we see that terms cancel, and
\[
\begin{split}
(A^* \theta)_q ( \hat{q}) &= \mathrm{(I)} + \mathrm{(II)} \\
& = \theta_{f_K^{w(q)}(q)} \left( Df_K^{w(q)}(q) [\hat{q}] + dw(q)[ \hat{q}]X_K^B(\Psi(q)) \right)  -  K(\Psi(q)) dw(q)[ \hat{q}] \\
& = dw(q)[ \hat{q}] (   \theta_q(X_K^B (\Psi(q))) - K (\Psi(q)) ) + \theta_q ( \hat{q}) + dF_{w(q)}(q)[\hat{q}],
\end{split}
\]
where the last line used \eqref{eq:relation_K_F}. But now we are done, since from \eqref{eq:the_function_F}, we see that
\[
dG(q)[ \hat{q}] = \frac{d}{dq} F(q, w(q)) [ \hat{q}] =  dF_{w(q)}(q)[\hat{q}] + dw(q)[ \hat{q}] (   \theta_q( X_K^B (\Psi(q))) - K(\Psi(q)) ) ,
\]
and hence $A^*\theta - \theta = dG$ as claimed.
\end{proof} 

There is another natural way the map $A$ shows up. Let us consider the \textbf{mapping cylinder} of the Poincar\'e map $P$:
\[
  E := B \times \R / \sim, \qquad (q,t) \sim (P(q), t -1).
\]  
The mapping cylinder $E$ is a trivial fibre bundle over $S^1$, with trivialisation 
\[
T : B \times S^1  \to E, \qquad T(q,t) := \left( (f_K^t)^{-1}(q ), t \right)   
\]
Let us denote by $\hat{ \varphi}_R^s : E \to E$ induced map, defined by $ \hat{\varphi}_R^s \circ T = T \circ \varphi_R^s $. Explicitly, 
\[
  \hat{\varphi}_R^s(q,t) = (q, t + \beta(s,f_K^t(q),t)),
\]
where $\beta$ is as \eqref{eq:reeb_vector_field_in_these_coords}. 
\begin{lem}
\label{lem:mapping_cylinder}
If $\hat{\varphi}: E \to E$ denotes the map induced by $\varphi$, then the first component of $\hat{\varphi}$ is given by $A$:
\[
  \hat{\varphi}(q,t) = (A(q), a(f_K^t(q),t)).
\]
Moreover one can write
\[
  a(f_K^t(q), t) =  \beta(h(q,t), f_K^{a(q,0)}(A(q)), a(q,0)),
\]
where $h(q,t)$ is the smallest positive number such that $\beta(h(q,t),q,0)=t$.
\end{lem}
\begin{proof}
By definition, writing $\varphi = ( \xi ,a)$ as before, the map $\hat{\varphi}$ is given by
\begin{equation}
\label{eq:varphi_messy}
  \hat{\varphi}(q,t) = \left(  \Big(f_K^{a(f_K^t(q),t)}\Big)^{-1}(\xi(f_K^t(q),t)), a(f_K^t(q),t)) \right) 
\end{equation}
Since $\varphi$ commutes with $\varphi_R^s$, so do $\hat{\varphi}$ and $\hat{\varphi}_R^s$, and we see that
\[
\begin{split}
\hat{\varphi}(q, \beta(h,q,0)) & = \hat{\varphi}(\hat{\varphi}_R^h(q,0)) \\
   &  = \hat{\varphi}_R^h( \hat{\varphi}(q,0)) \\
  & = \hat{\varphi}_R^h \left( \Big(f_K^{a(q,0)}\Big)^{-1}(\xi(q,0)), a(q,0))  \right) \\
 & = (A(q), a(q,0) + \beta(h ,  f_K^{a(q,0)}(A(q)), a(q,0))),
\end{split}
\]
where the last line used \eqref{eq:another_def_of_A}. Thus we can alternatively write \eqref{eq:varphi_messy} as
\begin{equation}
  \label{eq:varphi_less_messy}
  \hat{\varphi}(q,t) = (A(q), a(q,0) + \beta(h(q,t), f_K^{a(q,0)}(A(q)), a(q,0)),
\end{equation}
where $h(q,t)$ is the smallest positive number such that $\beta(h(q,t),q,0)=t$. 
\end{proof}

We will now get started on the proof of Proposition \ref{prop:reduction_to_poincare_map}.

\begin{proof}[Proof of Proposition \ref{prop:reduction_to_poincare_map}]
We have by definition that
\[
\HF^{\loc}(\widetilde{L}^c , \Gamma_k) = \HF^{\loc}( \Phi_{\widetilde{L}^c}^1 , P(\Gamma_k)).
\]
Since $\HF^{\loc}( \Phi_{\widetilde{L}^c}^1 , P(\Gamma_k))$ depends only on the germ of $\Phi_{\widetilde{L}^c}^1$ on a neighbourhood of $P(\Gamma_k)$, we can simplify things slightly and drop the various cutoff functions. Thus \textbf{for the rest of this section only} let us redefine
\[
\widetilde{H} : S \Sigma \times T^*\R \to \R, \qquad \widetilde{H}(x,r,\tau, \sigma) := \tau(r-1) + \frac{1}{2} \sigma^2,
\]
and write simply $\Phi$ for the Hamiltonian diffeomorphism of $S \Sigma \times T^*\R$ given by
\[
\Phi( x, r, \tau ,\sigma ) = ( \varphi(x), r, \tau , \sigma).
\]
Then 
\[
\HF^{\loc}( \Phi_{\widetilde{L}^c}^1 , P(\Gamma_k)) = \HF^{\loc}( \Phi \circ \Phi_{\widetilde{H}}^1 , P(\Gamma_k)).
\]
Set $S_k = P(\Gamma_k) \subset S \Sigma \times T^*\R$.  An isolating neighbourhood for $S_k$ can be taken for instance as
\[
\widetilde{N}_k = E \times (\tfrac{1}{2}, \tfrac{3}{2}) \times ( \eta + kT - \varepsilon, \eta + kT + \varepsilon) \times ( -\tfrac{1}{2}, \tfrac{1}{2}),
\]
where as before $E \cong B \times S^1$ is the mapping cylinder of $P$. The Hamiltonian diffeomorphism $\Phi_{\widetilde{H}} = \Phi_{\widetilde{H}}^1$ is given in a neighbourhood of $S_k \subset \widetilde{N}_k$ by
\[
\Phi_{\widetilde{H}}^s ( q, t, r , \tau ,\sigma) = \left( q,  t+ \beta(\tau s , f_K^t(q),t),  r, \tau - s \sigma - \tfrac{1}{2} s^2 (r-1 ), \sigma +s (r- 1)\right), 
\]
and thus the composition $\Phi \circ \Phi_{\widetilde{H}}^1$ is given by
\[
\Phi \circ \Phi_{\widetilde{H}}^1(q ,t ,r, \tau, \sigma) = \left(A(q), a\left(f_K^{t + \beta(\tau ,f_K^t(q),t)}(q),t + \beta (\tau, f_K^t(q), t)\right), r, \tau -\sigma - \tfrac{1}{2}(r-1), \sigma + r -1\right).
\]
Note by the definition of $E$ we can alternatively write this as 
\[
\Phi \circ \Phi_{\widetilde{H}}^1(q ,t ,r, \tau, \sigma) = \left(AP^k(q), a\left(f_K^{t + \beta(\tau ,f_K^t(q),t)}(q),t + \beta (\tau, f_K^t(q), t)\right) - k, r, \tau -\sigma - \tfrac{1}{2}(r-1), \sigma + r -1\right).
\]
To complete the proof of Proposition \ref{prop:reduction_to_poincare_map}, we must show that
\begin{equation}
\label{eq:reduction_in_local_coords}
 \HF^{\loc}( \Phi \circ \Phi_{\widetilde{H}}^1 , S_k) =  \HF^{\loc}( A  P^k , 0) \otimes \mathrm{H}^{\mathrm{sing}}(S^1 ; \Z_2)
\end{equation}
We will construct a family $ \{ \Psi_ \lambda \}_{ \lambda \in [ 0,1 ] }$ of (germs of)  diffeomorphisms of $\widetilde{N}_k$ such that 
\[
\Psi_0 = \Phi \circ \Phi_{\widetilde{H}}^1\] 
and such that
\[
\Psi_1( q, t,  r, \tau ,\sigma) = \left( A P^k (q) , t + \frac{1}{T}( \tau - kT  - \eta) , r, \tau - \sigma - \frac{1}{2} (r - 1),\sigma + r - 1 \right),
\]
and such that $S_k$ is a uniformly isolated set of fixed points for $\Psi_ \lambda$ for all $\lambda \in [0,1]$. Moreover at the same time we will construct a family of symplectic forms 
\[
\omega_ \lambda = \omega + d \zeta_ \lambda,
\]
where $\zeta_ \lambda \in \Omega^1(\widetilde{N}_k)$ are one-forms vanishing on $S_k$ satisfying 
\begin{equation}
\label{eq:zeta_req}
 \zeta_0 = 0, \qquad \zeta_1 = (1-r)\theta + (T + K)r dt, 
 \end{equation} 
so that in particular one has
\[
\omega_1 =  d \theta + T dr \wedge d t - d \sigma \wedge d \tau.
\]
In addition we will require that
\[
\Psi_ \lambda \in \mathrm{Ham}(\widetilde{N}_k , \omega_ \lambda)
\]
for all $\lambda \in [0,1]$.\\

Let us assume for the moment we have constructed such families $(\Psi_ \lambda , \zeta_ \lambda)$. It follows from Lemma \ref{lem:master_invariance} that
\[
\HF^{\loc}_{\omega_0}( \Psi_0 , S_k ) \cong \HF^{\loc}_{\omega_1}( \Psi_1 , S_k), 
\]
Next, we claim that
\[
\HF^{\loc}_{\omega_1}( \Psi_1 , S_k ) \cong \HF^{\loc}_{d \theta}( A  P^k , 0) \otimes \mathrm{H}(S^1 ; \Z_2).
\]
Indeed, if we write
\[
\Psi_1(q,t,r, \tau, \sigma) = ( AP^k (q), \Psi_1' (t, r, \tau, \sigma)),
\]
and $S_k = \{0 \} \times S_k' $ and $\omega_1 = d \theta + \omega_1'$, then by the K\"unneth formula in local Floer homology (Lemma \ref{lem:kunneth}), we have
\[
\HF^{\loc}_{\omega_1}( \Psi_1 , S_k ) \cong \HF^{\loc}_{d \theta}(AP^k ,0) \otimes \HF^{\loc}_{\omega_1'}(\Psi_1' , S_k' ).
\]
Finally, it is easy to check that $S_k '$ is a Morse-Bott set of fixed points for $\Psi_1'$, and thus $\HF^{\loc}_{\omega_1'}(\Psi_1' , S_k'  ) \cong \mathrm{H}^{\mathrm{sing}}(S_k ' ; \Z_2)$ cf. Example \ref{ex:morse_bott_component}.

It remains therefore to construct the families $\Psi_ \lambda$ and $\zeta_ \lambda$. There are several ways this can be done, and the argument is essentially contained in Section 3 of \cite{McLean2012}, compare also Section 6 of \cite{HryniewiczMacarini2012}. The idea is to homotope the function $b(q,t)$ from \eqref{eq:reparam_factor} via a family $\{ b_{\lambda} \}_{ \lambda \in [0,1]}$ of functions so that $b_0 = b$, $b_{\lambda}(0,t) = 1/T$ for all $\lambda$, and such that $b_1(q,t) \equiv 1/T$. An explicit, but unenlightening computation shows that for an appropriate choice of $b_{\lambda}$, we can choose a family $\{ \zeta_{\lambda} \}_{\lambda \in [0,1]}$ of  one-forms satisfying \eqref{eq:zeta_req} and a a family $ \{ \widetilde{H}_{\lambda} \}_{ \lambda \in [0,1]}$ of functions satisfying
\[
  \widetilde{H}_0 = \widetilde{H}, \qquad \widetilde{H}_1 = \widetilde{H}+ K,
\]
 such that the symplectic gradient $X_{\widetilde{H}_{\lambda} ; \omega_{\lambda}}$ of $\widetilde{H}_{\lambda}$ with respect to $\omega_{\lambda} := \omega + d \zeta_{\lambda}$ is of the form
\[
  X_{\widetilde{H}_{\lambda} ; \omega_{\lambda}}(q,t,r, \tau ,\sigma) = \tau b_{\lambda}(q,t) \left(  \partial_t + X_K^B(q,t) \right) - \sigma \partial_\tau+ (r-1)\partial_{\sigma}.  
\]
This means that the flow $\Phi^s_{\widetilde{H}_{\lambda}, \omega_{\lambda}} : \widetilde{N}_k \to \widetilde{N}_k$ is still of the form
\[
\Phi_{\widetilde{H}_{\lambda};\omega_{\lambda}}^s ( q, t, r , \tau ,\sigma)  
  = \left(  q,  t+ \beta_{\lambda}(\tau s , f_K^t(q),t),  r, \tau - s \sigma - \frac{1}{2} s^2 (r-1 ), \sigma +s (r- 1)\right),
\]
where $\beta_{\lambda}$ is the function obtained from $b_{\lambda}$ via Lemma \ref{lem:reparametrise}. In particular for $\lambda = 1$ one has 
\[
\Phi_{\widetilde{H}_1;\omega_1}^1 ( q, t, r , \tau ,\sigma)  
  = \left(  q,  t+ \frac{\tau}{T},  r, \tau - \sigma - \frac{1}{2}  (r-1 ), \sigma + (r- 1)\right).
\]
Moreover, using Lemma \ref{lem:mapping_cylinder}, one can find a family $a_{\lambda} : B \times S^1 \to S^1 $ of smooth functions such that $a_0 = a$, $a_{\lambda}(0,t) = t - \eta /T$ (mod 1) for all $\lambda$,  and such that  $a_1(q,t) \equiv t - \eta/T $ (mod 1), with the property that if $\Phi_{\lambda}$ is defined by 
\[
\Phi_{\lambda}(q,t, r, \tau ,\sigma) = (A(q),a_{\lambda}(f_K^t (q), t)), r, \tau ,\sigma ),
\]
 then $\Phi_{\lambda}$ is a family of $\omega_{\lambda}$-symplectomorphisms which commute with $\Phi_{\widetilde{H}_{\lambda} ;\omega_{\lambda}}^s$. Thus the composition $\Phi_{\lambda} \circ \Phi^1_{\widetilde{H}_{\lambda} ; \omega_{\lambda}}$ is given by 
\[
\Phi_{\lambda} \circ \Phi_{\widetilde{H}_{\lambda};\omega_{\lambda}}^1 ( q, t, r , \tau ,\sigma)  
  = \left( A(q), a_{\lambda}(f_K^{t + \beta_{\lambda}(\tau, f_K^t(q),t)}(q), t+ \beta_{\lambda}(\tau,  f_K^t(q),t)), r , \tau - \sigma - \frac{1}{2}(r-1), \sigma + r - 1 \right), 
\]
and in particular for $\lambda  = 1$ is simply 
\[
\begin{split}
\Phi_1\circ \Phi_{\widetilde{H}_1;\omega_1}^1 ( q, t, r , \tau ,\sigma)  
&   = \left( A(q), t + \frac{1}{T}(\tau - \eta ), r , \tau - \sigma - \frac{1}{2}(r-1), \sigma + r - 1 \right) \\
& = \left( AP^k(q),  \frac{1}{T}(\tau - kT - \eta ), r , \tau - \sigma - \frac{1}{2}(r-1), \sigma + r - 1 \right) .
\end{split}
\]
It follows directly from the construction that $S_k$ remains uniformly isolated during this deformation, and this completes the proof. 

\end{proof}

In the next section we will prove an extension of a result of Ginzburg and G\"urel  \cite{GinzburgGurel2010} on the persistence of local Floer homology. The next corollary is an immediate consequence of Proposition \ref{prop:reduction_to_poincare_map} and this result (stated as  Theorem \ref{thm:GG2} below). We continue to adopt the notation from Proposition \ref{prop:reduction_to_poincare_map}.

\begin{cor}
\label{cor:uniform_bound_for_us}
Suppose $\gamma : \R \to \Sigma$ is a closed invariant Reeb orbit for a strict contactomorphism $\varphi$. Assume that $\Gamma_k$ is an isolated subset of $\mathcal{P}_1(\widetilde{L}^c)$ for each $k \in \N$. Then there exists a constant $C > 0$ such that
\[
  \mathrm{rank\,}\HF^{\loc}(\widetilde{L}^c ; \Gamma_k) \le C, \qquad \forall\, k \in \N.
\]
\end{cor}

The proof of Theorem \ref{thm:infinitely_many} is by now a standard argument using Corollary \ref{cor:uniform_bound_for_us}.

\begin{proof}[Proof of Theorem \ref{thm:infinitely_many}]
Assume that $\varphi$ has only finitely many invariant Reeb orbits. We will show that the positive growth rate $\Gamma_+(M_1, \lambda_1)$ is at most 1. Since there are only finitely many invariant orbits, they are necessarily isolated (as invariant orbits). Let us enumerate these orbits as $\delta_1 ,\dots ,\delta_p, \gamma_1, \dots , \gamma_q$, where the orbits $\delta_j$ are not closed, and the orbits $\gamma_j$ are closed of minimal period $T_j$. As explained above, each $\delta_j$ gives rise to a unique component $\Delta_j \subset \mathcal{P}_1( \widetilde{L}^c)$, and each $\gamma_j$ givese rise to a family $\{ \Gamma_{j,k} \}_{k \in \N}$ of components of $\mathcal{P}_1( \widetilde{L}^c)$. By Corollary \ref{cor:uniform_bound_for_us}, there is a constant $C > 0$ such that for all $j, k$,
\[
  \mathrm{rank\,} \HF^{\loc}(\widetilde{L}^c , \Delta_j) \le C, \qquad  \mathrm{rank\,} \HF^{\loc}(\widetilde{L}^c , \Gamma_{j,k}) \le C.
\]
It then follows directly from Lemma \ref{lem:building_blocks} that the growth rate of $\Gamma_+( \widehat{\varphi})$ is at most 1. Theorem \ref{thm:HF=RFH} then completes the proof.
\end{proof}

% ============================= 
\section{An extension of the  Ginzburg-G\"urel theorem}
\label{sec:GG}

\subsection{Statement of the theorem}

In this section we will prove a minor extension of a result of Ginzburg and G\"urel on the persistence of the local Floer homology groups under iteration of an isolated fixed point of a Hamiltonian diffeomorphism. 

\begin{defn}
Suppose that $\varphi$ is a diffeomorphism of a smooth manifold $M^m$ and $x \in M$ is a fixed point of $ \varphi$. Let $\lambda_1 ,\dots , \lambda_m$ denote the eigenvalues of $D \varphi(x)$ (with multiplicities). We say that an integer $k \in \N$ is \textbf{$(\psi,x)$-admissible} or simply \textbf{admissible} if 
\[
\lambda_i^k = 1 \quad \then \quad \lambda_i = 1, \qquad \forall \, i = 1, \dots m.
\]
\end{defn}

If $x$ is an isolated fixed point of $\varphi $ and $k$ is admissible then $x$ is necessarily also an isolated fixed point of $\varphi$ (\cite[Proposition 1.1]{GinzburgGurel2010}). The following result is proved in \cite[Theorem 1.1]{GinzburgGurel2010}. 

\begin{thm}
\label{thm:GG}
Let $( M, \omega)$ denote a symplectically aspherical manifold, and suppose $\psi$ is a Hamiltonian diffeomorphism of $(M, \omega)$. Suppose $x \in M$ is an isolated fixed point of $ \psi$. Then for each admissible $k$ one has
\begin{equation}
\label{persistence_of_loc_FH}
\HF^{\loc}( \psi^k ,x) \cong \HF^{\loc}( \psi ,x),
\end{equation}
i.e. the local Floer homology groups persist under iteration (up to a degree shift).
\end{thm}

\begin{defn}
Suppose that $\varphi$ and $\psi$ are two diffeomorphisms of a smooth manifold $M^m$, and $x$ is an isolated fixed point of both $\varphi$ and $\psi$. Assume the linear maps $D \varphi (0)$ and $D \psi(0)$ commute. Let $\lambda_1 ,\dots , \lambda_m$ denote the eigenvalues of $D \varphi(x)$ (with multiplicities), and $\mu_1 , \dots , \mu_m$ denote the eigenvalues of $D \psi(x)$ (with multiplicites), ordered so that for every $k \in \N$ the eigenvalues of $D (\varphi \psi^k)(x)$ are precisely $\lambda_1 \mu_1^k ,\dots , \lambda_m \mu_m^k$. We say an integer $k \in \N$ is \textbf{$(\varphi,\psi,x)$-admissible} or simply \textbf{admissible} if 
\[
\lambda_i \mu_i^k = 1 \quad \then \quad \lambda_i = 1 \text{ and } \mu_i = 1, \qquad \forall \, i = 1, \dots m.
\]
\end{defn}

Here we will prove the following generalisation of Theorem \ref{thm:GG}.
\begin{thm}
\label{thm:GG2}
Let $( M, \omega)$ denote a symplectically aspherical manifold, and let $\varphi $  and $\psi $ denote two Hamiltonian diffeomorphisms of $(M, \omega)$. Suppose $x \in M$ is an isolated fixed point of both $ \varphi $ and $\psi$, and that the linear maps $D \varphi(0)$ and $D \psi(0)$ commute. Then there exists $k_0 > 0$ with the property that for every $k > k_0$ which is $(\varphi,\psi,x)$-admissible, one has
\[
\HF^{\loc}( \varphi \psi^{k_0} ,x) \cong \HF^{\loc}( \psi ,x).
\]
\end{thm}

\begin{rem}
In fact, in \cite{GinzburgGurel2010}, the authors prove rather more than is stated in Theorem \ref{thm:GG}. For instance, denoting the degree shift in Theorem \ref{thm:GG} by $m_k$, they show that the limit $ \lim_{k \to \infty} m_k /k$ converges to the \textbf{mean index} of $x$. See \cite{GinzburgGurel2010} for more information. We will not need these properties in this paper. However, the interested reader can find the analogue of these statements, together with their proofs, in \cite{Naef_thesis}.
\end{rem}

We will prove Theorem \ref{thm:GG2} in Section \ref{sec:completing_the_proof} below.  An easy corollary of Theorem \ref{thm:GG2} is the following result.

\begin{cor}
\label{cor:uniform_bound}
Let $( M, \omega)$ denote a symplectically aspherical manifold, and let $\varphi $  and $\psi $ denote two Hamiltonian diffeomorphisms of $(M, \omega)$. Suppose $x \in M$ is a fixed point of both $ \varphi $ and $\psi$, and the the linear maps $D \varphi(0)$ and $D \psi(0)$ commute, and that $x$ is an isolated fixed point of $\varphi \psi^k$ for all $k = 0 ,1, 2, \dots$. Then there exists a constant $C >0$ such that
\[
\mathrm{rank\,}\HF^{\loc}(\varphi \psi^k ,x) \le C, \qquad \text{for all } k  = 0 ,1 ,2 ,\dots.
\]
\end{cor}

\begin{proof}
There are two ways an integer can fail to be admissible. Firstly, if there exists $p \in \N$ such that both $D \varphi(0)$ and $D \psi(0)$ have a $p$th root of unity as an eigenvalue, say $\lambda$ and $\mu$ respectively. Then if $m$ is the minimal positive integer such that $\lambda \mu^m = 1$ then every integer of the form $m + l p$ will fail to be admissible. Secondly, it could happen that $D \varphi(0)$ has an eigenvalue $\lambda$ which is not a root of unity, and $D \psi(0)$ has an eigenvalue $\mu = \lambda^{1/k}$. Then $k$ will also not be admissible. However for each eigenvalue $\lambda$ of $D \varphi(0)$ that is not a root of unity, the second possibility can only happen for at most one iterate $\psi^k$. Thus there exists $k_0 \in \N$ such that for $k \ge k_0$, the only way $k$ could fail to be admissible is via the first possibility. 

Now let us deal with the first possibility. Suppose for simplicity that $D \varphi(0)$ and $D \psi(0)$ both have a $p$th root of unity ($p > 1$) as an eigenvalue, say $\lambda$ and $\mu$, but no other common roots of unity for eigenvalues. Let $1 \le m \le p$ be the minimal positive integer such that $\lambda \mu^m =1$. Then as we already noted, every number of the form $m + l p$ will fail to be admissible. But in this case we simply set $\tilde{\varphi} := \varphi \psi^m$ and $\tilde{\psi} := \psi^p$. Then by assumption, every integer $k \ge (k_0 - m)/p$ is $(\tilde{\varphi} , \tilde{\psi}, x)$-admissible, and we can apply Theorem \ref{thm:GG2} to $\tilde{\varphi}$ and $\tilde{\psi}$. Finally, the general case is similar: if $D \varphi(0)$ and $D \psi(0)$ have multiple common roots of unity  as an eigenvalue, then it is easy to see that we can find finitely many pairs $(\tilde{\varphi}_j,\tilde{\psi}_j) $ all of the form $\tilde{\varphi}_j = \varphi \psi^{m_j}$ and $\tilde{\psi}_j = \psi^{p_j}$ for some integers $(m_j,p_j)$ such that every integer $k \in \N$ is either $(\varphi, \psi,x)$-admissible, or $(\tilde{\varphi}_j,\tilde{\psi}_j,x)$-admissible for some $j$. See for instance \cite[Lemma 6.5]{HryniewiczMacarini2012} for a proof of a similar statement. 
\end{proof}

\begin{rem}
\label{rem:removing_commuting_would_be_real_nice}
We conjecture that the assumption in Corollary \ref{cor:uniform_bound} that $D \varphi (0)$ and $D \psi (0)$ commute is superfluous (but not in Theorem \ref{thm:GG2}) As explained in the Introduction, this would allow us to extend the results of this paper to prove the existence of infinitely many geometrically distinct leaf-wise intersections for \textbf{all} Hamiltonian diffeomorphisms,  rather than just those arising from lifts of strict contactomorphisms; this is the statement of Conjecture \ref{conj:gromoll_meyer_leafwise}.
\end{rem}

\subsection{Generating functions and local Morse homology} 
\label{sec:generating_functions}

Before proving Theorem \ref{thm:GG2} we will need to recall some preliminaries on generating functions and local Morse homology. Let us begin by fixing once and for all some sign conventions. We will always use coordinates $(x_1 , \dots , x_n , y_1, \dots ,y_n)$ on $ \R^{2n}$, and we equip $ \R^{2n}$ with the canonical symplectic form
\[ \omega_0 = \sum_{j=1}^n dy_j \wedge dx_j.
\]
 We denote by $ \bar{\R}^{2n}$ the same space but endowed with the symplectic form $ - \omega_0$. The symplectic gradient $ X_F$ of a smooth function $ F : \R^{2n} \to \R$ is given as usual by $ \omega_0 ( X_F , \cdot) = - dF( \cdot)$, which implies that 
\begin{equation}
\label{eq:sym_grad}
  X_F (x ,y) = ( \partial_2 F( x, y) , - \partial_1 F( x, y)).
\end{equation}
We will also work with the cotangent bundle $T^* \R^{2n}$. We will often use the notation $(q,p)$ to indicate elements of $T^* \R^{2n}$, so that $ q \in \R^{2n}$ and $ p \in T_q^* \R^{2n}$. We endow $T^* \R^{2n}$ with the symplectic form $ \omega_{\mathrm{can}} = \sum_{j=1}^n dp_j \wedge dq_j$. 

\begin{defn}
\label{def:symplecto_of_type_GF}
Consider the symplectomorphism $ \Phi : \R^{2n} \times \bar{ \R}^{2n} \to T^* \R^{2n}$ defined by
  \begin{equation}
  \label{eq:fav_example}
  \Phi(x,y, \bar{x} , \bar{y}) : = ( \bar{x}, y, y- \bar{y}, \bar{x} - x).
  \end{equation}
Note that $ \Phi$ carries the diagonal $ \Delta \subset \R^{2n} \times \bar{ \R}^{2n}$ onto the zero section $O_{ \R^{2n}} \subset T^* \R^{2n}$, and if 
  \[
  \label{subspace_N}
  N : = \R^n \times \{0 \} \times \{ 0 \} \times \R^n,
  \] 
then $ \Phi(N) = T_0^*\R^{2n}$.
\end{defn}

Given a diffeomorphism of $ \R^{2n}$, we denote by  
\[
 \mathrm{gr}( \varphi) = \left\{ (x, y , \varphi(x,y)) \mid (x,y) \in \R^{2n} \right\}
 \]
the graph of $ \varphi$ in $ \R^{2n} \times \bar{ \R}^{2n}$. Note that $ \mathrm{gr}( \varphi)$ is a Lagrangian submanifold of $\R^{2n} \times \bar{ \R}^{2n}$ if and only if $ \varphi$ is a symplectomorphism.  Similarly if $F : \R^{2n} \to \R$ is a smooth function, we denote by  
  \[
  \mathrm{gr}( dF) = \left\{ (q , dF(q))  \mid q \in \R^{2n} \right\}
  \]
the graph of the one form $dF$ inside $T^* \R^{2n}$.
Writing $ q = (x,y)$, so that $ x,y \in \R^n$, we can alternatively write  
  \[
  \mathrm{gr}( dF) = \left\{ (x ,y , \partial_1 F(x,y), \partial_2 F(x,y)  ) \mid (x,y) \in \R^{2n} \right\}.
  \]
The submanifold $\mathrm{gr}(dF)$ is always a Lagrangian submanifold of $T^* \R^{2n}$.

\begin{defn}
Suppose $ \varphi$ is a Hamiltonian diffeomorphism defined on an open neighbourhood $U$ of the origin $0 \in \R^{2n}$. Assume that $\varphi(0) = 0$ and that
\begin{equation}
\label{eq:criterion_for_generating_function}
\| D \varphi (z) - \mathrm{Id} \| <  \frac{1}{2}, \qquad \forall\, z \in U.
\end{equation}
This condition ensures that the submanifold $\mathrm{gr}( \varphi)$ is  sufficiently $C^0$-close to the diagonal $ \Delta$ so that there is a well defined function $F$ such that
  \begin{equation}
  \label{eq:gen_fun_eq}
  \Phi( \mathrm{gr}( \varphi) )= \mathrm{gr}( dF).
  \end{equation}
We call $F$ the \textbf{generating function for $ \varphi$}. The function $F$ is unique up to a constant, and hence we normalise $F$ by requiring that $F(0) = 0$. Note that $0$ is a critical point of $F$. One can show that $0$ is an isolated fixed point of $\varphi$ if and only if $0$ is an isolated critical point of $F$, and that there exists a constant $C$ (independent of $\varphi$) such that
  \begin{equation}
  \label{C^2small}
  \| F \|_{C^2(U)} \le C  \| \varphi - \mathrm{Id} \|_{C^1(U)}.
  \end{equation}
\end{defn}

To make things more explicit, let us temporarily use the (awful) notation $ \varphi(x,y) = (x_{ \varphi}, y_{ \varphi})$. Then from \eqref{eq:fav_example} and \eqref{eq:gen_fun_eq}, we see that 
\[
(x_\varphi  , y, y - y_ \varphi , x_ \varphi  - x) = ( x, y , \partial_1 F( x ,y), \partial_2 F( x ,y).
\]
Using \eqref{eq:sym_grad}, this is equivalent to 
\begin{equation}
\label{eq:can_gen_fun_eq}
  \varphi(x,y) - (x,y) = X_F ( x_ \varphi ,y).
\end{equation}
It is convenient to introduce the auxilliary map $ f $ defined by
\begin{equation}
\label{eq:auxilliary_diffeo}
  f(x,y) = (x_ \varphi, y).
\end{equation}
Since $ \varphi$ is $C^1$-close to the identity the map $f$ is a diffeomorphism on a suitably small neighbourhood of the origin. Then we can rewrite \eqref{eq:can_gen_fun_eq} as
\begin{equation}
\label{eq:can_gen_fun_eq_2}
  \varphi(x,y) - (x,y) = X_F (  f (x,y)).
\end{equation}

In fact, in this section we will only ever be concerned with the case when our Hamiltonian diffeomorphism $\varphi$ is maximally degenerate, meaning that all the eigenvalues of $D \varphi (0)$ are equal to $1$. In this case by making an appropriate symplectic change of basis we can always ensure that \eqref{eq:criterion_for_generating_function} holds. To prove this fact we use the following lemma, which is  just symplectic  linear algebra. The statement and proof are almost identical to \cite[Lemma 5.5]{Ginzburg2010} (which deals with the case of a single symplectic matrix), but for the convenience of the reader we repeat the proof here. In the statement, $\| \cdot \|$ denotes any norm on $\mathrm{Sp}(W,\omega)$, for instance the restriction of the Euclidean norm on $\mathrm{GL}(W,\R) \cong \R^{\dim W^2}$. 

\begin{lem}
\label{lem:baby_linear_algebra}
Suppose $A, B \in \mathrm{Sp}(W, \omega)$ are two commuting symplectic matrices. Assume that all the eigenvalues of both $A$ and $B$ are equal to 1. Then for any $\varepsilon > 0$ there exists $C \in \mathrm{Sp}(W,\omega)$ such that $\|CAC^{-1} - I_W \| < \varepsilon$ and $ \| CBC^{-1} - I_W \| < \varepsilon$.
\end{lem}

\begin{proof}
We prove the result by induction on $\dim W$. If $\dim W = 2 $ then the result is trivial. For the inductive step, set $K:= \ker (A - I_W) \cap \ker (B-I_W) $. Since $A$ and $B$ commute, the nilpotent matrices $A - I_W$ and $B - I_W$ also commute, and elementary linear algebra tells us that $ K \ne \{ 0 \}$. There are three cases to consider:
\begin{enumerate}
\item $K$ contains a symplectic subspace $V$.
\item $K$ is an isotropic subspace but not  Lagrangian subspace of $W$.
\item $K$ is a Lagrangian subspace of $W$. 
\end{enumerate}
Case (1) is easy, since in this case we can split $W= V \oplus V^{\omega}$, where $V^{\omega} $ is the symplectic orthogonal of $V$. Both $A$ and $B$ preserve this splitting, and $A|_V = B|_V = I_V$. Now apply the inductive hypothesis to $A|_{V^{\omega}}$ and $B|_{V^{\omega}}$. \\
Case (2) is more complicated. Since $K$ is not Lagrangian, this time $K^{\omega}$ contains a symplectic subspace $V$ which is complementary to $K$. Fix an isotropic subspace $U$ of $W$ which is complementary to $K^{\omega}$. Both $A$ and $B$ preserve $K$ and $K^{\omega}$. Moreover since $U \cong W / K^{\omega} \cong K^*$, we see that  
\[
  A|_K = B|_K = I_K, \qquad A|_U = B|_U = I_U.
\]  
Both $A$ and $B$ induce maps $A_0 , B_0 \in \mathrm{Sp}(V,\omega_V)$, where $\omega_V := \omega|_{V \oplus V}$, such that  $A_0$ and $B_0$ have all their eigenvalues equal to 1. By the inductive hypothesis, we can choose $C_0 \in \mathrm{Sp}(V, \omega_V)$ such that $\| C_0 A_0 C_0^{-1} -I_V \| < \varepsilon/2$ and $\| C_0 B_0 C_0^{-1} -I_V\| < \varepsilon/2$. Thus with respect to this splitting we can write 
\[
  A =  \begin{pmatrix}
  I_K & X & Y  \\
  0 & A_0 & Z \\
  0 & 0 & I_U
  \end{pmatrix}, \qquad
  B =  \begin{pmatrix}
  I_K & X' & Y'  \\
  0 & B_0 & Z' \\
  0 & 0 & I_U
  \end{pmatrix}
\]
for some matrix operators  $X,X' : V \to K$, $Y,Y' : U  \to K$ and $Z,Z' : U \to V$. Now fix a matrix $D \in \mathrm{GL}(K, \R)$ and consider the symplectic matrix 
\[
  C :=  \begin{pmatrix}
  D & 0 & 0 \\
  0 & C_0 & 0 \\
  0 & 0 & (D^*)^{-1} 
  \end{pmatrix},
\]
where as above we identified $U$ with $K^*$. Then we have
\[
  C A C^{-1} =  \begin{pmatrix}
  I_K & D X C_0^{-1} & DYD^* \\
  0 & C_0 A_0 C_0^{-1} & C_0 Z D^* \\
  0 & 0 & I_U
  \end{pmatrix}, \qquad
  C B C^{-1} =  \begin{pmatrix}
  I_K & D X' C_0^{-1} & DY'D^* \\
  0 & C_0 B_0 C_0^{-1} & C_0 Z' D^* \\
  0 & 0 & I_U
  \end{pmatrix}
\]
Since $D$ is close to zero if and only if $D^*$ is, it is clear from this expression that we can now choose $D$ so that both $ \| C A C^{-1} - I_W \| < \varepsilon$ and $\| C B C^{-1} - I_W \| <\varepsilon$. \\
Finally Case (3) is like Case (2), only easier. In this case we choose a complementary Lagrangian subspace $L$ of $W$ and consider $A$ and $B$ with respect to the decomposition $W = K \oplus L$: 
\[
  A =  \begin{pmatrix}
  I_K & X \\
  0 & I_L 
  \end{pmatrix}, 
  \qquad
  B =  \begin{pmatrix}
  I_K & X' \\
  0 & I_L 
  \end{pmatrix},
\]
for $X, X' : L \to K$. We identify $L$ with $K^*$, and as before choose a map $D \in \mathrm{GL}(K)$. Then for 
\[
  C =  \begin{pmatrix}
  D & 0 \\
  0 & (D^*)^{-1}
  \end{pmatrix}
\]
we again have 
\[
  C A C^{-1} = 
   \begin{pmatrix}
  I_K & D X D^* \\
  0 & I_L 
  \end{pmatrix}, \qquad
    C B C^{-1} = 
   \begin{pmatrix}
  I_K & D X' D^* \\
  0 & I_L 
  \end{pmatrix}.
\]
As before, it is clear for an appropriate choice of $D$ we can make $ \|CAC^{-1} - I_W\| < \varepsilon$ and $ \| C B C^{-1} - I_W \| < \varepsilon$.
\end{proof} 

The next corollary is an immediate consequence of the preceding lemma, together with the observation that the proof shows that one can choose $C$ so that $C \mapsto \| C A C^{-1} - I_W \|$ is continuous. 

\begin{cor}
\label{cor:parametrized_linear_algebra}
Suppose $ \varphi$ and $\psi$ are two Hamiltonian diffeomorphisms defined on the ball $B(r_0)  \subset \R^{2n}$, and that  $0$ is an isolated fixed point of  both $\varphi$ and $\psi$. Assume in addition that the linear maps $D \varphi(0)$ and $D \psi(0)$  commute, and that all the eigenvalues of $D \varphi(0)$ and $D \psi(0)$ are equal to $1$. Then there exists $T > 0$ and two smooth paths 
\[
C : [T, + \infty) \to \mathrm{Sp}(2n), \qquad  \text{and} \qquad r : [T, + \infty) \to (0 ,r_0)
\]
such that $ \lim_{s \to + \infty} r(s) = 0$  with the following properties. For each $s \ge T$, let
\[
  \varphi_s : = C(s) \circ \varphi \circ C(s)^{-1}, \qquad \psi_s : = C(s) \circ \psi \circ C(s)^{-1},
\]
and set
\[
   c( \varphi_s) : = \| \varphi_s - \mathrm{Id}\|_{C^1(B(r(s)))} , \qquad  c(\psi_s) := \| \psi_s - \mathrm{Id}\|_{C^1(B(r(s)))} .
\]
Then for each $s \ge T$, 
\[
  c(\varphi_s) < \frac{1}{s}, \qquad c(\psi_s) < \frac{1}{s}, 
\]
and moreover there exists a constant $ R> 0$ such that for all $s \ge T$,
\begin{equation}
  \label{eq:cs_is_continuous}
  R c( \psi_s) - c (\psi_s) > \frac{1}{2s}.
\end{equation}
\end{cor}

Let us quickly recall the definition of local Morse homology.
\begin{defn}
Let $F : M \to \R$ denote a smooth function on a manifold $M$, and suppose $x \in M$ is an isolated critical point of $F$. Fix neighbourhoods $U \subset V \subset M$ of $x$ such that $\crit F \cap V = \{x \}$. Choose a $C^1$-small perturbation $G$ of $F$ such that $F = G$ outside $U$ and such that $G|_V$ is a Morse function on $V$. Fix a Riemannian metric $g$ on $M$ such that $g$ is Morse-Smale for $G|_V$. By construction every (broken) gradient flow line of $- \nabla_g G$ of $G$ whose asymptotes lie in $V$ never leave $U$. Thus the subspace of the Morse complex $\mathrm{CM}(G)$ of  $G$ generated by the critical points of $G$ in $V$ is a subcomplex, and hence it makes sense to speak of its homology. We denote it by $\HM^{\loc}(F,x)$ and call it the \textbf{local Morse homology} of $F$ at $x$. The notation makes sense since the usual continuation arguments show that the homology of the subcomplex of $\mathrm{CM}(G)$ generated by the critical points of $G$ in $V$ is independent of the perturbation $G$.
\end{defn} 

A key property of local Morse homology is that if $\{F_t \}_{t \in [0,1]}$ is a smooth family of smooth functions and $x \in M$ is a \textbf{uniformly isolated } critical point (i.e. there exists a neighbourhood $U \subset M$ of $x$ such that $\left( \bigcup_{t \in [0,1]} \crit F_t \right) \cap U = \{x \}$) then the local Morse homology groups of $F_t$ at $x$ are independent of $t$. We refer the reader to \cite[Section  3.1]{Ginzburg2010} for more information about local Morse homology. \\

The following theorem, which is due to Ginzburg \cite{Ginzburg2010}, connects the local Floer homology of a maximally degenerate isolated fixed point of a Hamiltonian diffeomorphism $\varphi$ with the local Morse homology of its generating function. We state only the special case that we need.

\begin{thm}
\label{thm:floer=morse}
Suppose $\varphi $ is a Hamiltonian diffeomorphism defined on a small neighbourhood $U$ of the origin in $\R^{2n}$. Assume that $\varphi$ has an isolated fixed point at $0$, and that all the eigenvalues of $D \varphi(0)$ are equal to $1$. There exists a constant $\varepsilon_0 > 0$ (depending on $\varphi$) with the following property. Given  $0 < \varepsilon < \varepsilon_0$, choose  $C \in \mathrm{Sp}(2n)$ and $r > 0$ such that $\| C \circ \varphi \circ C^{-1} - \mathrm{Id} \|_{C^1(B(r))} < \varepsilon $ (such $C, r$ exist by  Corollary \ref{cor:parametrized_linear_algebra}). Then if $F:B(r) \to \R$ denotes the generating function of $C \circ \varphi \circ C^{-1}$, one has 
\[
\HF^{\loc}(\varphi, 0) \cong \HM^{\loc}(F,0).
\]
\end{thm}

The next result, which is the main one of this section,  is a minor extension of  \cite[Claim 4.1, p339]{GinzburgGurel2010}, which deals with the case where $\varphi = \mathrm{Id}$. 

\begin{thm}
\label{thm:large_iterates_always_work}
Suppose $\varphi ,\psi$ are two Hamiltonian diffeomorphisms defined on a neighbourhood $U$ of $ 0 \in \R^{2n}$. Assume that $0$ is an isolated fixed point of all three of $\varphi$, $\psi$ and $\varphi \psi$, and assume that the linear maps $D \varphi(0)$ and $D \psi(0)$ commute, and that all the eigenvalues of $D \varphi(0)$ and $D \psi (0)$ are equal to $1$. Then for all $k$ sufficiently large, one has
\[
  \HF^{\loc}( \varphi \psi^k ,0) \cong \HF^{\loc}( \psi ,0 ).
\]
\end{thm}

\begin{proof}
Let us begin by giving a heuristic idea of the proof. We warn the reader that this argument contains a technical gap, which will be fixed below. Since $D \varphi (0)$ and $D \psi(0)$ commute and have all their eigevalues equal to 1, by choosing an appropriate symplectic basis we may assume that both $\varphi$, $\psi$ and each iterate $\varphi \psi^k$ have well defined generating functions $F, G$ and $K_k$ respectively. By assumption $0$ is an isolated fixed point of all of $F,G$ and $K_k$. Fix an iterate $k$, and consider the function 
\[
H_k(z,t) : =t K_k(z) + (1-t)k G(z).
\]
The main step in the proof will be to show that there exists an integer $k_0$ such that if $k \ge k_0$ then $0$ is a uniformly isolated fixed point of $H_k$. Therefore by invariance of local Morse homology one has 
\[
\HM^{\loc}(K_k, 0) \cong \HM^{\loc}(kG,0).
\]
From this the result follows, since by Theorem \ref{thm:floer=morse} one has  $\HF^{\loc}(\varphi \psi^k,0) \cong \HM^{\loc}(K_k,0)$ and $\HF^{\loc}(\psi,0) = \HM^{\loc}(G,0)$ (and clearly $\HM^{\loc}(G,0)$ is invariant under replacing $G$ by a scalar multiple of $G$). This argument is essentially the same as the argument in \cite[Claim 4.1, 339]{GinzburgGurel2010}. Unfortunately there is a small gap in the reasoning above (which does not occur in the setting studied in \cite[Claim 4.1, 339]{GinzburgGurel2010}). Namely, in reality in order to define the generating functions $F,G$ and $K_k$, we first fix an integer $k_1 \in \N$, and then choose a symplectic basis such that all of $\varphi, \psi$ and $\varphi \psi^l$ for $1 \le l \le k_1$ are sufficiently $C^1$-close to the identity so as to admit generating functions. And herein lies the problem:  a priori, the integer $k_0$ depends on the original integer $k_1$ we choose, and of course the argument is meaningless unless we can make sure that $k_0 < k_1$! Luckily it turns out that this can be done (this is the point of \eqref{eq:cs_is_continuous} in Corollary \ref{cor:parametrized_linear_algebra}), but it complicates the argument somewhat, and is the explanation for profligate use of $s$'s in the proof below.

Let $C: [T, + \infty) \to \mathrm{Sp}(2n)$ and $r : [T, + \infty) \to (0,r_0)$ be as in Corollary \ref{cor:parametrized_linear_algebra}. Choose an integer $k_0 > R$, where $R$ is as in \eqref{eq:cs_is_continuous}. Now select $k_1  \gg 2k_0$, and choose $s > T$ large enough so that for each $1 \le l \le k_1$, the maps $\varphi_s$, $\psi_s$ and $\varphi_s \psi_s^l$ all admit generating functions $F_s$, $G_s$ and $K_{s,l}$ respectively, with

\[
  \HF^{\loc}( \varphi \psi^l , 0) \cong \HM^{\loc}(K_{s, l}, 0), \qquad \mathrm{and} \qquad \HF^{\loc}(\psi,0) \cong \HM^{\loc}(G_{s},0 ).
\] 
Note that it follows from \eqref{eq:cs_is_continuous} and the fact that $k_0 > R$ that up to shrinking $r(s)$, we may additionally assume that:
\begin{equation}
\label{eq:key!}
   k_0 \min_{ z \in B(r(s))} \| D \psi_{s} (z) - I \| - \max_{z \in B(r(s))} \|D \varphi_{s}(z)- I \| \ge \frac{1}{4s}.
\end{equation}
We now prove the result in three stages. As before, let us define auxilliary functions $f_{s}, g_{s}, k_{s,l}$ so that
\[
  \varphi_{s}(z) -z = X_{F_{s}}(f_{s}(z)), \qquad \psi_{s}(z)- z = X_{G_{s}}(g_{s}(z)), \qquad \varphi_{s} \psi_{s}^l (z) - z = X_{K_{s,l}}(k_{s,l}(z )).
\]

\textbf{Step 1:}
We prove that for each $1 \le l \le k_1$, 
\begin{equation}
  \label{eq:step1bound}
  \| X_{K_{s,l}}(k_{s,l}(z)) - l X_{G_{s}}(g_{ s}(z)) - X_{F_{ s}}(f_{s}(z)) \| = O(s^{-1})  \| X_{G_{s}}(g_{s}(z))\| .
\end{equation}
We argue by induction on $l$.  Since for $l \ge 1$ one has
\[
  \varphi_{s}\psi_{s}^k (z ) - z = \left(  \varphi_{s} \psi_{s}^{k-1} - \mathrm{Id} \right)(\psi_{s}(z)) + \psi_{s}(z) - z, 
\]
we see that
\[
  X_{K_{s,l}}(k_{s,l}(z)) = X_{K_{s, l-1}}(k_{s,l-1}(\psi_{s}(z))) + X_{G_{s}}(g_{s}(z)).
\]
Thus for $l = 1$ we can estimate
\[
\begin{split}
\| X_{K_{ s,1} }(k_{s ,1}(z)) - X_{F_{s}} (f_{s}(z)) - X_{G_{s}}(g_{s}(z)) \|  & = \|   X_{F_{s}} (f_{s} ( \psi (z))) - X_{F_{s}}(f_{s}(z )) \| \\
& \le \|X_{F_{s}} \|_{C^1} \| f_{s} \|_{C^1} \| \psi_{s} (z) - z \| \\
& \le   \|X_{F_{s}} \|_{C^1} \| f_{s} \|_{C^1} \| X_{G_{s}}(g_{s}(z)) \| \\
& = O(s^{-1})   \| X_{G_{s}}(g_{s}(z)) \|.
\end{split}
\]
Now for the inductive step we argue as follows:
\[
\begin{split}
     \| X_{K_{s,l}}(k_{s,l}(z)) -  l X_{G_{s}}(g_{ s}(z)) - X_{F_{ s}}(f_{s}(z)) \| & = \| X_{K_{s,l-1}}(k_{s,l-1}(\psi_{s}(z)) -  (l-1) X_{G_{s}}(g_{ s}(z)) - X_{F_{ s}}(f_{s}(z)) \| \\
     & \le    \| X_{K_{s, l-1}}(k_{s,l-1}(\psi_{s}(z)) ) - X_{K_{s, l-1}}(k_{s, l-1}(z) )\|  \\
      & \quad + \| X_{K_{s,l-1}}(k_{s,l-1}(z)) -  (l-1) X_{G_{s}}(g_{ s}(z)) - X_{F_{ s}}(f_{s}(z)) \|\\
      & \le  \| X_{K_{s, l-1}} \|_{C^1} \| k_{s, l-1} \|_{C^1} \| \psi_{s}(z) - z \| + O(s^{-1})   \| X_{G_{s}}(g_{s}(z)) \| \\
      & = O(s^{-1})   \| X_{G_{s}}(g_{s}(z)) \| + O(s^{-1})   \| X_{G_{s}}(g_{s}(z)) \| .
\end{split}
\]
The claim follows. As a consequence we also obtain the following inequality for $1 \le l \le k_1$:
\begin{equation}
  \label{eq:boundonXKl}
  \begin{split}
   \| X_{K_{s,l}}(k_{s,l}(z)) \| & \le \| X_{K_{s,l}}(k_{s,l}(z)) - l X_{G_{s}}(g_{ s}(z)) - X_{F_{ s}}(f_{s}(z)) \|  +l  \|   X_{G_{s}}(g_{ s}(z)) \| + \| X_{F_{ s}}(f_{s}(z)) \| \\
   &  \le (l + O(s^{-1}) \| X_{G_{s}}(g_{s}(z))\| + \| X_{F_{s}}(f_{s}(z) ) \| .
   \end{split}
\end{equation}

\textbf{Step 2:} 
Now consider the vector field
\[
Y_{s, l}(z,t) : = t X_{K_{s, l}} (k_{s, l} (z)) + (1-t) lX_{G_{s}}(g_{s}(z)) ,
\]
and observe that by Step 1,
\begin{equation}
\label{eq:Ybound1}
\begin{split}
\| Y_{s, l} (z, t) \| & \ge   l \|  X_{G_{s}}(g_{s}(z)) \| -  \|  X_{F_{s}}(f_{s}(z)) \|   - \| X_{K_{s, l}} (k_{s,l}(z)) -  l X_{G_{ s}}(g_{s}(z)) - X_{F_{s}}(f_{s}(z)) \| \\
& \ge (l - O(s^{-1})) \| X_{G_{s}}(g_{s}(z)) \| - \| X_{F_{s}}(f_{s}(z))\|.
\end{split}
\end{equation}
We now introduce for  $t \in [0,1]$ the homotopy
\[
 H_{s, l}(z,t) := t K_{s, l}(z ) + (1-t)lG_{s}(z).
\]
We have
\[
\|Y_{s,l}(z,t)- X_{H_{s,l}}(z,t) \| \le \|X_{K_{s, l}}(k_{s, l}(z)) - X_{K_{s, l}}(z) \| + l  \| X_{G_{s}}(g_{s}(z)) - X_{G_{s}}(z) \| .
\]
Since
\[
\begin{split}
\| X_{K_{s, l}} (k_{s, l}(z)) - X_{K_{s, l}}(z) \| & \le \|X_{K_{s, l}} \|_{C^1} \| k_{s, l}(z) - z  \| \\
& \le \|X_{K_{s, l}}\|_{C^1} \| \varphi_{s} \psi_{s}^l(z) - z \|  \\
& \le \| X_{K_{s, l}} \|_{C^1} \| X_{K_{s, l}} (k_{s, l}(z)) \| \\
& = O(s^{-1})  \| X_{K_{s, l}} (k_{s, l}(z)) \| \\
& = O(s^{-1}) \left( \| X_{G_{s}}(g_{s}(z))\| + \| X_{F_{s}}(f_{s}(z) ) \| \right) ,
\end{split}
\]
where the last line used \eqref{eq:boundonXKl}, and similarly  
\[
 \| X_{G_{s}}(g_{s}(z)) - X_{G_{s}}(z) \| = O(s^{-1}) \| X_{G_{s}}(g_{s}(z)) \|
\]
we see that 
\begin{equation}
\label{eq:X_H_bound}
\begin{split}
\|X_{H_{s,l}}(z,t)\| & \ge \| Y_{s,l}(z,t)\| - \| Y_{s,l}(z,t)- X_{H_{s,l}}(z,t) \| \\
& \ge ( l - O(s^{-1}))\|  X_{G_{s}}(g_{s}(z)) \| - (1 + O(s^{-1}))\| X_{F_{s}}(f_{s}(z) ) \| \\
& \ge (l-1)\|  X_{G_{s}}(g_{s}(z)) \| - 2 \| X_{F_{s}}(f_{s}(z) ) \| .
\end{split}
\end{equation}

\textbf{Step 3:}
We now prove that for each $2k_0  + 1 < l \le k_1$, 
\begin{equation}
  \label{eq:result_of_key}
  (l-1) \| X_{G_{s}}(g_{s}(z)) \| -2\|X_{F_{s}}(f_{s}(z)) \| \ge (l - 2 k_0 -1 ) \|X_{G_{s}}(g_{s}(z)) \| + \frac{1}{2s} \|z \|.
\end{equation}
Indeed,  for any diffeomorphism $\theta$ of $B(r)$ with $\theta(0) = 0$ one has for $z \in B(r)$ that
\[
  \theta(z) =  \left( \int_0^1 D \theta ( t z) \,dt \right) \cdot z,
\]
and thus in particular
\[
  \left( \min_{ w \in B(r)} \| D \theta (w) \|  \right)  \| z \| \le \| \theta (z) \| \le \left( \max_{ w \in B(r)} \| D \theta (w) \|  \right) \| z \|.
\]
Thus applying this with $\theta = \varphi_{s} - \mathrm{Id}$ and $\theta = \psi_{s} - \mathrm{Id}$, and using \eqref{eq:key!}, we see:
\[
  \begin{split}
  (l-1) \|X_{G_{s}}(g_{s}(z)) \| & =   (l - 2k_0 - 1)  \|X_{G_{s}}(g_{s}(z)) \| + 2 k_0 \| \psi_{s}(z) - z \| \\
 & \ge (l - 2k_0 - 1) \|X_{G_{s}}(g_{s}(z)) \| +2 k_0 \min_{w \in B(r(s))} \| D \psi_{s}(w) -I \| \|z \|  \\
  & \ge (l -2 k_0 - 1) \|X_{G_{s}}(g_{s}(z)) \| +  2\max_{w \in B(r(s))} \| D \varphi_{s}(w) -I \| \|z \| + \frac{1}{2s} \|z \| \\ 
  &\ge(l - 2k_0 - 1) \|X_{G_{s}}(g_{s}(z)) \| + 2 \| \varphi_{s}(z ) - z  \| + \frac{1}{2s} \|z \| \\
   &  = (l -2 k_0 - 1) \|X_{G_{s}}(g_{s}(z)) \| +2 \|X_{F_{s}}(f_{s}(z)) \| + \frac{1}{2s} \|z \|,
\end{split}
\]
which establishes \eqref{eq:result_of_key}. Now we combine \eqref{eq:X_H_bound} and \eqref{eq:result_of_key} to see that for $2k_0  + 1 < l \le k_1$, one has
\[
\begin{split}
\|X_{H_{s,l}}(z,t) \| & \ge (l - 1) \| X_{G_{s}}(g_{s}(z)) \| -2\| X_{F_{s}}(f_{s}(z)) \| \\
& \ge ( l - 2 k_0 - 1 ) \|  X_{G_{s}}(g_{s}(z)) \| +\frac{1}{2s} \| z \| .
\end{split}
\]
Since $g_{s}$ is a  diffeomorphism which fixes $0$, one has $g_{s}(z) = 0$ if and only if $z = 0 $. Since by assumption $0$ is a uniformly isolated zero of $X_{G_{s}}$, it follows that for $s > 0$ sufficiently small and $z$ sufficiently close to $0$, one has
\[
X_{H_{s,l}}(z,t) = 0 \qquad \iff \qquad  z = 0.
\]
Thus by invariance of local Morse homology, one has 
\[
  \HM^{\loc}(H_{s,l}(\cdot, 0), 0) \cong \HM^{\loc}(H_{s, l}(\cdot ,1),0),
\] 
which is what we wanted to prove. 
\end{proof}

\begin{rem}
\label{rem:different_proof}
Here is another way to prove Theorem \ref{thm:GG2} which is perhaps conceptually simpler. Let us denote by $F$ the generating function of $\varphi$, $G_k$ the generating function of $\psi^k$ and $K_k$ the generating function of $\varphi \psi^k$. Then for $k$ sufficiently large, one can show that 
\[
\HM^{\loc}(K_k,0) \cong \HF^{\loc}(F+G_k) \cong \HM^{\loc}(F + kG).
\]
Then the proof of Theorem \ref{thm:GG2} is completed via the following simple lemma:
\begin{lem}
\label{lem:morse_computation}
Suppose $F,G : B(r) \to \R$ are two smooth functions. Assume that the origin is an isolated critical point of both $F$ and $G$. Then there exists a constant $C = C\left(F,G\right) \ge 0$ such that for any two real numbers $  a,b$ such that $C<a<b$, the local Morse homology groups $\mathrm{HM}^\loc  (F + aG , 0)$ and $\mathrm{HM}^ \loc (F + bG ,0)$ coincide.
\end{lem}

\begin{proof}
First, choose $ 0 < r_0 < r/2$ such that neither $F$ or $G$ have any critical points in the punctured ball $\dot{B}(2r_0) := B(2r_0) \backslash \{0\}$. Consider the function
\[
h:  \dot{B}(r_0) \to \R, \qquad h(x) := \frac{\| X_F(x)\| }{ \| X_G(x) \|}.
\]
Suppose that $h$ is unbounded as $ |  x | \to 0$. Consider the homotopy $F_s(x)=F(x)+(sb+(1-s)a)G(x)$. Choose $r_1 > 0$ such that $h(x) > 2b$ on $\dot{B}(r_1)$.Then since 
\[
\|X_{F_s}(x) \| \ge( h(x) - (sb+ (1-s)a))  \| X_G(x)\| > b\|X_G(x ) \| ,
\]
we see that $0$ is a uniformly isolated zero of $X_{F_s}$, whence $\HF^{\loc}(F_s, 0)$ is independent of $s$ as required (in this case one can take the contant $C(F,G) = 0$). 
Suppose now that $h$ is bounded. Since for any function $H$ and any $c >0$ the (local) Morse homology of $H$ and $c H$ are isomorphic, it suffices to show that $\HM^{\loc}(\tfrac{1}{a}F +G) \cong \HM^{\loc}(\tfrac{1}{b}F+G)$. For this we consider the homotopy $G_s(x) : = \tfrac{1}{sb + (1-s)a)}F + G$. Define  
\[
k:  \dot{B}(r_0) \to \R, \qquad k(x) := \frac{\| X_G(x)\| }{ \| X_F(x) \|} = \frac{1}{h(x)}.
\]
By assumption there exists $\varepsilon > 0$ such that $k(x) > \varepsilon$ for all $x \in \dot{B}(r_0)$. Set $C(F,G) = 2/\varepsilon$. Then for $C<a<b$ one has
\[
\| X_{G_s}(x) \| \ge (k(x) - \tfrac{1}{a}) \|X_F(x)\| \ge \tfrac{ \varepsilon}{2}\|X_F (x) \|,
\]
and so again  $0$ is a uniformly isolated zero of $X_{G_s}$, whence $\HF^{\loc}(G_s, 0)$ is independent of $s$. This completes the proof.
\end{proof}
\end{rem}

\subsection{The proof of Theorem \ref{thm:GG2}}
\label{sec:completing_the_proof}
With these preliminaries out of the way, let us get started on the proof of Theorem \ref{thm:GG2}. Since the statement is a purely local statement, we may assume without loss of generality that $(M, \omega) =( \R^{2n}, \omega_0)$ and that $x = 0$. Let us abbreviate
\[
A : = D \varphi(0), \qquad B := D \psi(0), 
\]
so that by assumption $A, B \in \mathrm{Sp}(2n)$ commute. Following \cite{GinzburgGurel2010}, to prove Theorem \ref{thm:GG2} we will first prove the result in two special cases. 

\textbf{Case 1: The non-degenerate case}\\
Fix an admissible $k$, and suppose that $AB^k$ has no eigenvalues equal to 1. Then $0$ is a non-degenerate fixed point of $\varphi \psi^k$, and in particular is a Morse(-Bott) component of $\mathrm{Fix}(\varphi \psi^k)$. Thus 
\[
\HF^{\loc}( \varphi \psi^k , 0) \cong \mathrm{H}(\{ \mathrm{pt} \} ; \Z_2),
\]
by Example \ref{ex:morse_bott_component}. \\

\textbf{Case 2: The maximally degenerate case}\\
Fix an admissible $k$, and suppose that all the the eigenvalues of $AB^k$ are equal to $1$. Thus since $k$ is admissible, all the eigenvalues of $A$ and all the eigenvalues of $B$ are also all equal to 1. Thus in this case every $k$ is necessarily admissible, and we must show that for all $k$ sufficiently large the local Floer homology groups $\HF^{\loc}(\varphi \psi^k ,0)$ are isomorphic (up to a degree shift). This follows directly from Theorem \ref{thm:large_iterates_always_work}. 

\textbf{Case 3: The general case}\\
Fix an admissible $k$. Write $\R^{2n} = V \oplus W $, where $V$ and $W$ are linear $AB^k$-invariant subspaces such that $AB^k|_V$ has all its eigenvalues equal to 1 and $AB^k|_W$ has no eigenvalues equal to 1. By the argument of \cite[Section 4.5]{GinzburgGurel2010}, we can homotope $\varphi \psi^k$ to a Hamiltonian diffeomorphism $\theta$ in such a way so that $0$ remains a uniformly isolated fixed point, and such that $\theta$ is split, i.e: $\theta(z) = (\theta_V(z) , \theta_W(z)) \in V \oplus W$. Then by the K\"unneth formula (Lemma \ref{lem:kunneth}), one has 
\[
\HF^{\loc}(\varphi \psi^k ,0) \cong \HF^{\loc}(\theta_V, 0) \otimes \HF^{\loc}(\theta_W,0).
\]
Since $0$ is a non-degenerate critical point of $\theta_W$, the argument above tells us that $\HF^{\loc}( \theta_W, 0) \cong \mathrm{H}(\{ \mathrm{pt} \} ; \Z_2)$. We would like to apply Case 2 to $\theta_V$, but in order to do this we must exhibit $\theta_V $ as a product $\theta_V \cong \theta_1 \circ \theta_2^k$, where the $\theta_j :  V \to V$
are Hamiltonian diffeomorphisms such that the linear maps $D \theta_1 (0)$ and $D \theta_2(0)$ commute and have all their eigenvalues equal to 1. To accomplish this, consider two more splittings: 
\[
  \R^{2n} = V_A \oplus W_A, \qquad \R^{2n} = V_B \oplus W_B,
\]
where $V_A$ and $W_A$ are are linear $A$-invariant subspaces such that $A|_{V_A}$ has all its eigenvalues equal to 1 and $A|_{W_A}$ has no eigenvalues equal to 1, and similarly for $B$. As above, by \cite[Section 4.5]{GinzburgGurel2010} we can homotope $\varphi$ and $\psi$ to maps $\theta_{ \varphi}$ and $\theta_{\psi}$ in such a way that $0$ remains a uniformly isolated fixed point, and such that $\theta_{\varphi} = (\theta_{ \varphi,1}, \theta_{ \varphi ,2})$ and $\theta_{\psi} = (\theta_{ \psi,1}, \theta_{ \psi ,2})$ are split with respect to these decompositions. 
Since $k$ is admissible, it readily follows that
\[
  V = V_A \cap V_B,
\]
and since $\varphi $ and $\psi$ commute, the maps $\theta_{\varphi ,1}$ and $\theta_{\psi,1}$ commute and preserve $V$. Thus $\theta_V \cong \theta_{\varphi, 1} \circ \theta_{ \psi ,1}^k$, and we can apply Case 2 to deduce that (for $k$ sufficiently large):
\[
  \HF^{\loc}( \varphi \psi^k) \cong \HF^{\loc}( \theta_V, 0) \cong \HF^{\loc}(\theta_{\varphi ,1}  \circ \theta_{\psi,1}^k ,0 ) \cong \HF^{\loc}( \theta_{\psi, 1},0) \cong \HF^{\loc}(\psi,0).
\] 
\section{$L^\infty$-estimates}
\label{sec:proof_of_thm_estimate}

In this section we will prove Theorem \ref{thm:estimate}. In a slightly different setting the proof is carried out in \cite[Section 4]{AbbondandoloMerry2014a}. There are some minor modifications required here, and hence we give a fairly complete proof below, omitting only those stages which are identical to their counterparts in \cite[Section 4]{AbbondandoloMerry2014a}. 
\begin{lem}
\label{lem:L_infty_1}
We show that for any  $\widetilde{z}= (z, \tau, \sigma) \in \Lambda(\widetilde{M})$ and any $ \varepsilon >0 $, one has the implication:
\begin{equation}
\label{eq:small_gradient}
\| \nabla \A_{ \widetilde{L}^c} (\widetilde{z}) \|_{ L^2(S^1)} < \frac{\varepsilon}{2} \qquad \then \qquad \max_{ t \in [0, 1/2]} | H(z(t)) | \le \varepsilon.
\end{equation}
\end{lem}
\begin{proof}
We first prove the weaker statement that 
\begin{equation}
\label{eq:small_gradient_2}
\min_{ t \in [0,1/2]} | H(z(t)) | \le \| \nabla \A_{ \widetilde{L}^c} (\widetilde{z}) \|_{ L^2(S^1)}.
\end{equation}
This is clear if $H(z(t))=0$ for some $t\in[0,1/2]$. Thus without
loss of generality assume that $H(z(t))>0$ for all $t\in[0,1/2]$.
Then we have 
\begin{align*}
\min_{t\in[0,1/2]} | H(z(t)) | & =\min_{t\in[0,1/2]}H(z(t)) \\ 
& =\min_{t\in[0,1/2]}H(z(t))\int_{S^1} \kappa(t)dt\\
 & \leq\int_{S^1}\kappa(t)H(z(t))\,dt \\
 & \leq \| \sigma'- \kappa H(z) \| _{L^2(S^1)} \\
 & \leq \| \nabla \A_{ \widetilde{L}^c} (\widetilde{z}) \|_{ L^2(S^1)},
\end{align*}
where we used \eqref{eq:cutoff_function_kappa} in the second line and \eqref{eq:the_gradient} in the last line.  We now use \eqref{eq:small_gradient_2} to prove \eqref{eq:small_gradient}. Indeed, the hypotheses of \eqref{eq:small_gradient} together with \eqref{eq:small_gradient_2} tell us that
\[
\min_{t \in [0,1/2]} | H(z(t)) | <\frac{\varepsilon}{2}.
\]
If it is not the case that $| H(z(t))|<\varepsilon$ for
all $t\in[0,1/2]$ then there exists an interval $[t_{0},t_{1}]\subset[0,1/2]$
such that 
\[
\frac{\varepsilon}{2}\le | H(z(t)) | \leq\varepsilon,  \qquad \text{for all }t\in[t_0,t_1],
\]
with 
\[
\bigl|  H(z(t_0))- H(z(t_1)  \bigr|=\frac{\varepsilon}{2}.
\]
Then we estimate
\begin{align*}
\frac{\varepsilon}{2} 
& =|H(z(t_1))-H(z(t_0)) | \\
& =\left|\int_{t_0}^{t_1}\frac{d}{dt}H(z(t))\,dt\right|\\
 & \leq \int_{t_0}^{t_1} | dH(z)[z'] |dt=\int_{t_0}^{t_1} | d \lambda (X_{H}(z),z') | 
\, dt\\
 & \overset{(*)}{=} \int_{t_0}^{t_1}\left| d \lambda \left(X_H(z),z'- \tau X_{\kappa H}(z)-X_{L^c_t}(z) \right) \right|\, dt\\
 & =\int_{t_0}^{t_1}| X_{H} (z) |_{ J_t(z,\tau)}  | z'- \tau X_{\kappa H}(z)-X_{L^c_t}(z) |_{J_t(z,\tau)} \, dt \\
 & \le \| X_H \|_{L^{\infty}(M)} \int_{t_0}^{t_1}  | z'- \tau X_{\kappa H}(z)-X_{L^c_t}(z) |_{J_t(z,\tau)}  dt \\
 & \le \| X_H \|_{L^{\infty}(M)}  \| \nabla \A_{ \widetilde{L}^c} (\widetilde{z}) \|_{ L^2(S^1)},
\end{align*}
where $(*)$ used the fact that $d \lambda(X_H, X_H) = 0$ and that $X_{L^c_t} = 0 $ for $t \in [0,1/2]$.
Now \eqref{eq:small_gradient} follows, since from the definition \eqref{eq:hamiltonian_H}, one has $ \| X_H \|_{L^\infty(M} \le 1$. 
\end{proof}
\begin{lem}
\label{lem:L_infty_2}
Fix $c >C(\widehat{\varphi})$ and choose $r_0 > \max \{ 2 , c \}$ and assume that $I^*J$ is of contact type on $ \Sigma \times (r_0, + \infty)$.  Suppose that $\widetilde{z}=(z,\tau, \sigma)$ satisfies $z(S^1) \subset M_1 \cup_{ \Sigma} ( \Sigma \times (1,r_0])$ and $A,B >0$ are such that
\[
| \A_{\widetilde{L}^c}(\widetilde{z}) | \le A, \qquad \| \sigma \|_{L^2( S^1)} \le B.
\]
We prove that there exists a constant $C>0$ such that the implication
\begin{equation}
\label{eq:step2}
\| \nabla \A_{ \widetilde{L}^c} (\widetilde{z}) \|_{ L^2(S^1)} \le \frac{1}{8} \qquad \then \qquad \| \tau \|_{L^{\infty}(S^1)} \le C,
\end{equation}
holds. 
\end{lem}
\begin{proof}
First note that 
\begin{align*}
\| \tau' \|_{L^2(S^1)} & \le \|  \tau' - \sigma \|_{L^2(S^1)} + B \\
& \le \| \nabla \A_{ \widetilde{L}^c} \|_{L^2(S^1)} + B \\
& \le \frac{1}{8} + B.
\end{align*}
Next, we have
\begin{align*}
\left| \int_{S^1} z^* \lambda - \int_{S^1}\widetilde{L}^c(\widetilde{z}) \,dt \right| 
& \le | \A_{ \widetilde{L}^c}(\widetilde{z}) | + \left| \int_{S^1} \left\langle \tau', \sigma \right\rangle  \,dt \right| \\ 
& \le A + \| \sigma \|_{L^2(S^1)} \| \tau' \|_{L^2(S^1)} \\
& \le A + \frac{B}{8} + B^2.
\end{align*}
Set 
\[
N : = r_0 \| \lambda \|_{L^\infty(M_1)} < + \infty
\]
(note that in the line above we have written $M_1$ not $M$!).
Denote by 
\[
K:= \int_{S^1} \max_{z \in M} \left[ \lambda(X_{L_t^c}(z)) - L^c_t(z) \right]  \,dt .
\]
\begin{rem}
It follows from \eqref{eq:L_is_lifted_l} that  $(r \alpha)(X_{L_t}) = L_t$. Thus the constant $K$ is only non-zero due to the fact that we have introduced the cutoff function $\beta_c$.
\end{rem}
We now estimate
\begin{align*}
\left|\int_{S^1}\lambda(\tau \kappa X_H (z)-\tau\kappa H(z))dt\right| \le & {}  \left|\int_{S^1}z^*\lambda-\int_{S^1} \widetilde{L}^c_t(\widetilde{z }) \right|  \,dt + \left|\int_{S^1} \left[ \lambda(X_{L^c_t}(z)) - L^c_t(z) \right]\,dt \right|  \\ 
& + \left|\int_{S^1}\lambda(z'- \tau \kappa X_H(z)-X_{L^c_t}(z)dt\right|\\
 & {} \le A + \frac{B}{8} + B^2+ K \\ 
 & +N \int_{S^1} |z' - \tau  \kappa X_H(z) - X_{L^t_c}(z) |_{J_t} \, dt \\
 & \le A + \frac{B}{8} + B^2+ K + N \| \nabla \A_{\widetilde{L}^c}(\widetilde{z}) \|_{L^2(S^1)} \\
 & \le A + \frac{B}{8} + B^2+ K +  \frac{N}{8}.
\end{align*}
We claim that
\begin{equation}
\label{eq:min_tau}
\min_{ t \in [0,1/2]} | \tau(t) | \le \frac{4}{3}\left(A + \frac{B}{8} + B^2 + K +  \frac{N}{8}\right).
\end{equation}
Indeed, there is nothing to prove if $\tau(t)$ changes sign, so without loss of generality we may assume that $\tau(t) > 0$. Lemma \ref{lem:L_infty_1} tells us that the assumption that  $\| \nabla \A_{ \widetilde{L}^c} (\widetilde{z}) \|_{ L^2(S^1)} \le \tfrac{1}{8}$ implies $ - \tfrac{1}{4} \le H(z(t)) \le \tfrac{1}{4}$ for all $t \in [0,1/2]$, and hence 
\[
\lambda(X_H(z)) - H(z) \ge 1 - \tfrac{1}{4} \ge \tfrac{3}{4}.
\]
Then we have
\begin{align*}
A + \frac{B}{8} + B^2 + K +  \frac{N}{8} & \ge \left|\int_{S^1}\lambda(\tau \kappa  X_H(z)-\tau\kappa H(z))dt\right| \\
& \ge \frac{3}{4}\int_{0}^{1/2}\tau(t)\kappa(t)\,dt  \\
& \ge \frac{3}{4} \min_{t \in [0,1/2]} \tau(t),
\end{align*}
which proves \eqref{eq:min_tau}.
The proof is finally completed with
\begin{align*}
\| \tau \|_{L^\infty(S^1)} & \le \min_{t \in [0,1/2]} | \tau(t) | + \| \tau' \|_{L^1(S^1)} \\
& \le C:=  \frac{4}{3}\left( A + \frac{B}{8} + B^2 + K +  \frac{N}{8}\right) + \frac{1}{8}+ B.
\end{align*}
\end{proof}
\begin{lem}
\label{lem:L_infty_3}
If $ \widetilde{u} = (u , \eta ,\zeta)$ is any flow line satisfiying the assumptions of Theorem \ref{thm:estimate}, then for every $s\in \R$ one has
\begin{equation}
\label{eq:zeta_bound}
\| \zeta(s,\cdot)\|_{L^2(S^1)} \leq 3 \sqrt{A} + 1.
\end{equation} 
\end{lem}
\begin{proof}
First consider the function
\[
\zeta_\circ(s) := \int_{S^1}\zeta(s,t)\, dt.
\]
We claim $\zeta_{\circ} \equiv 0$. Indeed, by  \eqref{eq:floer}, $\zeta_{\circ}$ satisfies the ODE 
\begin{equation}
\label{eq:zeta_ODE}
\zeta_{\circ}' + \zeta_{\circ} = 0.
\end{equation} 
Moreover one has
\[
\begin{split}
\|\zeta_{\circ}'\|_{L^2(\R)}^2 &= \int_{- \infty}^{ + \infty} \left( \frac{d}{ds} \int_{S^1} \zeta(s,t)\, dt \right)^2\, ds  \\
& = \int_{- \infty}^{+ \infty} \left( \int_{S^1} \partial_s \zeta(s,t)\, dt \right)^2 \, ds \\
& \le  \int_{- \infty}^{+\infty} \int_{S^1} |\partial_s \zeta(s,t)|^2 \, dt \, ds  \\ 
&\le \int_{- \infty}^{+\infty}  \| \partial_s \widetilde{u} \|_{\mathrm{J}}^2 \, ds = \mathbbm{E}(\widetilde{u}) < 2A.
\end{split}
\]
The only solution to $\zeta_{\circ}$ to \eqref{eq:zeta_ODE} with $ \|  \zeta_{\circ}' \|_{L^2(\R)} < +\infty$ is the zero solution. Now consider the subset $\mathcal{S} \subset \R$ defined by
\[
\mathcal{S} := \left\{ s\in \R \mid \|\nabla \mathbbm{A}_{\widetilde{L}^c} (\widetilde{u}(s)) \|_{L^2(S^1)} \leq\sqrt{A} \right\}.
\]
By Chebychev's inequality, one has 
\[
|\R \setminus \mathcal{S}| \le \frac{1}{A} \int_{\R} \|\nabla \A_{\widetilde{L}^c} (\widetilde{u}(s)) \|_{\widetilde{J}}^2 \, ds =  \frac{1}{A} \E(\widetilde{u}) \leq 2,
\]
and hence given $s\in \R$, we can find $s_0\in \mathcal{S}$ such that  $|s-s_0|\leq 1$.
Using again \eqref{eq:floer}, we find that
\[
\begin{split}
\|\partial_t \zeta(s_0,\cdot) \|_{L^2(S^1)} &\le \|\partial_t \zeta(s_0,\cdot) - \kappa (\cdot) H(u(s_0,\cdot)) \|_{L^2(S^1)} + \|\kappa(\cdot) H(u(s_0,\cdot)) \|_{L^2(S^1)} \\ 
&\le \| \nabla \A_{\widetilde{L}^c} (\widetilde{u}(s_0)) \|_{L^2(S^1)} + 1 \\
& \le \sqrt{A} + 1.
\end{split}
\]
Since $\zeta(s_0,\cdot)$ has zero mean, the Poincar\'e inequality implies that
\[
\|\zeta(s_0,\cdot)\|_{L^2(S^1)} \le \|\partial_t \zeta(s_0,\cdot) \|_{L^2(S^1)} \leq  \sqrt{A} + 1.
\]
Moreover, since 
\[
\|\partial_s \zeta\|_{L^2(\R\times S^1)} \le \|\partial_s \widetilde{u}\|_{L^2(\R\times S^1)} = \sqrt{\mathbbm{E}(\widetilde{u})} \leq \sqrt{2A},
\]
we see that
\[
\begin{split}
\|\zeta(s,\cdot)\|_{L^2(S^1)} &= \|\zeta(s_0,\cdot)\|_{L^2(S^1)} + \int_{s_0}^s \frac{d}{d\sigma} \|\zeta(\sigma,\cdot)\|_{L^2(S^1)} \, d\sigma \\ 
&\le \sqrt{A} + 1+ \left| \int_{s_0}^s  \Bigl\| \frac{d}{d\sigma} \zeta(\sigma,\cdot) \Bigr\|_{L^2(S^1)} \, d\sigma \right| \\
&= \sqrt{A} + 1 + \left| \int_{s_0}^s  \Bigl( \int_{S^1} | \partial_s \zeta(\sigma,t)|^2\, dt \Bigr)^{1/2}\, d\sigma \right| \\ 
&\le \sqrt{A} + 1+ |s-s_0|^{1/2} \left| \int_{s_0}^s \int_{S^1} |\partial_s \zeta|^2 \, dt\, d\sigma \right|^{1/2} \\
&\le \sqrt{A} + 1 + \|\partial_s \zeta\|_{L^2(\R\times S^1)} \\
&\le \sqrt{A} + 1 + \|\partial_s \widetilde{u}\|_{L^2(\R\times S^1)} \\ 
&\le \sqrt{A} +1 + \sqrt{\E(\widetilde{u})} \\
& \le 3 \sqrt{A} +1,
\end{split}
\]
which finishes the proof of \eqref{eq:zeta_bound}. 
\end{proof}

Using Lemma \ref{lem:L_infty_3}, one can prove the following statement:
\begin{lem}
\label{lem:L_infty_4}
There exists a constant $B > 0$ such that if $ \widetilde{u} = (u , \eta ,\zeta)$ is any flow line satisfiying the assumptions of Theorem \ref{thm:estimate}, then for every $s\in \R$ one has
\begin{equation}
\label{eq:eta_bound}
\| \eta(s,\cdot)\|_{L^2(S^1)} \leq B.
\end{equation} 
\end{lem}
The proof is omitted, as it is word-for-word identical to the proof of \cite[Lemma 4.4]{AbbondandoloMerry2014a}. 
\begin{prop}
\label{prop:L_infty_5}
There exists a constant $C > 0$ such that if $ \widetilde{u} = (u , \eta ,\zeta)$ is any flow line satisfiying the assumptions of Theorem \ref{thm:estimate}, then one has
\begin{equation}
\label{eq:L_infty_bound}
\| \eta\|_{L^\infty(\R \times S^1)} \le C, \qquad \| \zeta\|_{L^\infty(\R \times S^1)} \le C.
\end{equation} 
\end{prop}
\begin{proof}
Consider the smooth function
\[
f : \R \times S^1 \to \C, \qquad f := \eta + i \zeta.
\]
We will show that $\|f\|_{L^{\infty}(\R\times S^1)}$ is uniformly bounded. Lemma \ref{lem:L_infty_3} and Lemma \ref{lem:L_infty_4} imply that there exists a constant $D >0$ such that
\begin{equation}
\label{boundf}
\|f(s,\cdot)\|_{L^2(S^1)} \leq D, \qquad \text{for all } s\in \R.
\end{equation}
From \eqref{eq:floer}, we see that 
\begin{equation}
\label{eq:cr}
\overline{\partial} f = \kappa H(u) + i \zeta,
\end{equation}
where
\[
\overline{\partial} = \partial_s + i \partial_t
\]
is the Cauchy-Riemann operator. Suppose $I_1 \subset I_2$ are open intervals such that $I_1$ is bounded and its closure is contained in $I_2$. The Calderon-Zygumund theorem implies that for any $ 1 <p < + \infty$, there exists a constant $m = m(p, I_1, I_2)>0$ such that
\begin{equation}
\label{eq:cz}
\|f\|_{W^{1,p}(I_1 \times S^1)} \le m \left( \|\overline{\partial} f\|_{L^p(I_2 \times S^1)} + \|f\|_{L^2(I_2 \times S^1)} \right).
\end{equation}
Fix now some $2 < p < + \infty$ and $k \in \Z$. Then by the Sobolev embedding theorem there exists a constant $b= b(p) >0$ such that 
\begin{equation}
\begin{split}
\label{eq:est_1}
\|f \|_{L^{\infty}((k,k+1) \times S^1)} & \le b \|f \|_{W^{1,p}((k,k+1) \times S^1)} \\
& \le bm \left( \|\overline{\partial} f\|_{L^p((k-1,k+2) \times S^1)} + \|f\|_{L^2((k-1,k+2) \times S^1)} \right) \\
& \le bm  \|\overline{\partial} f\|_{L^p((k-1,k+2) \times S^1)} + bmD.
\end{split}
\end{equation}
Using \eqref{eq:cr} and the Sobolev embedding theorem again, we see there exists a constant $e=e(p)>0 $ such that:
\begin{equation}
\label{eq:est_2}
\begin{split}
\|\overline{\partial} f\|_{L^p((k-1,k+2) \times S^1)} &= \| \kappa H(u) + i \zeta \|_{L^p((k-1,k+2) \times S^1)} \\
& \le 3^{1/p}+ \|\zeta\|_{L^p(I_1 \times S^1)} \\
 &\le 3^{1/p}  + e \|f\|_{W^{1,2}((k-1,k+2) \times S^1)}.
\end{split}
\end{equation}
Now we are in business, since by applying \eqref{eq:cz} again, this time with $p = 2$ and corresponding constant $m'$ that 
\begin{equation}
\label{eq:est_3}
\begin{split} 
\|f\|_{W^{1,2}((k-1,k+2) \times S^1)} &\leq m' \left( \|\overline{\partial} f\|_{L^2((k-2,k+3) \times S^1)} + \|f\|_{L^2((k-1,k+2) \times S^1)} \right) \\ 
&\le m' \left( \sqrt{5}  + \|\zeta\|_{L^2((k-2, k+3)\times S^1)} + \|f\|_{L^2((k-2,k+3) \times S^1)} \right) \\ 
&\le \sqrt{5} m' ( 1 + B +D ).
\end{split}
\end{equation}
Combining \eqref{eq:est_1}, \eqref{eq:est_2} and \eqref{eq:est_3}, we see that $f$ is uniformly bounded in $L^\infty((k,k+1) \times S^1)$. Since $k$ was arbitary and all the constants involved depend only on the length of the intervals involved, we obtain the desired $L^\infty$-bound. 
\end{proof}

To complete the proof of Theorem \ref{thm:estimate} one needs only show why the $u$-component cannot escape $M_1 \cup_{ \Sigma } ( \Sigma \times (1,r_0])$. This is a standard maximum principle argument; see for instance \cite[Section 4.2]{AbbondandoloMerry2014a}.

\section{Exact magnetic flows} % (fold)
\label{sec:exact_magnetic_flows}

In this section we prove Theorem \ref{thm:magnetic} from the Introduction. Let us recall the setup. Suppose $Q$ is a closed manifold and $\Omega$ is an \textbf{closed} 2-form on $Q$. One should think of $\Omega$ as representing a \textbf{magnetic field}. We use $\Omega$ to build a  \textbf{twisted symplectic form} $\omega = d \lambda + \pi^* \Omega$, on $T^*Q$, where as before $\lambda$ is the canonical Liouville 1-form. Suppose  $H:T^*Q \to \R$ is a \textbf{Tonelli} Hamiltonian: this means that $H$ is smooth function on $T^*Q$ which is $C^2$-strictly convex and superlinear on the fibres of $T^*Q$. We are interested in studying the flow of $\phi_H^t : T^*Q \to T^*Q$ of the symplectic gradient $X_H$ of $H$, taken with respect to the twisted symplectic form $\omega$. For instance, if $H(q,p) = \frac{1}{2} | p |^2 +U(q)$ is a mechanical Hamiltonian of the form kinetic plus potential energy, then $\phi_H^t$ can be thought of as modelling the motion of a charged particle in a magnetic field. We refer the reader to \cite{Ginzburg1996} for an in-depth treatment of magnetic flows in symplectic geometry. 

Given $e > 0$, let $\Sigma_e : = H^{-1}(e) \subset T^*Q$. Since $H$ is autonomous, the flow $\phi_H^t : T^*Q \to T^*Q$ of the symplectic gradient $X_H$ preserves the energy level $\Sigma_e$. A \textbf{magnetic geodesic} $\gamma : \R \to Q$ of energy $e$ is the projection to $Q$ of an orbit of $\phi_H^t|_{\Sigma_e}$.

Let us denote by $\mathcal{G}(H,\Omega)$ the group of \textbf{symmetries} of the system:
\[
  \mathcal{G}(H,\Omega) := \left\{  f \in \mathrm{Diff}(Q) \mid f^*\Omega = \Omega, \text{ and } H(f(q),p) = H(q, p \circ Df(q)), \ \forall \,(q,p) \in T^*Q \right\}.
  \]
For instance, if $H(q,p) = \frac{1}{2}| p |^2 + U(q)$ is a mechanical Hamiltonian, then elements of $\mathcal{G}(H,\Omega)$ are simply the isometries of $(Q,g)$ that preserve the 2-form $\Omega$ and the potential $U$. Let $\mathcal{G}_0(H,\Omega)$ denote the connected component of $\mathcal{G}(H,\Omega)$ containing $\mathrm{Id}$.

Assume now that $\Omega$ is \textbf{exact}. We define the \textbf{strict Ma\~n\'e critical value} $c_0= c_0(H, \Omega)$ by
\begin{equation}
\label{eq:mcv}
  c_0 := \inf_{ \theta} \sup_{q \in Q} H(q, -\theta_q),
\end{equation}
where the infimum\footnote{The fact that one takes $-\theta$ in the definition of $c_0$ is due to our sign conventions.} is over the set of all primitives $\theta$ of $\Omega$.  

\begin{lem}
\label{lem:rct}
If $e > c_0$ then $\Sigma_e \subset T^*Q$ is a hypersurface of restricted contact type in the symplectic manifold $(T^*Q, \omega)$.
\end{lem} 
\begin{proof}
Suppose $\theta$ is a primitive of $\Omega$ satisfying $\sup_{q \in Q}H(q, -\theta_q) \le e - \varepsilon$ for some $\varepsilon > 0$. Then we claim that  
\[
 (\lambda + \pi^* \theta )(X_H)|_{\Sigma_e} > 0.
 \]
 For this, fix $(q,p) \in \Sigma_e $ and let $h(s):= H(q, sp - (1-s)\theta_q)$. Then one computes that $(\lambda + \pi^*\theta)_q(X_H(q,p)) = h'(1)$. Since $H$ is Tonelli the function $h$ is convex. Since $h(0) \le e - \varepsilon$ and $h(1) = e$, we must have $h'(1) \ge \varepsilon$ as required. 
 \end{proof}
  
The main step in the proof of Theorem \ref{thm:magnetic} is the following result. 
\begin{prop}
\label{prop:invariance}
Suppose $e > c_0$. Then there exists a primitive $\theta$ of $\Omega$ such that
\[
  f^*\theta = \theta, \qquad \forall \, f \in \mathcal{G}_0(H, \Omega),
\] 
and such that 
\[
  \sup_{q \in Q}H(q, - \theta_q) \le e.
\] 
\end{prop}
\begin{proof}
The group $\mathcal{G}_0(H,\Omega)$ is a connceted  compact Lie group, according to \cite[Proposition 5]{Maderna2002}, and thus carries a left-invariant Haar measure $m$. Thus given any primitive $\theta$ of $\Omega$, we can average it to form a new primitive $\theta'$
\[
  \theta' : = \int_{\mathcal{G}_0(H,\Omega)} f^* \theta \,dm(f).
\]
By construction one has $\sup_{q \in Q}H(q, - \theta_q) = \sup_{q \in Q} H(q, -\theta_q')$, and the result follows.
\end{proof}

We can now prove Theorem \ref{thm:magnetic}, which we restate here for the convenience of the reader. 
\begin{thm}
\label{thm:magneticproof}
 Suppose $Q$ is a closed connected  manifold with the property that the Betti numbers of the free loop space $ \Lambda (Q)$ are asymptotically  unbounded. Suppose $e > c_0(H, \Omega)$. Then given any  symmetry $f \in \mathcal{G}_0(H, \Omega)$, there exist infinitely many invariant magnetic geodesics with energy $e$.
 \end{thm} 

\begin{proof}
By Lemma \ref{lem:rct} and Proposition \ref{prop:invariance}, we can choose an $f$-invariant primitive $\theta$ of $\Omega$ for which $\Sigma_e$ is a hypersurface of restricted contact type with respect to the primitive $\lambda + \pi^* \theta$ of $\omega$. The lifted symplectomorphism $\phi_f : T^*Q \to T^*Q$ defined by 
 \[
    \phi_f :T^*Q \to T^*Q , \qquad  \phi_f(q, p) = (f(q), p \circ Df(q)^{-1}),
  \] 
  restricts to define a strict contactomorphism of $(\Sigma_e, ( \lambda + \pi^* \theta)|_{\Sigma_e})$ which is contact-isotopic to the identity. The result now follows from Theorem \ref{thm:actual_theorem} and the computation of the Rabinowitz Floer homology of the pair $(\Sigma_e,T^*Q)$ in \cite{AbbondandoloSchwarz2009,Merry2011a,BaeFrauenfelder2010}. 
\end{proof}

\bibliographystyle{amsalpha}
\bibliography{/Users/iLyX/Dropbox/LaTeX/willmacbibtex}

\providecommand{\bysame}{\leavevmode\hbox to3em{\hrulefill}\thinspace}
\providecommand{\MR}{\relax\ifhmode\unskip\space\fi MR }
\providecommand{\MRhref}[2]{%
  \href{http://www.ams.org/mathscinet-getitem?mr=#1}{#2}
}
\providecommand{\href}[2]{#2}
\begin{thebibliography}{CFHW96}

\bibitem[AF10a]{AlbersFrauenfelder2010c}
P.~Albers and U.~Frauenfelder, \emph{Leaf-wise intersections and {R}abinowitz
  {F}loer homology}, J. Topol. Anal. \textbf{2} (2010), no.~1, 77--98.

\bibitem[AF10b]{AlbersFrauenfelder2010}
\bysame, \emph{Spectral {I}nvariants in {R}abinowitz {F}loer homology and
  {G}lobal {H}amiltonian perturbations}, J. Modern Dynamics \textbf{4} (2010),
  329--357.

\bibitem[AF12a]{AlbersFrauenfelder2012b}
\bysame, \emph{Infinitely many leaf-wise intersection points on cotangent
  bundles}, Global {D}ifferential {G}eometry, Proc. in {M}athematics, vol.~17,
  Springer-Verlag, 2012, pp.~437--461.

\bibitem[AF12b]{AlbersFrauenfelder2012a}
\bysame, \emph{Rabinowitz {F}loer homology: {A} {S}urvey}, Global
  {D}ifferential {G}eometry, vol.~17, Springer {P}roceedings in {M}athematics,
  no.~3, Springer-Verlag, 2012, pp.~437--461.

\bibitem[AM13]{AlbersMerry2013a}
P.~Albers and W.~J. Merry, \emph{Translated points and {R}abinowitz {F}loer
  homology}, J. Fixed Point Theory Appl. \textbf{13} (2013), no.~1, 201--214.

\bibitem[AM14]{AbbondandoloMerry2014a}
A.~Abbondandolo and W.~J. Merry, \emph{Floer homology of the time-energy
  extended phase space}, In preparation (2014).

\bibitem[AS06]{AbbondandoloSchwarz2006}
A.~Abbondandolo and M.~Schwarz, \emph{On the {F}loer homology of cotangent
  bundles}, Comm. Pure Appl. Math. \textbf{59} (2006), 254--316.

\bibitem[AS09]{AbbondandoloSchwarz2009}
\bysame, \emph{Estimates and computations in {R}abinowitz-{F}loer homology}, J.
  Topol. Anal. \textbf{1} (2009), no.~4, 307--405.

\bibitem[BF11]{BaeFrauenfelder2010}
Y.~Bae and U.~Frauenfelder, \emph{Continuation homomorphism in {R}abinowitz
  {F}loer homology for symplectic deformations}, Math. Proc. Camb. Phil. Soc.
  \textbf{151} (2011), 471--502.

\bibitem[BPS03]{BiranPolterovichSalamon2003}
P.~Biran, L.~Polterovich, and D.~Salamon, \emph{Propagation in {H}amiltonian
  dynamics and relative symplectic homology}, Duke Math. J. \textbf{119}
  (2003), 65--118.

\bibitem[CF09]{CieliebakFrauenfelder2009}
K.~Cieliebak and U.~Frauenfelder, \emph{A {F}loer homology for exact contact
  embeddings}, Pacific J. Math. \textbf{239} (2009), no.~2, 216--251.

\bibitem[CFHW96]{CieliebakFloerHoferWysocki1996}
K.~Cieliebak, A.~Floer, H.~Hofer, and K.~Wysocki, \emph{Applications of
  symplectic homology {II}: {S}tability of the action spectrum}, Math. Z.
  \textbf{223} (1996), 27--45.

\bibitem[CFO10]{CieliebakFrauenfelderOancea2010}
K.~Cieliebak, U.~Frauenfelder, and A.~Oancea, \emph{Rabinowitz {F}loer homology
  and symplectic homology}, Ann. Inst. Fourier \textbf{43} (2010), no.~6,
  957--1015.

\bibitem[CFP10]{CieliebakFrauenfelderPaternain2010}
K.~Cieliebak, U.~Frauenfelder, and G.~P. Paternain, \emph{Symplectic topology
  of {M}a\~n\'e's critical values}, Geometry and Topology \textbf{14} (2010),
  1765--1870.

\bibitem[GG10]{GinzburgGurel2010}
V.~Ginzburg and B.~G\"urel, \emph{Local {F}loer homology and the action gap},
  J. Symplectic Geometry \textbf{8} (2010), no.~3, 323--327.

\bibitem[Gin96]{Ginzburg1996}
V.~Ginzburg, \emph{On closed trajectories of a charge in a magnetic field. {A}n
  application of symplectic geometry}, Contact and symplectic geometry
  ({C}ambridge, 1994) (C.~B. Thomas, ed.), Publications of the {N}ewton
  {I}nstitute, vol.~8, Cambridge {U}niversity {P}ress, 1996, pp.~131--148.

\bibitem[Gin10]{Ginzburg2010}
\bysame, \emph{The {C}onley {C}onjecture}, Ann. Math. \textbf{172} (2010),
  no.~2, 1127--1180.

\bibitem[GM69]{GromollMeyer1969}
D.~Gromoll and W.~Meyer, \emph{Periodic geodesics on compact riemannian
  manifolds}, J. Diff. Geom. \textbf{3} (1969), 493--510.

\bibitem[Gro73a]{Grove1973a}
K.~Grove, \emph{Condition $({C})$ for the energy integral on certain path
  spaces and applications to the theory of geodesics}, J. Diff. Geom.
  \textbf{8} (1973), 207--223.

\bibitem[Gro73b]{Grove1973}
\bysame, \emph{Isometry-invariant geodesics}, Topology \textbf{13} (1973),
  281--292.

\bibitem[Gro78]{Gromov1978}
M.~Gromov, \emph{Homotopical effects of dilatations}, J. Diff. Geom.
  \textbf{13} (1978), 303--310.

\bibitem[GT76]{GroveTanaka1976}
K.~Grove and M.~Tanaka, \emph{On the number of invariant closed geodesics},
  Bull. Amer. Math. Soc. \textbf{82} (1976), 497--498.

\bibitem[GT78]{GroveTanaka1978}
\bysame, \emph{On the number of invariant closed geodesics}, Acta. Math.
  \textbf{140} (1978), 33--48.

\bibitem[HM12]{HryniewiczMacarini2012}
U.~Hryniewicz and L.~Macarini, \emph{Local contact homology and applications},
  arXiv:1202.3122 (2012).

\bibitem[LF51]{LyusternikFet1951}
L.~A. Lyusternik and A.~I. Fet, \emph{Variational problems on closed
  manifolds}, Doklady Akad. Nauk SSSR (N.S.) \textbf{81} (1951), 17--18.

\bibitem[Lu14]{Lu2014}
G.~Lu, \emph{Splitting lemmas for the {F}insler energy functional on the space
  of ${H}^1$-curves}, arXiv:1411.3209 (2014).

\bibitem[Mad02]{Maderna2002}
E.~Maderna, \emph{Invariance of global solutions of the {H}amilton-{J}acobi
  equation}, Bull. Soc. Math. France. \textbf{130} (2002), no.~4.

\bibitem[Maz14a]{Mazzucchelli2014a}
M.~Mazzucchelli, \emph{Isometry-invariant geodesics and the fundamental group},
  Preprint (2014).

\bibitem[Maz14b]{Mazzucchelli2014}
\bysame, \emph{On the multiplicity of isometry-invariant geodesics on product
  manifolds}, Alg. Geom. Topol. \textbf{14} (2014), 135--156.

\bibitem[McL11]{McLean2011}
M.~McLean, \emph{Computatability and the growth rate of symplectic homology},
  arXiv:1109.4466 (2011).

\bibitem[McL12]{McLean2012}
\bysame, \emph{Local {F}loer homology and infinitely many simple {R}eeb
  orbits}, Alg. Geom. Topol. \textbf{12} (2012), no.~4, 1901--1923.

\bibitem[Mer11]{Merry2011a}
W.~J. Merry, \emph{On the {R}abinowitz {F}loer homology of twisted cotangent
  bundles}, Calc. Var. Partial Differential Equations \textbf{42} (2011),
  no.~3-4, 355--404.

\bibitem[Mos78]{Moser1978}
J.~Moser, \emph{A fixed point theorem in symplectic geometry}, Acta. Math.
  \textbf{141} (1978), no.~1-2, 17--34.

\bibitem[Nae15]{Naef_thesis}
K.~Naef, \emph{In preparation}, Ph.D. thesis, 2015.

\bibitem[PP97]{PaternainPaternain1997a}
G.~P. Paternain and M.~Paternain, \emph{Critical values of autonomous
  {L}agrangians}, Comment. {M}ath. {H}elv. \textbf{72} (1997), 481--499.

\bibitem[San12]{Sandon2012}
S.~Sandon, \emph{On iterated translated points for contactomorphisms of
  $\mathbb{R}^{2n+1}$ and $\mathbb{R}^{2n} \times {S}^{1}$}, Internat. J. Math.
  \textbf{23} (2012), no.~2.

\bibitem[SW06]{SalamonWeber2006}
D.~Salamon and J.~Weber, \emph{Floer homology and the heat flow}, GAFA
  \textbf{16} (2006), 1050--1138.

\bibitem[SZ92]{SalamonZehnder1992}
D.~Salamon and E.~Zehnder, \emph{Morse {T}heory for {P}eriodic {S}olutions of
  {H}amiltonian {S}ystems and the {M}aslov {I}ndex}, Comm. Pure Appl. Math.
  \textbf{45} (1992), 1303--1360.

\bibitem[Tan82]{Tanaka1982}
M.~Tanaka, \emph{On the existence of infinitely many isometry-invariant
  geodesics}, J. Differential Geom. \textbf{17} (1982), 171--184.

\bibitem[Vit96]{Viterbo1996}
C.~Viterbo, \emph{Functors and computations in {F}loer homology with
  applications, {P}art {II}}, Preprint (1996).

\bibitem[Wei13]{Wiegel2013}
P.~Weigel, \emph{Orderable contact structures on {L}iouville-fillable contact
  manifolds}, arXiv:1304.3662 (2013).

\end{thebibliography}
\end{document}